\documentclass[12pt]{amsart}
\usepackage{amsmath}
\usepackage{amssymb,float}
\usepackage[colorlinks=true]{hyperref}
\usepackage[all]{xy}
\usepackage[toc,page,title,titletoc,header]{appendix}
\usepackage{xcolor}

\usepackage{tikz-cd}
\usepackage{graphicx}
\usepackage{epigraph}
\usepackage{mathtools}
\usepackage{bbm}

\begin{document}
	\pdfoutput=1
	\theoremstyle{plain}
	\newtheorem{thm}{Theorem}[section]
	\newtheorem*{thm1}{Theorem 1}
	
	\newtheorem*{thmM}{Main Theorem}
	\newtheorem*{thmA}{Theorem A}
	\newtheorem*{thm2}{Theorem 2}
	\newtheorem{lemma}[thm]{Lemma}
	\newtheorem{lem}[thm]{Lemma}
	\newtheorem{cor}[thm]{Corollary}
	\newtheorem{pro}[thm]{Proposition}
	\newtheorem{prop}[thm]{Proposition}
	\newtheorem{variant}[thm]{Variant}
	\newtheorem{fact}{Fact}
	\theoremstyle{definition}
	\newtheorem{notations}[thm]{Notations}
	\newtheorem{rem}[thm]{Remark}
	\newtheorem{rmk}[thm]{Remark}
	\newtheorem*{rmkU}{Remark}
	\newtheorem{rmks}[thm]{Remarks}
	\newtheorem{defi}[thm]{Definition}
	\newtheorem{exe}[thm]{Example}
	\newtheorem{claim}[thm]{Claim}
	\newtheorem{ass}[thm]{Assumption}
	\newtheorem{prodefi}[thm]{Proposition-Definition}
	\newtheorem{que}[thm]{Question}
	\newtheorem{con}[thm]{Conjecture}
	
	\newtheorem{exa}[thm]{Example}
	
	\newtheorem*{algstate}{Algebraic form of Theorem \ref{thmstattrainv}}
	
	\newtheorem*{dmlcon}{Dynamical Mordell-Lang Conjecture}
	\newtheorem*{condml}{Dynamical Mordell-Lang Conjecture}
	\newtheorem*{congb}{Geometric Bogomolov Conjecture}
	\newtheorem*{congdaocurve}{Dynamical Andr\'e-Oort Conjecture for curves}
	
	\newtheorem*{pdd}{P(d)}
	\newtheorem*{bfd}{BF(d)}

	\newtheorem*{probreal}{Realization problems}
	\numberwithin{equation}{section}
	\newcounter{elno}                
	\def\points{\list
		{\hss\llap{\upshape{(\roman{elno})}}}{\usecounter{elno}}}
	\let\endpoints=\endlist
	\newcommand{\Zhuchao}[1]{ \sf $\clubsuit\clubsuit\clubsuit$ Zhuchao : [#1]}
	\newcommand{\SH}{\rm SH}
	\newcommand{\Cov}{\rm Cov}
	\newcommand{\Tan}{\rm Tan}
	\newcommand{\res}{\rm res}
	\newcommand{\Om}{\Omega}
	\newcommand{\om}{\omega}
	\newcommand{\La}{\Lambda}
	\newcommand{\la}{\lambda}
	\newcommand{\mc}{\mathcal}
	\newcommand{\mb}{\mathbb}
	\newcommand{\surj}{\twoheadrightarrow}
	\newcommand{\inj}{\hookrightarrow}
	\newcommand{\zar}{{\rm zar}}
	\newcommand{\Exc}{{\rm Exc}}
	\newcommand{\Mod}{{\rm Mod}}
	\newcommand{\an}{{\rm an}}
	\newcommand{\red}{{\rm red}}
	\newcommand{\codim}{{\rm codim}}
	\newcommand{\Supp}{{\rm Supp\;}}
	\newcommand{\Leb}{{\rm Leb}}
	\newcommand{\rank}{{\rm rank}}
	\newcommand{\geom}{{\rm geom}}
	\newcommand{\Ker}{{\rm Ker \ }}
	\newcommand{\Pic}{{\rm Pic}}
	\newcommand{\Der}{{\rm Der}}
	\newcommand{\Div}{{\rm Div}}
	\newcommand{\Hom}{{\rm Hom}}
	\newcommand{\Corr}{{\rm Corr}}
	\newcommand{\im}{{\rm im}}
	\newcommand{\Spec}{{\rm Spec \,}}
	\newcommand{\Nef}{{\rm Nef \,}}
	\newcommand{\Frac}{{\rm Frac \,}}
	\newcommand{\Sing}{{\rm Sing}}
	\newcommand{\sing}{{\rm sing}}
	\newcommand{\reg}{{\rm reg}}
	\newcommand{\Char}{{\rm char\,}}
	\newcommand{\Tr}{{\rm Tr}}
	\newcommand{\ord}{{\rm ord}}
	\newcommand{\bif}{{\rm bif}}
	\newcommand{\AS}{{\rm AS}}
	\newcommand{\loc}{{\rm loc}}
	\newcommand{\FS}{{\rm FS}}
	\newcommand{\CE}{{\rm CE}}
	\newcommand{\PCE}{{\rm PCE}}
	\newcommand{\WR}{{\rm WR}}
	\newcommand{\PR}{{\rm PR}}
	\newcommand{\TCE}{{\rm TCE}}
	\newcommand{\diam}{{\rm diam\,}}
	\newcommand{\id}{{\rm id}}
	\newcommand{\NE}{{\rm NE}}
	\newcommand{\Gal}{{\rm Gal}}
	\newcommand{\Min}{{\rm Min \ }}
	\newcommand{\Hol}{{\rm Hol \ }}
	\newcommand{\Rat}{{\rm Rat}}
	\newcommand{\Hdim}{{\rm Hdim}}
	\newcommand{\dist}{{\rm dist}}
	\newcommand{\FL}{{\rm FL}}
	\newcommand{\fm}{{\rm fm}}
	\newcommand{\vol}{{\rm vol}}
	
	\newcommand{\Max}{{\rm Max \ }}
	\newcommand{\Alb}{{\rm Alb}\,}
	\newcommand{\Aff}{{\rm Aff}\,}
	\newcommand{\GL}{{\rm GL}\,}        
	\newcommand{\PGL}{{\rm PGL}\,}
	\newcommand{\Bir}{{\rm Bir}}
	\newcommand{\Bif}{{\rm Bif}}
	\newcommand{\Aut}{{\rm Aut}}
	\newcommand{\topo}{{\rm top}}
	\newcommand{\End}{{\rm End}}
	\newcommand{\Per}{{\rm Per}\,}
	\newcommand{\Preper}{{\rm Preper}\,}
	\newcommand{\Preim}{{\rm Preim}\,}
	\newcommand{\ie}{{\it i.e.\/},\ }
	\newcommand{\niso}{\not\cong}
	\newcommand{\nin}{\not\in}
	\newcommand{\soplus}[1]{\stackrel{#1}{\oplus}}
	\newcommand{\by}[1]{\stackrel{#1}{\rightarrow}}
	\newcommand{\longby}[1]{\stackrel{#1}{\longrightarrow}}
	\newcommand{\vlongby}[1]{\stackrel{#1}{\mbox{\large{$\longrightarrow$}}}}
	\newcommand{\ldownarrow}{\mbox{\Large{\Large{$\downarrow$}}}}
	\newcommand{\lsearrow}{\mbox{\Large{$\searrow$}}}
	\renewcommand{\d}{\stackrel{\mbox{\scriptsize{$\bullet$}}}{}}
	\newcommand{\dlog}{{\rm dlog}\,}    
	\newcommand{\longto}{\longrightarrow}
	\newcommand{\vlongto}{\mbox{{\Large{$\longto$}}}}
	\newcommand{\limdir}[1]{{\displaystyle{\mathop{\rm lim}_{\buildrel\longrightarrow\over{#1}}}}\,}
	\newcommand{\liminv}[1]{{\displaystyle{\mathop{\rm lim}_{\buildrel\longleftarrow\over{#1}}}}\,}
	\newcommand{\norm}[1]{\mbox{$\parallel{#1}\parallel$}}
	\newcommand{\boxtensor}{{\Box\kern-9.03pt\raise1.42pt\hbox{$\times$}}}
	\newcommand{\into}{\hookrightarrow}
	\newcommand{\image}{{\rm image}\,}
	\newcommand{\Lie}{{\rm Lie}\,}      
	\newcommand{\CM}{\rm CM}
	\newcommand{\Ma}{\mathbf{M}}
	\newcommand{\Teich}{\rm Teich\;}
	\newcommand{\genus}{{\rm genus}}
	\newcommand{\gonality}{{\rm gonal}}
	\newcommand{\sext}{\mbox{${\mathcal E}xt\,$}}  
	\newcommand{\shom}{\mbox{${\mathcal H}om\,$}}  
	\newcommand{\coker}{{\rm coker}\,}  
	\newcommand{\sm}{{\rm sm}}
	\newcommand{\pgcd}{\text{pgcd}}
	\newcommand{\trd}{\text{tr.d.}}
	\newcommand{\tensor}{\otimes}
	\newcommand{\hotimes}{\hat{\otimes}}
	
	\newcommand{\CH}{{\rm CH}}
	\newcommand{\tr}{{\rm tr}}
	\newcommand{\e}{\rm SH}
	
	\renewcommand{\iff}{\mbox{ $\Longleftrightarrow$ }}
	\newcommand{\supp}{{\rm supp}\,}
	\newcommand{\esssup}{{\rm ess\,sup}}
	\newcommand{\ext}[1]{\stackrel{#1}{\wedge}}
	\newcommand{\onto}{\mbox{$\,\>>>\hspace{-.5cm}\to\hspace{.15cm}$}}
	\newcommand{\propsubset}
	{\mbox{$\textstyle{
				\subseteq_{\kern-5pt\raise-1pt\hbox{\mbox{\tiny{$/$}}}}}$}}
	\newcommand{\sA}{{\mathcal A}}
	\newcommand{\sB}{{\mathcal B}}
	\newcommand{\sC}{{\mathcal C}}
	\newcommand{\sD}{{\mathcal D}}
	\newcommand{\sE}{{\mathcal E}}
	\newcommand{\sF}{{\mathcal F}}
	\newcommand{\sG}{{\mathcal G}}
	\newcommand{\sH}{{\mathcal H}}
	\newcommand{\sI}{{\mathcal I}}
	\newcommand{\sJ}{{\mathcal J}}
	\newcommand{\sK}{{\mathcal K}}
	\newcommand{\sL}{{\mathcal L}}
	\newcommand{\sM}{{\mathcal M}}
	\newcommand{\sN}{{\mathcal N}}
	\newcommand{\sO}{{\mathcal O}}
	\newcommand{\sP}{{\mathcal P}}
	\newcommand{\sQ}{{\mathcal Q}}
	\newcommand{\sR}{{\mathcal R}}
	\newcommand{\sS}{{\mathcal S}}
	\newcommand{\sT}{{\mathcal T}}
	\newcommand{\sU}{{\mathcal U}}
	\newcommand{\sV}{{\mathcal V}}
	\newcommand{\sW}{{\mathcal W}}
	\newcommand{\sX}{{\mathcal X}}
	\newcommand{\sY}{{\mathcal Y}}
	\newcommand{\sZ}{{\mathcal Z}}
	\newcommand{\A}{{\mathbb A}}
	\newcommand{\B}{{\mathbb B}}
	\newcommand{\C}{{\mathbb C}}
	\newcommand{\D}{{\mathbb D}}
	\newcommand{\E}{{\mathbb E}}
	\newcommand{\F}{{\mathbb F}}
	\newcommand{\G}{{\mathbb G}}
	\newcommand{\HH}{{\mathbb H}}
	\newcommand{\LL}{{\mathbb L}}
	\newcommand{\J}{{\mathbb J}}
	\newcommand{\M}{{\mathbb M}}
	\newcommand{\N}{{\mathbb N}}
	\renewcommand{\P}{{\mathbb P}}
	\newcommand{\Q}{{\mathbb Q}}
	\newcommand{\R}{{\mathbb R}}
	\newcommand{\T}{{\mathbb T}}
	\newcommand{\U}{{\mathbb U}}
	\newcommand{\V}{{\mathbb V}}
	\newcommand{\W}{{\mathbb W}}
	\newcommand{\X}{{\mathbb X}}
	\newcommand{\Y}{{\mathbb Y}}
	\newcommand{\Z}{{\mathbb Z}}
	\newcommand{\ch}{{\mathbbm {1}}}
	\newcommand{\bk}{{\mathbf{k}}}
	
	\newcommand{\bp}{{\mathbf{p}}}
	\newcommand{\ep}{\varepsilon}
	\newcommand{\bbk}{{\overline{\mathbf{k}}}}
	\newcommand{\Fix}{\mathrm{Fix}}
	
	\newcommand{\tor}{{\mathrm{tor}}}
	\renewcommand{\div}{{\mathrm{div}}}
	
	\newcommand{\trdeg}{{\mathrm{trdeg}}}
	\newcommand{\Stab}{{\mathrm{Stab}}}
	
	\newcommand{\OK}{{\overline{K}}}
	\newcommand{\ok}{{\overline{k}}}
	
	\newcommand{\cf}{[c.f. ?]}
	\newcommand{\jy}{jy:}
	\title[]{Genus and Gonality of Small Curves, Dynamical Uniform Boundedness, and Bifurcation}
	
	\author{Zhuchao Ji}
	
	\address{Institute for Theoretical Sciences, Westlake University, Hangzhou 310030, China}
	
	\email{jizhuchao@westlake.edu.cn}
	
	\author{Junyi Xie}

	
	\address{Beijing International Center for Mathematical Research, Peking University, Beijing 100871, China}
	
	\email{xiejunyi@bicmr.pku.edu.cn}

	
	\date{\today}

	\begin{abstract}
		We prove that a complex one-parameter algebraic family of endomorphisms of $\mathbb{P}^N$ has a non-empty bifurcation set if and only if the genus grows superlinearly in degree along every generic sequence of dynatomic curves. In this case, gonality also tends to infinity.  In particular, for non-isotrivial families on $\mathbb{P}^1$, gonality tends to infinity  along distinct dynatomic curves, while superlinear genus growth holds for every such sequence exactly when the family is not a flexible Lattès family.  We also establish analogous results for generic sequences of small curves.
		
	As applications, we prove uniform boundedness results for iterated preimages over number fields and geometric uniform boundedness results for preperiodic points over function fields.
		
		The proof combines arithmetic equidistribution, woven currents, and bifurcation theory. The bifurcation mechanism is what forces the growth of genus and gonality.
	\end{abstract}
	
	\maketitle
	
	\tableofcontents

	\section{Introduction}
	
	\subsection{Genus and gonality of small curves}
	Let \begin{align*}
		f:\La\times \P^N&\to \La\times \P^N \\
		(t,z)&\mapsto (t,f_t(z))
	\end{align*}
	be an algebraic family of endomorphisms over $\C$ parametrized by  an irreducible quasi-projective $\La$, where $f_t:\P^N\to \P^N$ are endomorphisms of fixed algebraic degree $d\geq 2$ over $\C$. The integer $d$ is called the \emph{algebraic degree} of $f$. Such a family is called a {\em one-parameter} algebraic  family of endomorphisms of $\P^N$ over $\C$. A one-parameter algebraic family  is called {\em isotrivial} if for every $t\in \La(\C)$, $f_t$ is conjugate to the same fixed endomorphism.
	
	\medskip 
	Let $\pi_1:\La\times \P^N\to \La$ be the canonical projection. Let $\Gamma$ be a Zariski closed one-dimensional subvariety  in $\La\times \P^N$.  We call $\Gamma$  a {\em horizontal curve} if $\Gamma$ is  irreducible and $\pi_1|_\Gamma:\Gamma\to \La$ is dominant.  The following are two types of examples of dynamically meaningful horizontal curves.

	{\em Dynatomic curves:} If $n>m\geq 0$ are two integers, the equation $f^n(t,z)=f^m(t,z)$ defines an algebraic curve in $\La\times \P^N$. A {\em dynatomic curve} is by definition any irreducible component of such curves.  In other words, a dynatomic curve is a family of preperiodic points.

	{\em Preimage curves:}  Let $\Gamma_0$ be a fixed horizontal curve. If $n$ is a non-negative integer,   $f^{-n}(\Gamma_0)$ is an algebraic curve in $\La\times \P^N$. A {\em preimage curve} is by definition any irreducible component of such curves.  In other words, a preimage curve is a family of iterated preimage points. 
	
	\medskip
	When $f$ is isotrivial, the geometry of dynatomic curves is very simple. Up to a finite base change $\La'\to \La$, we may assume that $f$ is trivial; then every dynatomic curve is isomorphic to $\La$ by the canonical projection $\pi_1$.   When $f$ is non-isotrivial,  we expect that the complexity of the geometry of dynatomic curves reflects the complexity of the family of dynamical systems. 
	
	The guiding question of this paper is to make this expectation quantitative.  We ask whether dynamically defined horizontal curves, such as dynatomic curves and preimage curves, must acquire large genus and gonality as soon as the family genuinely varies.  More generally, we study horizontal curves which are small for the canonical dynamical height.  In this form, dynatomic curves and preimage curves are treated by the same principle: arithmetic smallness in a non-isotrivial dynamical family should force geometric complexity.
	
	In the case $N=1$, there  have been substantial works on the geometry of dynatomic curves and preimage curves, most of them focusing on the unicritical family $f_t(z)=z^d+t$ and the flexible Latt\`es family, while little is known for general one-parameter families.  We refer the reader to the survey paper \cite[Section 9 and 10]{benedetto2019current} for some references. Here a {\em Latt\`es map} is by definition an endomorphism of $\P^1$ that is semi-conjugate to an endomorphism of an elliptic curve. A {\em flexible Latt\`es family} is by definition a non-isotrivial one-parameter algebraic family of Latt\`es maps. 
	
	\medskip
	
	Let $\Gamma$ be an irreducible quasi-projective curve over $\C$.  We define $\genus(\Gamma)$ to be the geometric {\em genus} of any projective compactification of $\Gamma$. We define the {\em gonality} $\gonality(\Gamma)$ to be the minimal degree of a dominant rational map from $\Gamma$ to $\P^1$. We have $\gonality(\Gamma)\leq \lfloor (\genus(\Gamma)+3)/2\rfloor$, where $\lfloor \cdot \rfloor$ is the floor function. In particular, if a sequence of irreducible quasi-projective curves $\Gamma_n$ satisfies $ \gonality(\Gamma_n) \to +\infty,$ then $ \genus(\Gamma_n) \to +\infty.$
	
	\begin{fact}\label{fact:gonalityimage}
		Let $\phi:C_1\to C_2$ be a non-constant morphism between irreducible quasi-projective curves over $\C$ of degree $e$. Then
		$$\gonality(C_2)\leq \gonality(C_1)\leq e\,\gonality(C_2).$$
	\end{fact}
	This is a standard property of gonality; see for instance \cite[Chapter III]{ACGH85}. Indeed, the left inequality follows by taking the norm of a gonal function on $C_1$ along the extension $\C(C_1)/\C(C_2)$, and the right inequality follows by composing $\phi$ with a gonal function on $C_2$.

	\subsection{Gonality Conjecture}
	An aim of our paper is to solve the following Gonality Conjecture. 
	\begin{con}\label{conj:gonality}
		For any non-isotrivial one-parameter algebraic family of endomorphisms of $\P^1$ over $\C$, if $\Gamma_n$ is a sequence of distinct dynatomic curves, then $\gonality(\Gamma_n) \to +\infty.$
	\end{con} 
	The above Gonality Conjecture is the one-dimensional case of \cite[Conjecture 10.13 (c)]{benedetto2019current}. It implies that $\genus(\Gamma_n) \to +\infty$, which implies the one-dimensional case of \cite[Conjecture 10.13 (d)]{benedetto2019current}.

	Previously, this conjecture was only known  for  the  unicritical family $f_t(z)=z^d+t$ by Doyle-Poonen \cite{MR4065068}, and  for the  flexible Latt\`es family by Nguyen-Saito \cite{nguyen1996d}.
	%

	\medskip
	
	We actually solve  a  more general version of this conjecture, see Theorem \ref{thm: 1-dim Gonality} below. To formulate our result, we need the notion of {\em Moriwaki height.} This notion was first introduced by Moriwaki in \cite{MR1779799}. In our paper, we use the version developed by Yuan--Zhang \cite{yuan2021} in the framework of adelic line bundles.
	
	Let $f$ be a non-isotrivial one-parameter algebraic family of endomorphisms of $\P^N$ of degree $d\geq 2$ over $\C$. Let $K$ be a subfield of $\C$ which is finitely generated over $\Q$ such that $f$ is over $K$. Following Yuan--Zhang \cite{yuan2021} (see Section \ref{Sec_mwheight} for details), for every polarization $\overline{H}\in \widehat{\Pic}(K/\Z)_{\rm nef,\Q}$, we can define the canonical Moriwaki height of every horizontal curve $\Gamma$ in $\La\times \P^N$ over $\overline K$, denoted by $\hat{h}^{\overline{H}}_f(\Gamma)\in \R_{\geq 0}$, which satisfies the following two properties:
	\smallskip
	
	(1) $\hat{h}^{\overline{H}}_f(f(\Gamma))=d\hat{h}^{\overline{H}}_f(\Gamma)$;
	\smallskip
	
	(2)({\em Northcott property:}) If $\overline{H}$ is big and nef, then for each $D\geq 0$ and  $B\geq 0$, the number of horizontal curves $\Gamma$ over $\overline{K}$ with  $[K(\Gamma):K]\leq D$ and $\hat{h}^{\overline{H}}_f(\Gamma)\leq B$ is finite. 
	
	\smallskip
	
	As a consequence of the  above two properties,  when 
	$\overline{H}$ is big, 
	$\hat{h}^{\overline{H}}_f(\Gamma)=0$ 
	if and only if $\Gamma$ is a dynatomic curve.  If $\Gamma$ is a preimage curve, $f^n(\Gamma)=\Gamma_0$, 
	then $\hat{h}^{\overline{H}}_f(\Gamma)=d^{-n} \hat{h}^{\overline{H}}_f(\Gamma_0)$. In particular, if  $(\Gamma_n)_{n\geq 1}$ is a sequence of distinct dynatomic curves or preimage 
	curves, then $\hat{h}^{\overline{H}}_f(\Gamma_n)\to 0$. 
	Following Yuan--Zhang \cite[Section 5.4.1]{yuan2021}, a sequence  $\Gamma_n, n\geq 0$ is called \emph{small for $f$} if for any  $\overline{H}\in\widehat{\mathrm{Pic}}(K/\Z)_{\mathrm{nef,\Q}}$,  $\hat{h}_{f}^{\overline{H}}(\Gamma_n)$ converges to $0.$
	Any sequence of dynatomic or preimage curves is small for $f$.

	The following  Theorem \ref{thm: 1-dim Gonality} and Theorem \ref{thm: high-dim Genus} are our first main results. These results in particular hold when $(\Gamma_n)_{n\geq 1}$ is a sequence of dynatomic curves or preimage curves.
	
	\begin{thm}[Gonality  Conjecture for horizontal  curves with small heights]\label{thm: 1-dim Gonality}
		Let $f$ be a non-isotrivial one-parameter algebraic family of endomorphisms of $\P^1$ of degree $d\geq 2$ over $\C$, which is not the flexible Latt\`es family. Let $K$ be any subfield of $\C$ which is finitely generated over $\Q$ such that $f$ is over $K$.   Let $(\Gamma_n)_{n\geq 1}$ be a sequence of distinct horizontal curves over $\overline K$ that is small for $f$. Then
		$$\gonality(\Gamma_n) \to +\infty;$$
		and
		$$\frac{\genus(\Gamma_n)}{\deg(\Gamma_n)} \to +\infty.$$

	\end{thm}
	
	Here $\deg(\Gamma_n)$ is the degree of the canonical projection  map $\pi_1|_{\Gamma_n}:\Gamma_n\to \La$.

	If $\Gamma_n$ is a sequence of dynatomic curves and $f$ is the flexible Latt\`es family, Nguyen-Saito\cite{nguyen1996d} proved that $\gonality(\Gamma_n) \to +\infty.$  Thus  we can prove Conjecture \ref{conj:gonality}  by combining Nguyen-Saito's result and Theorem \ref{thm: 1-dim Gonality}.

	Let us also emphasize the second conclusion of Theorem \ref{thm: 1-dim Gonality}.  The divergence of $\genus(\Gamma_n)/\deg(\Gamma_n)$ reflects a genuinely dynamical source of ramification.  This behavior is different from the towers naturally arising from abelian schemes: over the good reduction locus such covers are typically \'{e}tale, and the possible ramification is confined to a fixed bad set; Riemann--Hurwitz then gives at most linear genus growth in the degree.  In our setting, bifurcation creates new geometric complexity along the family, and this is what forces the normalized genus to diverge.
	
		%
		%
		
		
		\medskip
		The proof of Theorem \ref{thm: 1-dim Gonality} relies  on one-dimensional bifurcation theory  developed by Ma\~n\'e--Sad--Sullivan \cite{mane1983dynamics}, Lyubich \cite{lyubich1983some,lyubich1984an} and DeMarco \cite{demarco2003dynamics}.  Our next goal is to extend Theorem \ref{thm: 1-dim Gonality} to higher dimension.  A bifurcation theory for families of endomorphisms of $\P^N$ over $\C$ for $N\geq 2$ was developed by Bassanelli-Berteloot \cite{bassanelli2007bifurcation} and  by Berteloot--Bianchi--Dupont \cite{berteloot2018dynamical}. The {\em stable set} of a family of endomorphisms of $\P^N$ is by definition the locus where the Lyapunov exponent function is pluriharmonic, and  the {\em bifurcation set} is the complement of the stable set. By the recent work of Bianchi--Ji--Xie \cite{bianchi2026equilibrium}, the existence of an equilibrium web implies stability for arbitrary holomorphic families of endomorphisms of $\P^N$.  See Section \ref{section:bifur} for more details. 
		
		In dimension one, stability theory, together with McMullen's rigidity theorem, we have a strong description of stable non-isotrivial families: apart from the flexible Latt\`es phenomenon, stability forces isotriviality. For endomorphisms of $\P^N$, $N\geq 2$, no comparable classification of stable families is currently known. We also discuss in Sections \ref{sec:height} and \ref{sec:assumptionA} an adelic-line-bundle viewpoint on this rigidity problem, where stability is related to the degeneracy of the bifurcation class on moduli.
		
		Throughout the paper, a Hermitian manifold is always assumed to be {\bf connected}.
		
		We note that the assumption that $f$ has non-empty bifurcation set implies that  $f$ is non-isotrivial. 
		
		\medskip

		Let $f$ be a one-parameter algebraic family of endomorphisms of $\P^N$ over $\C$, parametrized by $\La$. A sequence of horizontal curves $(\Gamma_n)_{n\geq 1}$ is called {\em generic} if for any proper subvariety $H\subset \La\times \P^N$, there are only finitely many $\Gamma_n$ contained in $H$. In particular, when $N=1$, a distinct sequence of horizontal curves  is generic. 
		
		The following  result extends Theorem \ref{thm: 1-dim Gonality} to higher dimension.
		
		\begin{thm}[Genus and gonality of  horizontal curves with small heights for families of endomorphisms of $\P^N$]\label{thm: high-dim Genus}
			Let $f$ be a one-parameter algebraic family of endomorphisms of $\P^N$ over $\C$ with non-empty bifurcation set.  Let $K$ be any subfield of $\C$ which is finitely generated over $\Q$ such that $f$ is over $K$.   Let $(\Gamma_n)_{n\geq 1}$ be a generic sequence of horizontal curves over $\overline K$ that is small for $f$. Then
			$$\gonality(\Gamma_n) \to +\infty;$$
			and 
			$$\frac{\genus(\Gamma_n)}{\deg(\Gamma_n)} \to +\infty.$$
			
			\medskip

		\end{thm}
		For dynatomic curves, the assumption that the bifurcation set is non-empty is necessary for the growth of genus relative to degree in Theorem \ref{thm: high-dim Genus}. When $N=1$, the bifurcation set is empty if and only if $f$ is isotrivial or $f$ is a flexible Latt\`es family, by McMullen's rigidity theorem \cite[Theorem 2.2]{McMullen1987}, see also \cite{ji2023homoclinic}.

		\begin{thm}\label{thm:stable-dynatomic}
			Let $f$ be a one-parameter algebraic family of endomorphisms of $\P^N$ of degree $d\geq 2$ over $\C$, parametrized by a smooth quasi-projective curve $\La$. Then the following are equivalent:

			(1) The bifurcation set of $f$ is empty;

			(2) There exists a generic sequence of dynatomic curves $(\Gamma_n)_{n\geq 1}$ such that
			$$\sup_{n\geq 1}\frac{\genus(\Gamma_n)}{\deg(\Gamma_n)}<+\infty.$$
		\end{thm}

		Theorems \ref{thm: 1-dim Gonality}, \ref{thm: high-dim Genus} and \ref{thm:stable-dynatomic} are proved in Section \ref{sec:gonality-proof}.
		\medskip
		
		\begin{rmk}
			When $N\geq 2$, the assumption ``$(\Gamma_n)_{n\geq 1}$ is a generic sequence" cannot be replaced by ``$(\Gamma_n)_{n\geq 1}$ is a distinct sequence" due to the possible existence of  periodic isotrivial subvariety of positive dimension.
			
		\end{rmk}
		\medskip
		

		A {\em Latt\`es map} on $\P^N$ is by definition an endomorphism of $\P^N$ over $\C$  that is semi-conjugate (by a finite morphism $\pi:A\to \P^N$) to an endomorphism of an abelian variety $A$.  The Latt\`es locus containing all Latt\`es maps is a proper subvariety in $\End_d(\P^N)$.   
		
		
		
		We have the following examples of one-parameter algebraic families of endomorphisms of $\P^N$ with non-empty bifurcation set for every $N\geq 1$.
		\begin{thm}\label{thm: example2}
			Let $f$ be a one-parameter algebraic family of endomorphisms of $\P^N$ of degree $d\geq 2$ over $\C$ which is not a Latt\`es family.  If there is a parameter $t\in \La(\C)$ such that $f_t$ is a Latt\`es map, then the bifurcation set of $f$ is non-empty. 
			
		\end{thm}
		
		\begin{prop}
			\label{prop:very-general-assumptionA}
			Let $f$ be a non-isotrivial one-parameter algebraic family of endomorphisms of $\P^N$ of degree $d\geq 2$ over $\C$. Assume that the image of the moduli map of $f$ passes through a general point of $\sM_{d,N}$. Then the bifurcation set of $f$ is non-empty.
		\end{prop}
		
		Theorem \ref{thm: example2} and Proposition \ref{prop:very-general-assumptionA} are proved in Section \ref{sec:assumptionA}.
		
		Theorem \ref{thm: 1-dim Gonality} and \ref{thm: high-dim Genus} have consequences toward the Uniform Boundedness Conjecture for iterated preimages and preperiodic points for one parameter families. The link is provided by a theorem of Frey: a curve with infinitely many algebraic points of bounded degree must have bounded gonality.  Thus gonality growth turns into uniform boundedness once one organizes preperiodic points or iterated preimages into suitable horizontal curves.
		
		\subsection{Uniform Boundedness Conjecture for  iterated preimages over number fields}
		
		Let $g$ be an endomorphism of $\P^N$ over a number field. Let $a\in \P^N(\overline{\Q})$.  
		Let $$\Preim(g,a):=\left\{x\in \P^N(\overline{\Q}):g^n(x)=a\;\text{for some}\; n\geq 0\right\}.$$
		
		Let $D\geq 1$. We are interested in the set of rational  iterated preimages of $a$ with degree bounded by $D$:
		$$\Preim(g,a,D):=\left\{x\in\Preim(g,a): [\Q(x):\Q]\leq D\right\}.$$

		By standard properties of the canonical height associated to $g$, for fixed $g$ and $a$, the cardinality of $\Preim(g,a,D)$ is finite.  It is a  natural problem to see how this cardinality  depends on $g$ and $a$. This problem is reminiscent of the classical problem of bounding torsion points on abelian varieties.
		
		Let $\End_{d,N}$ be the space of all degree $d$ endomorphisms of $\P^N$ over $\C$. The group $\PGL_{N+1}(\C)= \Aut(\P^N(\C))$ acts on $\End_{d,N}(\C)$ by conjugacy. The geometric quotient $\sM_{d,N}:=\End_{d,N}/\PGL_{N+1}$ is the (coarse) \emph{moduli space},  which is a complex affine variety.  Let $\pi:\End_{d,N} \to  \sM_{d,N}$ be the canonical projection.  If $f$ is an algebraic family of degree $d$ endomorphisms of $\P^N$ over $\C$, parametrized by a quasi-projective variety $\La$, then the canonical projection  $\pi$ induces a morphism $\pi_\La:\La\to \sM_{d,N}$. We say the family $f$ is {\em maximal varying} if $\pi_\La:\La\to \sM_{d,N}$ is quasi-finite.  If $\dim \La=1$, then maximal varying is equivalent to non-isotrivial. By a {\em marked point}, we mean a morphism $a:\La\to \P^N$, $t\mapsto a_t$. 
		
		For $N=1$, a  slightly weaker version of the  following conjecture was proposed by A. Levin \cite{MR2993960}.  Instead of for iterated preimages, a  similar Uniform Boundedness Conjecture for preperiodic points was proposed by Morton-Silverman \cite{morton1994rational}.  Fakhruddin \cite{MR1995861} showed that Morton-Silverman's conjecture \cite{morton1994rational} implies the Uniform Boundedness Conjecture for  torsion points on abelian varieties. 
		\begin{con}[Uniform Boundedness for  iterated preimages]\label{con: ubc for preimage}
			Let $f$ be a maximal varying algebraic family of degree $d$ endomorphisms of $\P^N$ parametrized by a quasi-projective variety $\La$ and let $a$ be a marked point, both over a number field. Then for each $D\geq 1$, there exists $C=C(f,a, D)>0$ and  a proper subvariety $V_{f,a}$ in $\La\times \P^N$ over $\overline{\Q}$ depending only on $(f,a)$, such that 
			$$\#\left(\Preim(f_t,a_t, D)\setminus V_{f,a}(\overline{\Q})\right)\leq C$$
			for every $t\in \La(\overline \Q)$ such that $[\Q(t):\Q]\leq D$.  When $N=1$, we require that  $V_{f,a}=\emptyset$.
		\end{con}

		\medskip
		
		As a consequence of Theorem \ref{thm: 1-dim Gonality}, we completely solve Conjecture \ref{con: ubc for preimage} for one-parameter families of endomorphisms of $\P^1$, i.e. when $N=1$ and $\dim \La=1$.  
		
		\begin{thm}\label{thm:1-dim preimages}
			Let $f$ be a non-isotrivial one-parameter algebraic family of degree $d$ endomorphisms of $\P^1$ parametrized by a quasi-projective curve $\La$ and let $a$ be a marked point, both over a number field. Then for each $D\geq 1$,  there exists $C=C(f,a, D)>0$ such that 
			$$\#\Preim(f_t,a_t, D)\leq C$$
			for every $t\in \La(\overline \Q)$ such that $[\Q(t):\Q]\leq D$.
		\end{thm}
		\medskip
		For $N\geq 1$ and $\dim \La=1$, as a consequence of Theorem \ref{thm: high-dim Genus}, we solve  Conjecture \ref{con: ubc for preimage}  under the assumption that $f$ has non-empty bifurcation set and the marked point has Zariski dense forward orbit.
		\begin{thm}\label{thm:high-dim preimages}
			Let $N\geq 1$. Let $f$ be a one-parameter algebraic family of degree $d$ endomorphisms of $\P^N$ parametrized by a quasi-projective curve $\La$ and let $a$ be a marked point, both over a number field. Assume that the bifurcation set of $f$ is non-empty. Let $\Gamma_a\subset \La\times \P^N$ be the graph of $a$. Assume that the forward orbit $\{f^n(\Gamma_a)\}_{n\geq 0}$ is Zariski dense in $\La\times \P^N$. Then for each $D\geq 1$,  there exists $C=C(f,a, D)>0$  such that 
			$$\#\Preim(f_t,a_t, D)\leq C$$
			for every $t\in \La(\overline \Q)$ such that $[\Q(t):\Q]\leq D$.
		\end{thm}
		
		Theorems \ref{thm:1-dim preimages} and \ref{thm:high-dim preimages} are proved in Section \ref{sec:uniform-boundedness-proof}.
		
		\begin{rem}
			(1) The assumption that $f$ is non-isotrivial in Theorem \ref{thm:1-dim preimages} cannot be removed.  Let $g$ be an endomorphism of $\P^1$ over a number field. Let $f$ be the trivial one-parameter algebraic family parametrized by $\P^1$, i.e. $f_t=g$ for every $t\in \P^1(\overline{\Q})$.   Let $a:\P^1\to \P^1$ be the identity morphism.  Let $t_0\in \P^1(\overline{\Q})$ with infinite $g$-orbit.  Let $D:=[\Q(t):\Q]$.  For $n\geq 0$,  let $t_n:=g^n(t_0)$. Since $\left\{t_0,\dots, t_n\right\}\subset \Preim(f_{t_n},a_{t_n}, D)$,  we have $\#\Preim(f_{t_n},a_{t_n}, D)\geq n+1$. This implies that the constant $C$ in Theorem \ref{thm:1-dim preimages} does not exist.
			
			(2) The Zariski dense orbit assumption in Theorem \ref{thm:high-dim preimages} is used to remove the exceptional subvariety from the conclusion. Without this assumption, when $N\geq 2$, one should in general allow an exceptional subvariety due to the possible existence of periodic isotrivial subvarieties of positive dimension.
			
		\end{rem}
		
		\medskip

		\subsection{The Geometric Uniform Boundedness Conjecture}

		Let $B$ be an irreducible quasi-projective curve over $\C$, and let $k=\C(B)$ be the corresponding complex function field.  Throughout the paper, by a complex function field we mean a function field over a complex quasi-projective curve.  We define $\gonality(k):=\gonality(B)$.  An endomorphism $f$  on $\P^N$ over $k$ can be naturally seen as a one-parameter algebraic family of endomorphisms of $\P^N$, parametrized by a Zariski open subset of  $B$.  We say that $f$ is {\em isotrivial} if the associated one-parameter algebraic family is isotrivial. For each $D\geq 1$, we define $$\Preper(f,D):=\left\{x\in \P^N(\overline{k}): x\;\text{is $f$-preperiodic,}\;\gonality(k(x))\leq D\right\}.$$ 
		
		Theorem \ref{thm: 1-dim Gonality} and Theorem \ref{thm: high-dim Genus} have consequences on the Geometric Uniform Boundedness Conjecture (GUBC). The original Uniform Boundedness Conjecture proposed by Morton-Silverman \cite{morton1994rational} is for number fields, the following is a function field analogue of Morton-Silverman's conjecture. 
		\begin{con}[Geometric Uniform Boundedness Conjecture]\label{conj: UBC}
			Given integers $D\geq 1$, $N\geq 1$ and $d\geq 2$, there exists $C=C(d,N,D)>0$ such that if $f$ is a degree $d$ non-isotrivial endomorphism of $\P^N$ over  a complex  function field $k$, then there exists a proper subvariety $V_{f,D}\subset \P^N$ over $\overline{k}$ such that 
			$$\#\left(\Preper(f,D)\setminus V_{f,D}(\overline{k})\right)\leq C.$$
			When $N=1$, we require that  $V_{f,D}=\emptyset$.
		\end{con}
		\begin{rmkU}
			In view of Fact \ref{fact:gonalityimage}, if $\Preper(f,D)$ is non-empty, then one automatically has $\gonality(k)\leq D$. Thus, in the above conjecture, one only needs to consider complex function fields $k$ of gonality at most $D$.
		\end{rmkU}

		When $N\geq 2$, in general $V_{f,D}$ can be non-empty due to the possible existence of periodic isotrivial subvariety of positive dimension.
		
		\medskip
		
		Conjecture \ref{conj: UBC} was previously only known for two special one-parameter families, namely the unicritical family $f_t(z)=z^d+t$ and the flexible Latt\`es family.   
		
		Let us be more precise about {\bf the meaning of saying that GUBC holds along a one-parameter algebraic family.} Let $\sM_{d,N}$ be the moduli space of degree $d$ endomorphisms of $\P^N$ over $\C$. Let $S\subset \sM_{d,N}$ be an irreducible algebraic curve over $\C$.  Let $f$ be a degree $d$ non-isotrivial endomorphism of $\P^N$ over a complex function field $k=\C(B)$. After shrinking $B$, the canonical projection $\pi$ induces a canonical morphism $\pi_B:B\to \sM_{d,N}$. We say that $f$ {\em lies over} $S$ if $\pi_B(B)$ is a Zariski open subset of $S$.
		

		\begin{defi}
			Let $S\subset \sM_{d,N}$ be an irreducible algebraic curve over $\C$.    We say that the GUBC holds along $S$ if, for every integer $D\geq 1$, there exists $C=C(D,S)>0$ such that if $f$ is a degree $d$ non-isotrivial endomorphism of $\P^N$ over a complex function field $k$ which lies over $S$, then there exists a proper subvariety $V_{f,D}\subset \P^N$ over $\overline{k}$ such that 
			$$\#\left(\Preper(f,D)\setminus V_{f,D}(\overline{k})\right)\leq C.$$
			
			Moreover when $N=1$, we require that   $V_{f,D}=\emptyset$.
		\end{defi}
		
		As a consequence of Theorem \ref{thm: 1-dim Gonality}, we establish GUBC for every one-parameter algebraic family of endomorphisms of $\P^1$.
		
		\begin{thm}[One-parameter Geometric Uniform Boundedness for endomorphisms of $\P^1$]\label{thm: 1-dim UBC}
			The GUBC holds along every irreducible algebraic curve $S\subset \sM_{d,1}$ over $\C$.
		\end{thm}

		%
		%
		%
		
		 One can also define this condition for an irreducible algebraic curve $S\subset \sM_{d,N}$: we say that $S$ has non-empty bifurcation set if a one-parameter algebraic family whose moduli map is generically finite onto $S$ has non-empty bifurcation set.  As a consequence of Theorem \ref{thm: high-dim Genus}, we have
		
		\begin{thm}[One-parameter Geometric Uniform Boundedness for endomorphisms of $\P^N$]\label{thm: high-dim UBC}
			
			Let  $S\subset \sM_{d,N}$ be an  irreducible algebraic curve over $\C$ with non-empty bifurcation set. Then the GUBC holds along $S$. 
		\end{thm}
		
		Theorems \ref{thm: 1-dim UBC} and \ref{thm: high-dim UBC} are proved in Section \ref{sec:uniform-boundedness-proof}.
		


		\subsection{Previous results}
		
		\subsubsection{Results on dynatomic curves and the Gonality Conjecture.}
		For the unicritical family $f_t(z)=z^d+t$, the genus and topology of periodic and preperiodic dynatomic curves were studied by Bousch \cite{bousch1992}, Morton \cite{morton1996certain}, and Gao \cite{MR3460633}. In particular, Morton obtained genus estimates for periodic dynatomic curves, and Gao computed the genera of preperiodic dynatomic curves. The gonality growth for the unicritical family was proved by Doyle--Poonen \cite{MR4065068}, using these geometric inputs together with finite-field arguments and the Castelnuovo--Severi inequality. The flexible Latt\`es family was treated by Nguyen--Saito \cite{nguyen1996d}. By contrast, the mechanism in this paper is bifurcational: in a non-isotrivial dynamical family, bifurcation forces geometric complexity. Arithmetic equidistribution and woven currents are used to connect this bifurcation mechanism with the geometry of small curves.

		\subsubsection{Results for the Uniform Boundedness Conjecture for iterated preimages(=Conjecture \ref{con: ubc for preimage}).}
		All previous results for Conjecture \ref{con: ubc for preimage} concern one-parameter families of endomorphisms of $\P^1$.
		\begin{points}
			\item When $f$ is the quadratic polynomial family $f_t(z)=z^2+t$ and  $a$ is a constant marked point, Faber et al \cite{MR2480563} proved Conjecture \ref{con: ubc for preimage} for all but finitely many $a\in \overline \Q$.
			
			\item Ingram \cite{MR2870098} proved Conjecture \ref{con: ubc for preimage} for the one-parameter flexible Latt\`es family on $\P^1$.
			
			\item In the setting of abelian varieties over number fields,  Cadoret--Tamagawa \cite[Theorem 1.1]{MR2897693} proved the uniform boundedness of $\ell$-primary torsion points for abelian varieties over a fixed base curve. Ji--Song--Xie \cite[Theorem 1.6]{ji2026geometric} gave a geometric proof of Cadoret--Tamagawa's theorem and extended it to the setting of iterated preimages of the $\times \ell$ map.
			
			\item   A. Levin \cite{MR2993960} and Ingram \cite{MR3064415} obtained partial results for Conjecture \ref{con: ubc for preimage}  for certain one-parameter families of endomorphisms of $\P^1$  when the number of places of bad reduction is controlled.
			
			\item Explicit bounds on the number of rational iterated preimages for the quadratic polynomial family $f_t(z)=z^2+t$  were studied by Faber \cite{MR2607018}, Faber--Hutz--Stoll \cite{MR2854215}, Hutz--Hyde--Krause \cite{MR2905234} and Sano \cite{sano2025number}. Sano \cite{sano2025number} also obtained explicit bounds for the family $f_t(z)=z^d+t$, $d\geq 3$.
		\end{points}
		%
		%
		\subsubsection{Results for the Geometric Uniform Boundedness Conjecture for preperiodic points(=Conjecture \ref{conj: UBC}).}
		
		

		All of the following results concern endomorphisms of $\P^1$, except the results for abelian varieties:
		
		\begin{points}
			\item Nguyen-Saito  \cite{nguyen1996d} proved that GUBC holds along  the flexible Latt\`es family. 
			
			\item  Doyle-Poonen \cite{MR4065068} proved that GUBC holds along  the unicritical family $f_t(z)=z^d+t$.
			They also obtained parallel results over function fields of positive characteristic. 
			
			\item Also in the positive characteristic case, Doyle--Faber \cite{MR4693952} recently extended Doyle--Poonen's result to some other one-parameter families. 
			
			\item Assuming the $abcd$ Conjecture (a strong form of the $abc$ conjecture), Looper \cite{looper2021dynamical} \cite{looper2021uniform} proved Conjecture \ref{conj: UBC} in the set of polynomials over complex function fields. 
			
			\item Benedetto \cite{MR2339471}, Canci \cite{MR3441644} and Canci-Paladino\cite{MR3556260} obtained partial results of Conjecture \ref{conj: UBC} concerning the case when the number of the places of bad reduction of $f$ is uniformly bounded from above.
			
			\item In the setting of abelian varieties over complex function fields, Ji--Song--Xie \cite[Theorem 1.4]{ji2026geometric} proved that the genus of a generic sequence of horizontal curves with small Moriwaki heights over a fixed base curve goes to infinity. This extends previous results of Cadoret--Tamagawa \cite[Theorem 1.2]{MR2897693}, \cite[Theorem 1.3]{MR2842079}. Bakker--Tsimerman proved the Geometric Torsion Conjecture for abelian varieties with real multiplication \cite{MR3825605}. Remarkably, Looper--Yap recently completely solved the Geometric Torsion Conjecture for abelian varieties \cite{LooperYap2026uniform}.


		\end{points}
		
		\subsubsection{Results for the Uniform Boundedness Conjecture for preperiodic points over number fields.}
		This is the number-field analogue of Conjecture \ref{conj: UBC}.
		\begin{con}[Morton-Silverman \cite{morton1994rational}]
			Let $d\geq 2$, $N\geq 1$. Let $k$ be a number field.    Then  there exists  a constant
			$C = C(d, N, [k:\Q])>0$ such that  if $f$ is  an endomorphism of $\P^N$ of degree $d$ over $k$,  then
			$$\Preper(f,k)\leq C,$$ 
			where
			$$\Preper(f,k):=\left\{x\in \P^N(k): x\;\text{is $f$-preperiodic}\right\}.$$
			
		\end{con}
		All known results for the above conjecture are for endomorphisms of $\P^1$.
		\begin{points}
			\item Using  deep results of Mazur \cite{MR482230}, Kamienny \cite{MR1172689}, and Merel \cite{MR1369424}, one can show that the above conjecture holds for the set of Latt\`es maps on $\P^1$.
			
			
			\item Assuming the $abcd$ Conjecture (a strong form of the $abc$ conjecture), Looper \cite{looper2021dynamical} \cite{looper2021uniform} proved that the above conjecture holds for the set of polynomials on $\P^1$. Recently, Robin Zhang \cite[Corollary~1]{zhang2026uniform} proved uniform boundedness of $K$-rational preperiodic points for unicritical polynomials $z^d+c$ over each fixed number field $K$ when $d\geq 4$, and also when $d=3$ provided that $\Q(\sqrt{-3})\not\subset K$.
			
			\item Besides the above results, many other works provide evidence of Morton-Silverman's conjecture.  We recommend \cite{MR4693952} and \cite{benedetto2019current} for a survey of those results.
		\end{points}
		
		\medskip

		\subsection[Woven currents theory]{Woven currents  theory}

		Our main new analytic idea in this paper  is to combine bifurcation theory  with the theory of woven currents. This approach leads to two results of independent interest in the theory of woven currents. 
		
		Our work builds on woven current theory, originally developed by Bedford--Lyubich--Smillie \cite{Bedford-Lyubich-Smillie}, and its further developments due to Cantat \cite{cantat01}, Dujardin \cite{dujardin2003laminar, dujardin04, Dujardin2012, Dujardin2013}, De Th\'elin \cite{dethelin04, dethelin08} and Dinh \cite{dinh05}. See also \cite{dujardin23} for a nice  introduction of woven currents theory. 

		Let $X$ be a complex smooth quasi-projective variety and let $\Gamma_n$ be a sequence of  algebraic curves in $X$.  We consider weak limits of the normalized currents of integration $[\Gamma_n]/\vol(\Gamma_n)$. The limiting objects produced by the curves are positive closed currents. The assumptions that the ratio of genus and volume is bounded,  or gonality is bounded  force these limits to remember how they were built from curves. 
		

{\em The case that $\Gamma_n$ satisfying the genus--volume relation.}  We extend Dujardin's intersection theory for woven currents on surfaces \cite{dujardin2003laminar,dujardin04} and his higher-dimensional result for rational curves \cite[Theorem~4.7]{Dujardin2012} to algebraic curves satisfying the genus--volume relation in arbitrary dimension. Our proof differs from Dujardin's in the way we pass to the limit. We develop the theory in an abstract measure-theoretic framework on spaces of analytic curves whose volumes are uniformly bounded on each compact subset. By taking weak limits of finite Borel measures on these spaces, we obtain the uniformly woven currents used to approximate intersections. See Theorem~\ref{thm: strongly approximable}.
		
		{\em The case that $\Gamma_n$ has bounded gonality.} We also develop the intersection theory for woven currents arising as limits of algebraic curves with bounded gonality.   For a sequence of irreducible curves of bounded gonality, we may pass to a subsequence of fixed gonality $m$. The fibers of a degree-$m$ map from each curve to $\P^1$ gives us  a rational curve in the symmetric product $\operatorname{Sym}^m(X)$. Using this construction, we realize $S$ as the pushforward, under the first projection $X^m\to X$, of a current satisfying the same intersection approximation on $X^m$. This result is new even when $X$ is a surface.  See Theorem~\ref{thm:sa gonality}.

		\medskip
		
		\subsection{Strategy of the proof}
		
		We first explain the proof of Theorem \ref{thm: 1-dim Gonality}, which contains the main ideas of the paper.  The argument is by contradiction.  Suppose that there is a small sequence of distinct horizontal curves $\Gamma_n$ for which the desired genus or gonality growth fails.  Arithmetic equidistribution connects the smallness of the curves with the relative dynamical Green current. This provides the arithmetic input through which bifurcation enters the argument.

		  For the genus statement, the argument is carried out directly on $\La\times \P^1$: if $\genus(\Gamma_n)/\deg(\Gamma_n)$ does not tend to infinity, then we use the  genus-volume estimates to construct a weak woven-type structure for the limiting currents.  If the gonality of $\Gamma_n$ is constantly $m$, we construct a weak woven-type structure for the limiting currents on the $m$-th symmetric product of $\La\times \P^1$.

		 In both cases, this weak woven structure produces normal families of curves in the total  space $\La\times \P^1$.
		
	 The actual source of complexity of the geometry of dynatomic curves  is bifurcation.   In a non-isotrivial one-dimensional family which is not flexible Latt\`es, normality of the family of curves produced above is incompatible with the motion of repelling cycles on the bifurcation locus.  Using the stability theory of Ma\~n\'e--Sad--Sullivan and Lyubich, together with McMullen's rigidity theorem, we get a contradiction.  This proves Theorem \ref{thm: 1-dim Gonality}.
		
		The higher-dimensional theorem follows the same philosophy, but the role of one-dimensional stability theory is replaced by the bifurcation theory for endomorphisms of $\P^N$.  The uniform boundedness results for iterated preimages can be  obtained from the  growth of gonality by using Frey's theorem, combine with Northcott's theorem on a fixed fiber. For preperiodic points over function fields, we get the uniform bound still by the  growth of gonality.  
		
		\medskip
		
		\subsection{Organization of the paper}
		
		In Sections \ref{sec:wovencurrents} and \ref{sec:woven11} we develop the analytic theory of woven currents needed in the proof. Section  \ref{sec:wovencurrents}  establishes the basic compactness and decomposition properties. Sections \ref{sec:woven11} study the limiting behavior of woven currents arising from algebraic curves and intersection theory. In Section \ref{section:bifur} we recall the bifurcation theory. In Section \ref{sec:slicing} we recall the slicing theory.  In Section \ref{sec:height} we recall the height theory for polarized dynamical systems over finitely generated fields, following Yuan--Zhang, and discuss the adelic viewpoint on bifurcation classes. In Section \ref{sec:gonality-proof} we prove the gonality and genus-growth theorems and the characterization of stability by dynatomic curves (Theorems \ref{thm: 1-dim Gonality}, \ref{thm: high-dim Genus} and \ref{thm:stable-dynatomic}). Section \ref{sec:uniform-boundedness-proof} derives the uniform boundedness results over number fields and complex function fields (Theorem \ref{thm:1-dim preimages}, \ref{thm:high-dim preimages}, \ref{thm: 1-dim UBC} and \ref{thm: high-dim UBC}). Finally, In Section \ref{sec:assumptionA} we  prove the examples and genericity statements for families with non-empty bifurcation set (Theorem \ref{thm: example2} and Proposition \ref{prop:very-general-assumptionA}).

		\subsection*{Acknowledgement}
		The authors would like to thank Fran\c{c}ois Bertel-oot, Fabrizio Bianchi, Romain Dujardin, and Jit Wu Yap for the comments on the first version of the paper.  Zhuchao Ji would like to thank Beijing International Center for Mathematical Research in Peking University for the invitation.  The main result of this work was reported by the authors in several occasions, including at  ``International Symposium on
			Dynamical Systems'' held at Peking University in  April 2025,   and at ``Harbin 2026 Conference on Complex Dynamics and Arithmetic Dynamics'' held at Harbin Institute of Technology in  January 2026.  We acknowledge the organizers for their kind invitation.
		\medskip
				\subsection*{Funding}
				Zhuchao Ji is supported by the National Key R\&D Program of China (No. 2025YFA1018300), NSFC Grant (No.12401106), and ZPNSF Grant (No. XHD24A0201). Junyi Xie is supported by NSFC Grant (No. 12271007)  and by the Xplorer Prize from the New Cornerstone Science Foundation.
		\medskip
		\subsection*{AI disclosure}
		All  statements and proofs in the paper were developed by the authors. After a draft was written,   an AI assistant Xiaozhua, running in the OpenClaw environment with OpenAI models was used to check and edit the paper. 
		\medskip

		\section{Uniformly woven currents}\label{sec:wovencurrents}
		\subsection{Basic properties of uniformly woven currents}\label{sec:basicwoven}
		
		In this subsection, we consider positive closed currents that admit a decomposition as an integral of closed irreducible analytic subvarieties with uniformly bounded volume from above on any compact subset, which are called {\em uniformly woven currents}.
		
		\medskip

		Let $(X,\omega)$ be a Hermitian manifold of dimension $N$.  The volume of an analytic subvariety $V$ ( resp. more generally when $S$ is a positive current) of $X$ is computed with respect to $\omega$ and is denoted by $\vol(V)$ (resp. $\vol(S)$).  For a subset $K\subset X$, we denote $\Ma_K(S)$ the volume of the positive current $S$ on $K$. 
		
		\medskip
		
		Let $K_1\subset K_2\subset \cdots$ be an exhaustion by compact sets of $X$, i.e. $K_i\subset K_{i+1}^\circ$ and $\cup_{i=1}^\infty K_i=X$, where $K^\circ$ is the interior of $K$.   We fix this exhaustion  once for all. 
		
		\medskip
		
		Let $0\leq p\leq N$.  For a compact subset $K\subset X$, let $\sD^p(K)$ be the Banach space of continuous $(p,p)$-forms with support contained in $K$, where the norm is the supremum norm.  Let $\sD^p(X)$ be the space of continuous $(p,p)$-forms with compact support, i.e. $\sD^p(X)=\cup_{i=1}^\infty \sD^p(K_i)$.  Via the continuous embedding $j_n: \sD^p(K_n)\to \sD^p(K_{n+1})$, $\sD^p(X)$ is a  {\em strict LB space} (i.e. strict inductive limit of an increasing sequence of Banach spaces), see \cite[Page 58]{schaefer1999topological}. Let $\sD'_p(X)$ be the dual space of $\sD^p(X)$.
		
		Let $\sC_p$ be the space of positive closed  bidimension $(p,p)$-currents. Since positive currents has order $0$ \cite[Proposition 1.14]{Demailly},  $\sC_p$  is  a  subset of  $\sD'_p(X)$.  In the following we discuss the metrizability of certain subset of the space of $\sC_p$. 
		
		Let $M_i>0, i\geq 1$ be a sequence of real numbers. Let $\sC_{p,M_i}$ be the  subset of $\sC_p$ containing elements $S$ such that $\Ma_{K_i} S\leq M_i$.  Then
		$$\tilde{\sC_p}:=\cap_{i=1}^\infty\; \sC_{p,M_i}$$ is pointwise bounded, i.e. for every $\Phi\in \sD^p(X)$, the set $$\left\{\left\langle S,\Phi\right\rangle: S\in \tilde{\sC_p}  \right\}$$
		is bounded in $\C$.  For a strict LB space $E$, a pointwise bounded subset in the dual space $E'$ is equicontinuous \cite[Page 83, Theorem 4.2]{schaefer1999topological}. Thus $\tilde{\sC_p}$ is equicontinuous.   For a separable topological vector space $E$, the weak-* topology on an equicontinuous subset of  $E'$ is metrizable \cite[Page 87 Theorem 4.7]{schaefer1999topological}.  Since $\sD^p(X)$ is separable, we get
		
		\begin{prop}\label{prop:metrizable}
			The weak-* topology on $\tilde{\sC_p}$ is metrizable. 
		\end{prop}

		\medskip
		
		We now focus our attention on analytic subvarieties with bounded volume on compact subsets, which are subsets of  $\tilde{\sC_p}$.
		\begin{defi}
			An  {\em effective analytic cycle $V$ of pure dimension $p$}  is a finite summation 
			$$V=\sum_W n_W  W,$$
			where $n_W\in \Z_{\geq 1}$ and each $W$ in the summation  is a closed irreducible analytic subvariety of $X$ of dimension $p$.
		\end{defi}
		
		We denote $\supp V:=\cup_W W$ the support of $V$, which is a (maybe reducible) closed  analytic subvariety of $X$ of pure dimension $p$.  We define $\vol (V):=\sum_W n_W \vol(W)$.   We denote $[V]$ the  current induced by $V$. 
		
		\medskip

		\begin{defi}
			Let $M_i>0, i\geq 1$ be a sequence of real numbers. A {\em bounded volume space} of dimension $p$, denoted by $\sV_p$,  is a collection of effective analytic cycles $V$ of pure dimension $p$ such that  for every $i\geq 1$,  $$\Ma_{K_i}(V)\leq M_i.$$
		\end{defi}

		Here we specifically emphasize that by our definition, the zero current also belongs to $\sV_p$.  
		\medskip

		%
		
		

		{\bf The topology on $\sV_{p}$:}  Since  $\sV_{p}$ can be naturally seen as a subset of   $\tilde{\sC_{p}}$, we let the topology on  $\sV_{p}$ be the  subspace topology induced by the weak-* topology on  $\tilde{\sC_{p}}$.     With this topology, $\sV_{p}$  is a metrizable space since  $\tilde{\sC_{p}}$ is a metrizable space.

		%
		
		
		Recall that the  Hausdorff topology on subsets of $X$  is defined by the following basis of neighborhoods: $\sU_{K,\ep}(A)$ consists of subsets $B$ such that the set $B\cap K$ is contained in the $\ep$-neighborhood of $A\cap K$, and the other way around, where $K\subset X$ is a compact subset.

		Let $K\subset X$ be any fixed  compact subset.  Let $\sV_{p,K}$ be the closed subset of $\sV_{p}$ containing elements $V$ such that $\supp V\cap K\neq \emptyset$.
		
		We have the following  Bishop's theorem \cite{bishop1964conditions}, see  \cite[Section 16.1, Proposition 1]{MR1111477} for details. Recall that on a metrizable space, compactness and sequential compactness are equivalent. 
		\begin{thm}\label{thm: Bishop}
			
			The space  $\sV_{p}$ is a compact  metrizable space.  For each compact subset $K\subset X$,  $\sV_{p,K}$ is compact. Moreover if  $V_n\to V$ in $\sV_{p}$, then $\supp(V_n)\to \supp V$ in the sense of  Hausdorff topology.
			
		\end{thm}
		
		It  is clear that  if $K_1\subset K_2\subset \cdots$ is an exhaustion by compact sets of $X$, then $\sV_{p,K_1}\subset \sV_{p,K_2}\cdots$ is an exhaustion by compact sets of $\sV_{p}$.

		\medskip
		
		\begin{defi}\label{defi:irreduciblespace}
			Let $\sV^*_{p,K}$ be the subset of $\sV_{p,K}$ containing all elements $V=\sum n_W W$ such that:
			
			\smallskip
			
			(1) there is only one irreducible component $W_0$ such that $W_0\cap K\neq \emptyset$;
			
			\smallskip
			(2) moreover $n_{W_0}=1$.  	
		\end{defi}

		By Theorem \ref{thm: Bishop}, $\sV^*_{p,K}$ is an open subset of $\sV_{p,K}$. 
		
		\medskip

		The spaces $\sV_p$, $\sV_{p,K}$ and $\sV^*_{p,K}$  are  locally compact Polish spaces (a Polish space is a separable completely metrizable space).  In this class of spaces, measure theory exhibits very favorable behavior in the following sense:
		
		\smallskip
		
		(1) On  a Polish space, a finite Borel measure is Radon \cite[Page 107, Theorem 17.11]{Kechris1995};
		
		\smallskip
		
		(2) On a  locally compact  Hausdorff space, the Riesz representation theorem holds.

		\medskip
		
		We will frequently make use of the two properties stated above throughout this article. For simplicity,  a positive measure on a topological space will always mean a Borel measure in this article.

		\begin{defi}[Uniformly woven current]\label{defi:uniformwoven}
			A positive closed bidimension $(p,p)$ current $S$ on $X$ is called {\em uniformly woven}  if there is a bounded volume space $\sV_p$  such that we can write
			
			$$S=\int_{\sV_{p}} [V_a] \;d\nu(a),$$
			where each $V_a$ is an element in $\sV_{p}$ and $\nu$ is a finite measure on $\sV_p$. 
		\end{defi}

		\medskip
		
		A quick application of the measure theory on $\sV^*_{p,K}$  is the following:
		
		\begin{prop}\label{prop:irreducible}
			
			Let $S$ be a bidimension $(p,p)$   uniformly woven current on $X$,
			$$S=\int_{\sV_{p}} [V_a] \;d\nu(a).$$ 
			Let $K\subset X$ be a compact subset. Then there exist  a finite measure $\mu$ on $\sV^*_{p,K}$ such that for every continuous  test form $\Phi$ satisfying  $\supp \Phi\subset K$, we have 
			$$\left\langle S,\Phi\right\rangle =\int_{\sV^*_{p,K}}\left\langle  [V_a],\Phi\right\rangle  \;d\mu(a).$$
			
		\end{prop}
		\begin{proof}
			
			Let $C_c(\sV^*_{p,K})$ be the space of  continuous functions on $\sV^*_{p,K}$ with compact support.  Let $\chi\in C_c(\sV^*_{p,K})$. Then $\chi$ naturally induces  a   function  $\psi_\chi$ on $\sV_{p}$: for each $V\in \sV_{p}$, there is a unique way to write $$V=\sum_{i\in I}  n_i W_i+\sum_{j\in J}  n_j Z_j$$ as a finite summation of irreducible subvarieties, where each $W_i$ satisfies $W_i\cap K\neq \emptyset$, and each $Z_j$ satisfies $Z_j\cap K= \emptyset$. Define
			\begin{equation}\label{eqn:psi}
				\psi_\chi(V):=\sum_{i\in I}  n_i \chi(W_i). 
			\end{equation}  
			
			Let $U\Subset X$ be a relatively compact open subset of $X$ such that $K\subset U$. By Lelong's theorem \cite[Page 189, Proposition 1]{MR1111477},  there exists $c_U>0$ such that every irreducible closed analytic subvariety $W$ in $U$ satisfying $W\cap K\neq \emptyset$ has $\Ma_U(W)\geq c_U$. On the other hand, by the definition of the bounded volume space $\sV_p$, there exists $M_U<+\infty$ such that $\Ma_U(V)\leq M_U$ for every $V\in\sV_p$. Hence in the summation (\ref{eqn:psi}), we have $\sum_{i\in I} n_i\leq M_U/c_U$, so $\sum_{i\in I} n_i$  is uniformly bounded from above. This implies $\psi_\chi$ is a bounded  function. We denote  $\eta_K(V)$  by  $\eta_K(V):=\sum_{i\in I} n_i$, which is a non-negative integer.

			We claim that $\psi_\chi$  is a Borel measurable function.  For each $n\geq 0$, let $\sV^{(n)}_{p,K}$ be the subset of $\sV_{p}$ containing all elements $V$ such that $\eta_K(V)=n$.   Then each $\sV^{(n)}_{p,K}$ is a Borel measurable subset, and we can write $\sV_p$ as a disjoint union of Borel measurable subsets as 
			$$\sV_p=\sqcup_{n=0}^\infty \sV^{(n)}_{p,K}.$$
			In fact by Lelong's theorem,  $\sV^{(n)}_{p,K}=\emptyset$ when $n$ is large enough. To show $\psi_\chi$  is a Borel measurable function, it suffices to show $\psi_\chi$  restricted on each $\sV^{(n)}_{p,K}$  is  continuous.  Let $V_m\in \sV^{(n)}_{p,K}$, $m\geq 0$ be a sequence such that $V_m\to \tilde{V}\in \sV^{(n)}_{p,K}$.  Let $\tilde{W}$ be an irreducible component of $\tilde{V}$ such that $\tilde{W}\cap K\neq \emptyset$, let $n_{\tilde{W}}$ be its multiplicity.  Let $\Omega$ be a sufficiently small neighborhood of $\tilde{W}$ in $X$.  Let $F_m$ be the set  containing all irreducible components $W$ of $V_m$ such that $W\subset \Omega$.   Let $F'_m\subset F_m$ be the subset containing elements $W$ such that $W\cap K\neq \emptyset$. Let $F''_m:=F_m\setminus F'_m$.  For $m$ large enough, since  multiplicity does not decrease under limits, we have 
			$$\sum_{W\in F'_m} n_W\leq \sum_{W\in F_m} n_W \leq n_{\tilde{W}},$$
			where $n_W$ is the multiplicity of $W$ in $V_m$.
			
			On the other hand, since all $V_m$ and $\tilde{V}$ are in $\sV^{(n)}_{p,K}$,  for $m$ large enough we have
			$$\sum_{\tilde{W}}\sum_{W\in F'_m} n_W= \sum_{\tilde{W}}n_{\tilde{W}}=n,$$
			and
			$$V_m=\sum_{\tilde{W}}\sum_{W\in F'_m} n_W W,$$
			where the summation is taken over all irreducible component of $\tilde{W}$ of $\tilde{V}$ such that $\tilde{W}\cap K\neq \emptyset$.
			
			As a consequence, we have 
			$$\sum_{W\in F'_m} n_W= n_{\tilde{W}}.$$
			Thus for any sequence $W_m$ such that $W_m\in F'_m$, we have 
			$$W_m\to \tilde{W}$$
			in the space $\sV^*_{p,K}$.  This implies $\psi_\chi(V_m)\to \psi_\chi(\tilde{V})$. Hence $\psi_\chi$ is continuous on  $\sV^{(n)}_{p,K}$. This implies $\psi_\chi$  is Borel measurable on $\sV_p$. 
			

			\smallskip
			Since $\psi_\chi$  is a bounded Borel measurable function, $\psi_\chi$ is integrable with respect to $\nu$, i.e.  $\int |\psi_\chi |\;d\nu<+\infty$.  The map $\chi\mapsto \int \psi_\chi \;d\nu$  defines a positive linear functional on $C_c(\sV^*_{p,K})$. Since $\sV^*_{p,K}$   is a locally compact Hausdorff space, by Riesz representation theorem, there exists a unique Radon measure $\mu$ on $\sV^*_{p,K}$  such  that  for every $\chi\in C_c(\sV^*_{p,K})$,  
			\begin{equation}\label{eqn:2.1}
				\int \chi \;d\mu= \int \psi_\chi \;d\nu. 
			\end{equation}
			
			We first show that $\mu$ is finite. Indeed, the integer-valued function
			$$\eta_K(V):=\sum_{i\in I} n_i$$
			is uniformly bounded on $\sV_p$, by the Lelong lower bound for the volume of each irreducible component meeting $K$ and by the uniform volume upper bound in the bounded volume space. If $\chi\in C_c(\sV^*_{p,K})$ satisfies $0\leq \chi\leq 1$, then $0\leq \psi_\chi\leq \eta_K$. Hence
			$$\int_{\sV^*_{p,K}}\chi\,d\mu=\int_{\sV_p}\psi_\chi\,d\nu\leq \left(\sup_{\sV_p}\eta_K\right)\nu(\sV_p)<+\infty.$$
			Since $\sV^*_{p,K}$ is locally compact and $\sigma$-compact, we can select a monotonically increasing sequence of continuous functions with compact support $\chi_n$ such that $$\mu(\sV^*_{p,K})=\lim_{n\to+\infty}\int_{\sV^*_{p,K}}\chi_n\,d\mu,$$ which implies that $\mu(\sV^*_{p,K})<+\infty$.
			
			We now extend (\ref{eqn:2.1}) from compactly supported continuous functions to bounded continuous functions on $\sV^*_{p,K}$. Let $\chi$ be such a function. Choose continuous functions $\rho_m\in C_c(\sV^*_{p,K})$ such that $0\leq \rho_m\leq 1$ and $\rho_m\to 1$ pointwise. Applying (\ref{eqn:2.1}) to $\rho_m\chi$ and letting $m\to+\infty$, the dominated convergence theorem gives
			$$\int_{\sV^*_{p,K}}\chi\,d\mu=\int_{\sV_p}\psi_\chi\,d\nu,$$
			where on the right hand side we use the same definition of $\psi_\chi$ as in (\ref{eqn:psi}); the domination follows from $|\psi_{\rho_m\chi}|\leq \|\chi\|_\infty\eta_K$.
			
			Let $\Phi$ be a bidegree $(p,p)$ continuous test form on $X$ such that $\supp\Phi\subset K$. The function
			$$\chi_\Phi(W):=\left\langle [W],\Phi\right\rangle,\qquad W\in \sV^*_{p,K},$$
			is bounded and continuous. Since $\Phi$ is supported in $K$, for every $V\in\sV_p$ we have
			$$\psi_{\chi_\Phi}(V)=\left\langle [V],\Phi\right\rangle.$$
			Therefore
			$$\left\langle S,\Phi\right\rangle=\int_{\sV_p}\left\langle [V_b],\Phi\right\rangle\,d\nu(b)=\int_{\sV^*_{p,K}}\left\langle [V_a],\Phi\right\rangle\,d\mu(a).$$
			This completes the proof.
			
		\end{proof}
		%

		\begin{rmk}
			In Proposition \ref{prop:irreducible}, if we let $K=\overline{\Omega}$, where $\Omega\Subset X$ is a relatively compact open subset, then the restriction $S|_\Omega$ satisfies 	$$S|_\Omega=\int_{\sV^*_{p,K}} [V_a]  \;d\mu(a).$$
		\end{rmk}


		Let $S$ be a positive closed $(p,p)$-current on $X$. The {\em trace measure} of $S$ is by definition the measure $\sigma_S:=S\wedge \omega^{N-p}$. Let $R$ be a positive closed $(1,1)$-current on $X$. On a local chart, we can write $R=dd^c u$, where $u$ is a p.s.h. function. If $u\in L^1_{loc}(\sigma_S)$, then we define $dd^c u\wedge S:=dd^c(uS)$. If on each  local chart, the local potential of $R$ satisfies  the integrability condition, then we can globally define $R \wedge S$.
		
		\begin{defi}\label{def:admissible}
			In this case we say that $R\wedge S$ is {\em admissible}.	
		\end{defi} 
		If $R$ has bounded local potential, then $R\wedge S$ is always admissible.

		\medskip
		Let $\sM(\sV_{p})$ be the space of all   positive measures on $\sV_{p}$ with mass smaller or equal to one. Recall that the weak-* topology on  $\sM(\sV_{p})$ is defined by the weakest topology such that for every continuous function $\phi$ on $\sV_{p}$ with compact support, the map $\nu\mapsto \int \phi \;d\nu$ is continuous. We have the following result.
		
		\begin{prop}\label{prop: limit}
			Let $X$ be a  Hermitian manifold of dimension $N$.  Let $\nu_n$ be a sequence of  measures in $\sM(\sV_{p})$  such that $\nu_n\to \nu$ in the weak-*  sense. Let $S_n$ and $S$ be the corresponding uniformly woven currents.  Let $T$ be a positive closed $(1,1)$-current on $X$ with continuous local potential. Then $S_n\wedge T\to S\wedge T$ in the sense of currents.
		\end{prop}
		\begin{proof}
			We need to show for every real smooth $(p-1,p-1)$-form $\Omega$ with compact support, we have $\langle S_n\wedge T, \Omega\rangle \to \langle S\wedge T, \Omega\rangle$. Since  $$\langle S\wedge T, \Omega\rangle=\int_{\sV_{p}} \langle [V_a]\cap T, \Omega\rangle\; d\nu(a),$$
			and by the similar formula for $\langle S_n\wedge T, \Omega\rangle$, we only need to show that the function $I_{T,\Omega}: \sV_{p}\to \R$, $V\mapsto  \langle [V]\wedge T, \Omega\rangle$ is continuous with compact support.  First we show $I_{T,\Omega}$ has compact support. Let $\sU\subset \sV_{p}$ be the subset of elements  $V$ such that $\supp V\cap \supp \Omega\neq \emptyset$.   Since $\supp \Omega$ is compact, by Theorem \ref{thm: Bishop}, $\sU$ is compact.  It is clear that for $v\notin \sU$, $I_{T,\Omega}(V)=0$. This implies that $I_{T,\Omega}$ has compact support.  It remains to show $I_{T,\Omega}$  is continuous. Let $V_n\to V$ in $\sV_{p}$.  The problem is local, so we assume $T=dd^c u$, where $u$ is a continuous p.s.h. function. Then $$I_{T,\Omega}(V_n)=\langle [V_n]\wedge dd^c u, \Omega\rangle=\langle dd^c (u[V_n]) , \Omega\rangle=\langle [V_n], u\;dd^c\Omega\rangle.$$
			Similarly $I_{T,\Omega}(V)=\langle [V], u\;dd^c\Omega\rangle.$ Since $u$ is continuous and $V_n\to V$ in $\sV_{p}$, we have $I_{T,\Omega}(V_n)\to I_{T,\Omega}(V)$, hence $I_{T,\Omega}$  is continuous, which finishes the proof.
		\end{proof}
		\medskip
		
		The following result holds for the intersections of integral of positive closed currents, in particular it holds for  the intersections of uniformly woven currents. The case $N=2$ was discussed in \cite[Lemma 2.7]{diller2011dynamics}.
		
		\begin{prop}\label{prop: uniform intersection}
			Let $(X,\omega)$ be a Hermitian manifold of dimension $N$, $N\geq 2$.  Let $S$ be a positive closed $(p,p)$-current  on $X$ and let $R$ be a positive closed $(1,1)$-current on $X$ such that $S\wedge R$ is admissible.  Assume that $S$ and $R$ admit decompositions $S=\int S_a\;d\nu(a)$ and $R=\int R_b\;d\la(b)$. Here $\nu$ (resp. $\la$) is a  positive measure on the space of positive closed $(p,p)$-currents  (resp. $(1,1)$-currents) on $X$, and each $S_a$   (resp. $R_b$) is a  positive closed $(p,p)$-current  (resp. $(1,1)$-current) on $X$. Then for $\nu\tensor\la$-a.e. $(a,b)$, $S_a\wedge R_b$ is admissible, and we have
			
			\begin{equation}\label{3.1}
				S\wedge R=\int (S_a\wedge R_b) \; d(\nu\tensor\la)(a,b).
			\end{equation}
		\end{prop}
		\begin{proof}
			The result is local, so we can assume that $S$ and $R$ are positive closed currents on the unit ball $B\subset \C^N$ with finite mass.
			
			
			We first show that for $\nu$-a.e. $a$, $S_a\wedge R$ is admissible. Let $u$ be a potential of $R$, i.e. $dd^c u=R$ on $B$. Let $B'\subset \subset B$ be a smaller ball.  By our assumption, $\int_{B'} |u| \;d\sigma_S<+\infty$, where $\sigma_S$ is the trace measure. Since $\sigma_S$ admits a decomposition $\sigma_S=\int \sigma_{S_a} \;d\nu(a)$, for $\nu$-a.e. $a$, we have $\int_{B'} |u| \;d\sigma_{S_a}<+\infty$. This implies that  for $\nu$-a.e. $a$, $S_a\wedge R$ is admissible. 
			
			Let $B'\subset\subset  B$ be a smaller ball compactly contained in $B$. By a classical method of constructing p.s.h. functions via integral kernels \cite[Section 2]{messaoud2000operateur} (see also  \cite[Lemma 2.8]{diller2011dynamics}, \cite[Theorem 2.7]{dinh2008dynamics}), there is a constant $C=C(B')>0$ such that for every positive closed $(1,1)$-current $T$ on $B$ with finite mass, there is a canonical p.s.h. function $u$ on $B'$ such that:
			\begin{points}
				\item $dd^c u=T$ on $B'$;
				\item $u\leq 0$ on $B'$;
				\item $||u||_{L^1(B')}\leq C\Ma(T)$,
			\end{points}
			where $\Ma(T)$ is the mass of $\sigma_T$. We call $u$ the canonical potential of $T$ on $B'$. 
			
			Next we show that  if $S_a\wedge R$ is admissible, then for $\la$-a.e. $b$, $S_a\wedge R_b$ is admissible. Let $u_b$ be the  canonical potential of $R_b$ on $B'$.  Let $u:=\int u_b\;d\la(b)$.  Then $u$  is a well-defined non-positive p.s.h. function on $B'$ such that $dd^c u=R$ on $B'$. Since by our assumption $\int -u\;d\sigma_{S_a}<+\infty$, for $\la$-a.e. $b$, we have $\int -u_b\;d\sigma_{S_a}<+\infty$.  This implies that  for $\la$-a.e. $b$, $S_a\wedge R_b$ is admissible. 
			
			The above discussion  implies that  for $\nu\tensor\la$-a.e. $(a,b)$, $S_a\wedge R_b$ is admissible. 
			
			Finally we show the integral formula (\ref{3.1}) holds. On $B'$ we have 
			
			$$u S=\int (u_b S_a) \; d(\nu\tensor\la)(a,b).$$
			Apply $dd^c$ to the both sides we have 
			$$S\wedge R=\int (S_a\wedge R_b) \; d(\nu\tensor\la)(a,b)$$ on $B'$. Since this holds on each ball $B'\Subset B$, we get the desired formula (\ref{3.1}) holds on $B$, which  finishes the proof.
		\end{proof}
		
		\medskip

		\subsection{Uniformly woven currents on a complex analytic space}
		
		In this subsection we discuss uniformly woven currents on a (maybe singular) complex analytic space.  
		
		A complex analytic space is as defined in \cite[Definition 5.2]{Demailly}. Examples include all quasi-projective varieties over $\C$.

		Let $X$ be an irreducible complex analytic space of dimension $N$. Smooth functions and forms on $X$ are understood on $X^\reg$ and are required, near every singular point and after a local embedding $j:X\xhookrightarrow[\loc.]{}\C^l$, to be restrictions of smooth functions and forms on an open subset of $\C^l$. A {\em Hermitian form} on $X$ is defined similarly: it is a Hermitian $(1,1)$-form on $X^\reg$ which, under such a local embedding, is the restriction of a Hermitian form on an ambient open subset. A complex analytic space together with a Hermitian form is called a {\em Hermitian analytic space}. These notions are independent of the choice of local embedding; see \cite[Section 1]{MR813252}. The volume of subvarieties in $X$ is computed with respect to this Hermitian form.
		
		Currents and the operators $d,d^c$ and $dd^c$ on $X$ are defined by duality; we refer to \cite{MR813252} for the standard background. A function $u:X\to \R\cup\{-\infty\}$ is p.s.h. if locally, under a local embedding into some $\C^l$, it is the restriction of a p.s.h. function on the ambient space. A positive closed $(1,1)$-current $R$ on $X$ has {\em continuous} (resp. {\em bounded}) local potential if locally $R=dd^c u$ for a continuous (resp. bounded) p.s.h. function $u$ on $X$. Let $S$ be a positive closed bidimension $(p,p)$-current on $X$, $1\leq p\leq N-1$. If $u$ is locally integrable with respect to the trace measure of $S$, we define
		$$S\wedge R:=dd^c(uS),$$
		which is independent of the choice of local potential $u$ and is a positive closed bidimension $(p-1,p-1)$ current.

		Let $\sO_X$ be the ring of holomorphic functions on $X$.  We call $X$ {\em Stein}  if $X$ can be embedded into a Stein manifold as a closed analytic subvariety. Assume $X$ is Stein.   An open subset $D\subset X$ is called {\em Runge} if $\sO_X$ is dense in $\sO_U$ with respect to the compact-open topology. 
		
		Based on Coltoiu's \cite{coltoiu1991traces} construction of the extension of p.s.h. function, we show that for a bounded p.s.h. function on $X$,  locally there always exists an extension that remains a bounded p.s.h.  function.
		
		\begin{prop}\label{prop:extendpsh}
			Let $u$ be a bounded p.s.h. function on $X$. Then for every $x\in X$, there exists a local embedding $j:X\xhookrightarrow[\loc.]{}U$ near $x$, where $U$ is an open subset of $\C^l$, and a bounded p.s.h. function $\tilde{u}$ on $U$, such that $\tilde{u}|_{U\cap X}=u.$  
		\end{prop}

		\begin{proof}
			We construct $\tilde{u}$ by the method of Coltoiu's \cite{coltoiu1991traces}. Without loss of generality we may assume that $X$ is a closed analytic subvariety in $\D^l\subset \C^l$ with $x=0$, where $\D$ is the unit disk in $\C$. Then $X$ is a Stein space since $\D^l$ is Stein.  By adding a constant on $u$, we may further assume  that $u>0$ on $X$.  Consider the following open subset $B\subset X\times \D$: 
			$$B:=\left\{(z,w)\in X\times \D: \log|w|+u(z)<0\right\},$$
			Since $ \log|w|+u(z)$ is a p.s.h. function on $X\times \D$, $B$ is Runge in $X\times \D$. 
			
			
			By \cite[Theorem 3]{coltoiu1991traces}, there exists a Runge open subset $\Omega\subset \D^l\times \D$ with $\Omega\cap (X\times \D)=B$.  For any point $(z,w)\in \Omega$, consider the distance $\delta_\Omega(z,w)$ to the boundary $\partial \Omega$ along the $w$-direction. Since $\Omega$ is Runge in $\D^l\times \D$ (hence Stein), the function $-\log \delta_\Omega$ is p.s.h. on $\Omega$ \cite{hormander1990introduction}.   For any $z\in \D^l$, set 
			$$\tilde{u}(z):=-\log \delta_\Omega(z,0).$$
			It is clear that $(z,0)\in \Omega$.  Then $\tilde{u}$ is a p.s.h. function on $\D^l$.     Since $\Omega\cap (X\times \D)=B$, from the definition of $\tilde{u}$ we get $\tilde{u}=u$ on $X$.  Since $\delta_\Omega(z,0)< 1$ for every $z\in \D^l$, $\tilde{u}>0$ on $\D^l$.  Since a p.s.h. function is bounded above on compact subsets,   $\tilde{u}$ is a bounded p.s.h. function on $\frac{1}{2}\D^l$.  This completes the proof by setting $U:=\frac{1}{2}\D^l$. 
		\end{proof}

		\medskip
		
		On a  Hermitian analytic space $X$, the bounded volume space $\sV_p$, $\sV_{p,K}$ and $\sV^*_{p,K}$ are defined as in Section \ref{sec:basicwoven}, where we replace the Hermitian manifold $X$ by the Hermitian analytic space $X$.  The definition of uniformly woven currents on $X$  is also defined following the same principle, see Definition \ref{defi:uniformwoven} .

		\medskip
		\subsection{Pullback of uniformly woven currents by a  holomorphic finite quotient map}
		
		In general, it is a delicate matter to define the pullback operater acting on positive closed currents under a holomorphic map between complex analytic spaces, especially in the presence of singularities. We will provide a well-defined notion of pullback operater for uniformly woven currents  in the case of a holomorphic finite quotient map.  
		
		On the other hand, the pushforward of current by a proper (i.e. the preimage of a compact subset is compact) holomorphic  map is always well-defined. In fact, via duality, it suffices to define the pullback of a smooth form with compact support by a proper holomorphic map, which is always a smooth form with compact support.
		
		\medskip

		%
		%
		%
		%
		%
		%
		%
		%
		%
		%
		%

		%

		Let $(Y,\omega_Y)$ be a Hermitian manifold of dimension $N$. Denote $\Aut(Y)$ the group of holomorphic automorphisms of $Y$.  Let $G$ be a finite subgroup of $\Aut(Y)$.  Let $$Z:=Y/G$$be the quotient space, which is an irreducible normal complex analytic space of dimension $N$, see \cite{MR84174}.  Let $\pi:Y\to Z$ be the quotient map.  Then $\pi$ is a finite map (i.e. proper with finite fibers). The topological degree $\deg \pi$ is equal to $|G|$. Let $\omega_Z$ be a Hermitian form on $Z$ so that $(Z,\omega_Z)$ becomes a Hermitian analytic space. 
		
		Let $0\leq p\leq N$.  Let $V$ be a $p$-dimensional closed  irreducible analytic subvariety of $Z$. Thus $\pi^{-1}(V)$ is a maybe reducible $p$-dimensional closed analytic subvariety of $Y$.  Assume $\pi^{-1}(V)=\sum_{a\in F} W_a$, where $F$ is a finite set and each $W_a$ is a $p$-dimensional closed  irreducible analytic subvariety of $Y$.  Thus $G$ acts transitively on $F$.   Let $G_a$ be the stabilizer of $W_a$, which is a finite subgroup of $\Aut(W_a)$.
		
		\begin{lem}\label{lem:stabilizer}
			Let $W$ be an irreducible complex analytic space, let $G$ be a finite subgroup of the holomorphic automorphism group $\Aut(W)$.  For each $x\in W$, let $G_x$ be the stabilizer of $x$.  Then there exists a unique subgroup $H$ of $G$ such that
			$$\left\{x\in W: G_x=H\right\}$$
			is a Zariski open subset. 
		\end{lem}
		\begin{proof}
			The uniqueness is clear since $W$ is irreducible. We will show that  such subgroup exists. Let 
			$$H:=\left\{g\in G: g(x)=x\;\;\text{for every}\;x\in W\right\}.$$
			Then $H$ is a subgroup of $G$.  For each $g\notin H$, the set
			$$\Fix (g):=\left\{x\in W: g(x)=x\right\}$$
			is a closed proper analytic subvariety of $W$.  Then we have 
			$$\left\{x\in W: G_x=H\right\}=W\setminus \cup_{g\notin H} \Fix(g)$$
			is Zariski open in $W$. 
		\end{proof}
		
		\medskip

		Come back to our setting.  The finite group $G_a$ acts holomorphically on $W_a$.  Let $H_a$ be the subgroup of $G_a$ given in Lemma \ref{lem:stabilizer}.  Since $G$ acts  transitively on $F$, all $H_a$ are conjugate in $G$,  therefore $H_a$ has constant cardinality.    Let $\eta_V:=|H_a|$. We may thus define:
		\begin{defi}[Pullback of a closed  analytic subvariety] \label{defi:pullbacksubvariety}
			The pullback $\pi^* V$ of a $p$-dimensional closed  irreducible analytic subvariety of $Z$ is an  effective analytic cycle of pure dimension $p$, such that 
			$$\pi^*V:=\sum_{a\in F} \eta_V W_a.$$
			
			If $V=\sum_W n_W$ is an effective cycle of pure dimension $p$ in $Z$, we define
			$$\pi^*V:=\sum_W n_W \pi^*(W),$$ which is  an effective cycle of pure dimension $p$ in $Y$.
			
		\end{defi}

		\medskip
		The following proposition is crucial.
		\begin{prop}\label{prop:pullbackunique}
			Let $V$ be an effective analytic cycle of pure dimension $p$ in $Z$.  Then $\pi^*V$ is the unique effective analytic cycle $R$  of pure dimension $p$ such that:
			\smallskip
			
			(1) $\pi_*([R])=\deg \pi \;[V]$;
			\smallskip
			
			(2) $g_*([R])=[R]$ for every $g\in G$.
		\end{prop}
		\begin{proof}
			Let $V=\sum_{i=1}^n n_i V_i$, where each $V_i$ is a $p$-dimensional closed  irreducible analytic subvariety of $Z$.   Let $R$ be an effective analytic cycle $R$  of pure dimension $p$ satisfying the above two properties. Then we can decompose $R=\sum_{i=1}^n R_i$, where each $R_i$  is an  effective analytic cycle of pure dimension $p$ such that $\supp R_i=\pi^{-1}(V_i)$.   Then each $R_i$ still satisfies the above two properties.  Thus without loss of generality, we may assume $V$ is a $p$-dimensional closed  irreducible analytic subvariety of $Z$. 
			
			Let $R$ be an effective analytic cycle $R$  of pure dimension $p$ satisfying the above two properties. Let $F$ be the set of all irreducible components of $\pi^{-1}(V)$. Property (1) implies that we can write $R=\sum_{a\in F} \phi_a W_a$, where $\phi_a\in \Z_{\geq 0}$.  For each $a\in F$, 
			$$\pi_*([W_a])=\deg(\pi|_{W_a})\; [V],$$
			where $\deg(\pi|_{W_a})$ is the topological degree of $\pi|_{W_a}$, which is equal to the length of the $G_a$-orbit of a generic point $x\in W_a$.     Since $G$ acts   transitively on $F$, all $G_a$ are conjugate in $G$,  therefore $G_a$ has constant cardinality. Assume $L_V:=|G_a|\in \Z_{\geq 1}$.  By the orbit–stabilizer theorem, 
			$$\deg(\pi|_{W_a})=\frac{|G_a|}{|H_a|}=\frac{L_V}{\eta_V},$$
			where $H_a,\eta_V$ are given in Lemma \ref{lem:stabilizer}.
			
			Again since $G$ acts   transitively on $F$, property (2) implies that $\phi_a$ is a constant.  Assume $\phi_a\equiv \phi\in \Z_{\geq 1}$.  Thus by  property (1) we have
			$$\pi_*([R])=\sum_{a\in F} \phi \deg(\pi|_{W_a})\; [V]=|G| \; [V],$$
			which is equivalent to 
			\begin{equation}\label{eqn:property1}
				\phi\cdot |F|\cdot \frac{L_V }{\eta_V} =|G|.
			\end{equation}

			Since $G$ acts  transitively on $F$,   By the orbit–stabilizer theorem we have $|G|=L_V|F|$, thus (\ref{eqn:property1}) is equivalent to $\phi=\eta_V$. Thus we must have $R=\sum_{a\in F} \eta_V W_a=\pi^*V$,  which finishes the proof. 
		\end{proof}

		%

		\medskip

		If  $\dim V=0$, one can pullback a point. We can thus define pushforward of a function via duality:
		
		\begin{defi}[Pushforward of a function]\label{defi:pushforwardfunction}
			The pushforward of a function $\phi:Y\to \R$ is a function  $\pi_* \phi:Z\to \R$  given by
			
			$$ \pi_* \phi(x):=\sum_{y\in \pi^{-1}(x)} \eta_y  \phi(y),$$
			where $\pi^*x=\sum_{y\in \pi^{-1}(x)}\eta_y\, y$ is the effective zero-cycle given by Definition \ref{defi:pullbacksubvariety}.
		\end{defi}
		
		We have the following basic property of pushforward of a function.
		\begin{prop}\label{prop:pushforwardcontinuous}
			If $\phi:Y\to \R$ is a continuous function, then $\pi_*\phi$ is a continuous function on $Z$. 
		\end{prop}
		\begin{proof}
			Let $z_n\to z$ in $Z$. We need to show $\pi_*\phi(z_n)\to  \pi_*\phi(z)$.  By definition, $\pi_*\phi(z_n)=\left\langle [\pi^*{z_n}], \phi \right\rangle$ and  $\pi_*\phi(z)=\left\langle [\pi^*z], \phi \right\rangle$, where $[\pi^*{z_n}]$ (resp. $[\pi^*z]$) are associated  finite positive measures supported on $\pi^{-1}(z_n)$ (resp. $\pi^{-1}(z)$).  We can multiply $\phi$ by a cut off function supported on a compact neighborhood of 
			$\pi^{-1}(z)$ to make $\phi$ compactly supported, without affecting the value of $\left\langle [\pi^*{z_n}], \phi \right\rangle$ and $\left\langle [\pi^*z], \phi \right\rangle$.  Thus it suffices to show $[\pi^*{z_n}]\to [\pi^*z]$ in the weak-* sense.  Assume by contradiction that $[\pi^*{z_n}]\not\to [\pi^*z]$. Since the masses of $\pi^*{z_n}$ are constant ($=|G|$),  by passing to a subsequence we may assume  $\pi^*{z_n}\to \mu\neq [\pi^*z]$ in the weak-* sense.

			For each $n\geq 1$, we have $\pi_*([\pi^*{z_n}])=\deg\pi\; [z_n]$, and $g_*([\pi^*{z_n}])=[\pi^*{z_n}]$ for every $g\in G$.   Since $\pi^*{z_n}\to \mu\neq [\pi^*z]$ in the weak-* sense,   $\mu$ is a finite sum of Dirac masses satisfying $\pi_*(\mu)=\deg\pi\; \delta_z$, and $g_*\mu=\mu$ for every $g\in G$.  By Proposition \ref{prop:pullbackunique}, we have $\mu=[\pi^*{z}]$, which is a contradiction.  Thus $\pi_*\phi$ is a continuous function on $Z$. 
		\end{proof}
		\medskip

		Apply Chern-Levine-Nirenberg inequality and Proposition \ref{prop:extendpsh}, we get:
		
		\begin{lem}\label{lem:volumeestimate}
			Let $T$ be a positive closed $(1,1)$-current on $Z$ with bounded local potential.   Let $U\Subset Z$ be an open subset. Let $K\subset U$ be a compact subset.  Then there exists a constant $C=C(U,K,T)>0$ such that  for every positive closed  bidimension $(p,p)$ current $S$ on $Z$, we have
			$$\Ma_K(S\wedge T^p)\leq C\;\Ma_U(S\wedge \omega_Z^p).$$
		\end{lem}
		
		\begin{proof}
			The problem is local, so we may assume that $Z$ is embedded in an open subset $ \Omega \subset \C^l$ as a closed analytic subvariety and $T=dd^c u$, where $u$ is a p.s.h. function on $\Omega$. We may further assume  that there exists a Hermitian form $\omega$ on $\Omega $ with $\omega|_Z=\omega_Z$.  By Proposition \ref{prop:extendpsh} we may further assume that  $u$ is bounded in $\Omega$.    For a compact subset $K\subset \Omega$,   by Chern-Levine-Nirenberg inequality \cite[Page 146]{Demailly}, there is a constant $C=C(\Omega, K)>0$ such that
			\begin{align*}
				\Ma_K(S\wedge (dd^c u)^p)\leq C||u||^p_{\infty, U} \Ma_U(S\wedge \omega^p).
			\end{align*}
			This completes the proof.
		\end{proof}
		\medskip
		
		Let $L_1\subset L_2\subset \cdots$ be an exhaustion by compact sets of $Z$, i.e. $L_i\subset L_{i+1}^\circ$ andf $\cup_{i=1}^\infty L_i=Z$.   Define $K_i:= \pi^{-1}(L_i)$. Thus $K_1\subset K_2\subset \cdots$ is an exhaustion by compact sets of $Y$. We fix these  two exhaustions by compact sets.    Bounded volume space are defined with respect to those exhaustions by compact sets.   The following Proposition is important. 
		
		\begin{prop}\label{prop:continuouspullback}
			
			Assume in addition that $(Y,\omega_Y)$ is K\"ahler.  Let $\sV_{p,Z}$ be a $p$-dimensional bounded volume space of $Z$.  Then there exists a $p$-dimensional bounded volume space $\sV_{p,Y}$ of $Y$ such that $\pi^*V\in \sV_{p,Y}$ for every $V\in \sV_{p,Z}$. Moreover the pullback operator
			$$\Pi: \sV_{p,Z}\to \sV_{p,Y}, \;V\mapsto \pi^*V$$
			is continuous. 
		\end{prop}
		\begin{proof}
			
			Let $T:=\pi_*(\omega_Y)$. Then $T$ is a positive closed $(1,1)$-current on $Z$.  Let $u$ be a local potential of $\omega_Y$.  By Proposition \ref{prop:pushforwardcontinuous}, $\pi_*u$ is continuous.  By \cite[Corollaire 1.12 and Proposition 1.13]{MR813252},  $\pi_*u$ is a p.s.h. function on an open subset of $Z$ and $\pi_*u$ is a local potential of $T$, i.e. $dd^c (\pi_* u)=T$ locally.  Thus $T$ has continuous local potential.

			Let $M'_i>0$ in the definition of $\sV_{p,Z}$ such that for every element $V\in \sV_{p,Z}$,  $\Ma_{L_i} V\leq M'_i$.   Since $T$ has continuous local potential, its local potentials are bounded on every relatively compact open subset. Hence Lemma  \ref{lem:volumeestimate} applies on $L_{i+1}$, and for each $i\geq 1$ there exists a constant $C_i>0$ such that  for every $V\in \sV_{p,Z}$, 
			\begin{equation}\label{eqn:2.5}
				\Ma_{K_i} ( [\pi^*V]\wedge \omega_Y^p)=\Ma_{L_i} ( [V]\wedge T^p)\leq C_i \Ma_{L_{i+1}} ( [V]\wedge \omega_Z^p)\leq C_iM'_{i+1}.
			\end{equation}
			Let $M_i:=C_iM'_{i+1}$.  We define $\sV_{p,Y}$  as the set of all effective analytic cycles $V$ of pure dimension $p$  such that $\Ma_{K_i} V\leq M_i$.  Then by (\ref{eqn:2.5}),  $\pi^*V\in \sV_{p,Y}$ for every $V\in \sV_{p,Z}$. 
			
			It remains to show that the pullback operator $\Pi:  \sV_{p,Z}\to \sV_{p,Y}$ is continuous.  Since $\sV_{p,Y}$ and  $\sV_{p,Z}$ are both metric spaces, we need to show that   for any sequence $V_n$ in $\sV_{p,Z}$ such that $V_n\to V\in \sV_{p,Z}$,   $\pi^*(V_n)\to \pi^*V$ in $\sV_{p,Y}$.  By Theorem \ref{thm: Bishop}, $\sV_{p,Y}$ is compact. Let $R$ be a subsequential limit of $\pi^*(V_n)$ in $\sV_{p,Y}$. It suffices to show that $R=\pi^*V$.  By Proposition \ref{prop:pullbackunique}, for each $n\geq 1$, 
			
			(1) $\pi_*([\pi^*(V_n)])=\deg\pi\;[V_n]$;
			
			\smallskip
			
			(2) $g_*([\pi^*(V_n)])=[\pi^*(V_n)]$ for every $g\in G$.
			
			\smallskip
			
			Since the pushforward operater $\pi_*$  and $g_*$ are continuous, we get:
			\smallskip
			
			(i) $\pi_*([R])=\deg \pi \;[V]$;
			\smallskip
			
			(ii) $g_*([R])=[R]$ for every $g\in G$.
			
			\smallskip

			Then by the uniqueness part of Proposition \ref{prop:pullbackunique}, we have $R=\pi^*V$. This completes the proof. 
		\end{proof}

		\medskip
		
		For a finite measure $\mu$ on $\sV_{p,Z}$, the pushforward $\Pi_*\mu$ is a finite measure on $\sV_{p,Y}$ since  $\Pi$ is continuous (Proposition \ref{prop:continuouspullback}).   Thus any woven representation of a uniformly woven current on $Z$ admits a natural lift to a uniformly woven current on $Y$.
		
		\begin{defi}\label{defi:pullbackwoven}
			Assume in addition that $(Y,\omega_Y)$ is K\"ahler.  Let 
			$$S=\int_{\sV_{p,Z}} [V_a]\;d\mu(a)$$
			be a uniformly woven current on $Z$.  For this woven representation of $S$, we define its lifted current on $Y$ by
			$$\pi^*_{\mu}S:=\int_{\sV_{p,Y}} [V_b] \;d(\Pi_*\mu)(b),$$
			where $\Pi: \sV_{p,Z}\to \sV_{p,Y}$ is the pullback operator and $\sV_{p,Y}$ is given in Proposition \ref{prop:continuouspullback}. 
		\end{defi}

		\begin{rmk}
			The lifted current $\pi^*_\mu S$ is defined here relative to the chosen woven representation of $S$. We do not use representation-independence in this paper; the above construction is sufficient for our applications.
			We will address the question of the independence of this woven representation in future work.
		\end{rmk}

		\medskip

		\section{Bidimension (1,1) woven currents and intersection theory}\label{sec:woven11}
		
		\subsection{Currents that are approximated by algebraic curves satisfying genus-volume relation}\label{subsection:3.1}
		
		In this subsection, we consider positive closed bidimension $(1,1)$ currents in a quasi-projective variety $X$ over $\C$   that can be approximated by  effective algebraic cycles of pure dimension one satisfying a genus-volume relation.  Dujardin \cite{dujardin2003laminar}, \cite{dujardin04} developed an intersection theory for woven currents of bidimension $(1,1)$ in two dimensions, and in \cite{Dujardin2012} he generalized this two-dimensional result to higher dimensions. The results we present in this section can be regarded as an extension of Dujardin's work: while \cite{Dujardin2012} considers currents that arise as limits of rational curves, in this section we treat currents that are limits of curves satisfying a genus–volume relation.


		\begin{defi}[Genus-volume relation]\label{defi:genus-volume}
			Let $(V_n)_{n\geq 1}$ be a sequence of effective algebraic cycles of pure dimension one in $\P^N$ over $\C$.  We say that $(V_n)_{n\geq 1}$ satisfies the {\em genus-volume relation}  if there is a constant $C>0$ such that for every $n\geq 1$, we have 
			$$\genus (V_n)\leq C\;\vol(V_n).$$
		\end{defi}
		Here for $V=\sum_W n_W W$ an effective algebraic cycle of pure dimension one, where each $W$ is an  irreducible algebraic curve in $\P^N$ and $n_W\in \Z_{\geq 1}$,  we define $\genus(V):=\sum_W  n_W \genus (W)$.  By genus  of $W$ we mean  the geometric genus of any projective compactification of $W$. The volume is computed using Fubini-Study form on $\P^N(\C)$. 
		\medskip

		%
		%
		%

		%
		
		%
		
		Let $I$ be a linear subspace of $\P^N(\C)$ of dimension $N-2$, and let $V$ be a projective line in $\P^N(\C)$ such that $V\cap I=\emptyset$. Let $\pi:\P^N(\C)\setminus I\to V\cong\P^1(\C)$ be the projection defined as follows: for $x \in \P^N(\C)\setminus I$, let $H_x$ be the unique hyperplane passing through $x$ and $I$; then $\pi(x)$ is the unique point in $H_x\cap V$. 
		
		Let $\Gamma$ be an irreducible algebraic curve in $\P^N(\C)$ of genus $g$ such that $\Gamma\cap I=\emptyset$.   Let $U$ be a simply connected open subset of $\P^1(\C)$, and let $W$ be an irreducible component of $\Gamma\cap \pi^{-1}(U)$.  We say that $W$ is a {\em good component} if the map $\pi:  W\to U$ is biholomorphic. Equivalently, this means that $W$ is a graph over $U$.  Otherwise $W$ is called a {\em bad component}. For a bad component $W$, the multiplicity $m(W,U)$ is by definition the topological degree of $\pi: W\to U$. We denote by $b(\Gamma, U)$ the number of bad components in $\Gamma$ over $U$ counted with multiplicity, i.e. $b(\Gamma, U)=\sum_{W} m(W,U)$, where the summation takes over all irreducible components $W$ of $\Gamma\cap \pi^{-1}(U)$.

		\begin{lem}\label{bad components}
			Let $\Gamma$ be an irreducible algebraic curve in $\P^N(\C)$ of genus $g$.  Let $I$ be a linear subspace of $\P^N(\C)$ of dimension $N-2$ such that $\Gamma\cap I=\emptyset$.  let $(U_i)_{1\leq i\leq l}$ be a disjoint collection of simply connected open subset of $\P^1(\C)$. Then 
			\begin{equation*}
				\sum_{i=1}^l b(\Gamma,U_i)\leq 4\vol(\Gamma)+4g-4,
			\end{equation*}
			where $\vol(\Gamma)$ is computed with respect to the Fubini-Study form  $\omega_{\P^N}$.
		\end{lem}
		\begin{proof}
			Let $p:\hat{\Gamma}\to \Gamma$ be the normalization of $\Gamma$, hence $\hat{\Gamma}$ is a complex smooth algebraic curve of genus $g$. Considering the branched covering $\pi\circ p:\hat{\Gamma}\to \P^1(\C)$, by Riemann-Hurwitz formula we have 
			
			$$2g-2=-2\deg(\pi\circ p)+\sum_{x\in R} (e_x-1),$$
			where $R$ is the critical set of $\pi\circ p:\hat{\Gamma}\to \P^1(\C)$, and $e_x$ is the critical multiplicity at $x$. 
			
			Since $\deg(\pi\circ p)=\deg (\pi|_\Gamma)=\vol (\Gamma)$ (both are equal to the number of the intersection points of $\Gamma$ with a generic hyperplane), and since $e_x\geq 2$ when $x\in R$,  we have
			$$\sum_{x\in R} e_x\leq 4\vol(\Gamma)+4g-4.$$
			
			If $\pi:W\to U_i$ is a bad component, then the multiplicity $m(W,U_i)$  is the same as $\deg ((\pi\circ p)|_{\hat{W}})$, where $\hat{W}:=p^{-1}(W)$.  Since $U_i$ is simply connected,  we have
			
			$$\deg ((\pi\circ p)|_{\hat{W}})\leq \sum_{x\in R\cap \hat{W}} e_x.$$
			For the proof, see the last paragraph in page 750 of the proof of \cite[Lemma 3.2]{dujardin2003laminar}.
			
			Since $U_i$, $1\leq i\leq l$ are disjoint, we have 
			\begin{equation*}
				\sum_{i=1}^l b(\Gamma,U_i)=	\sum_{i=1}^l \sum_W  m(W, U_i)\leq \sum_{x\in R} e_x \leq 4\vol(\Gamma)+4g-4,
			\end{equation*}
			where, in the second term, the sum is taken over all irreducible components $W$ of $\Gamma\cap\pi^{-1}(U_i)$. This completes the proof.
		\end{proof}
		\medskip

		\begin{defi}[Affine coordinate]\label{defi:affine}
			An affine coordinate is the following data:
			
			\begin{points}
				\item $H$ is a hyperplane of $\P^N(\C)$;
				\item $I_1,\dots, I_N$ are  $N$  linear subspaces  in $H$  of dimension $N-2$, which are in general positions.
				
				\item For each $i$, $\pi_i: \P^N(\C)\setminus I_i\to \P^1(\C)$ is a projection as in Subsection \ref{subsection:3.1} such that $\pi_i^{-1}(\infty)\in H$. 
			\end{points} 
		\end{defi}
		
		In the above definition,  the map 
		$$(\pi_1,\dots,\pi_N):\P^N(\C)\setminus H \to \C^N$$
		is an isomorphism, which gives an affine coordinate via those $\pi_i$.  In particular, for each $z\in \C$, the fiber  $\pi_i^{-1}(z)$ is isomorphic to $\C^{N-1}$.
		
		\medskip

		Let $r>0$. For each $1\leq i\leq N$, let $\sQ_i$ be a subdivision of $\C$ into {\em affine cubes} of size $r$ (with respect to the Euclidean metric $\omega_\C$).  Here by an affine cube we mean an open affine cube, and by subdivision we mean that the union of affine cubes form an open dense subset of $\C$, and two different affine cubes have empty intersection.  An affine cube in $\C^N$ is by definition the set $\cap_{1\leq i\leq N}\; \pi_i^{-1}(Q_i)$, where $Q_i$ is an affine cube in $\C$ subordinate to the subdivision $\sQ_i$. Let $\sQ$ be the subdivion of $\C^N$ given by the above affine cubes in $\C^N$. Let $\Omega_\sQ$ be the union of all affine cubes in $\C^N$  subordinate to the subdivision $\sQ$, whcih is an open and dense subset of $\C^N$.  Let $\partial \sQ:=\C^N\setminus \Omega_\sQ$ be the boundary of $\sQ$. 
		
		
		
		\medskip
		
		Let $B$ be an open  ball of $\C^N$. We fix the ball $B$, an affine coordinate, subdivisions $\sQ_i$ of $\C$ of size $r$ and   subdivision $\sQ$ of $\C^n$ of size $r$. 
		
		\medskip
		
		The following lemma is crucial. 
		\begin{lem}\label{lem: estimate}
			Let  $(V_n)_{n\geq 1}$ be a sequence of  effective algebraic cycles of pure dimension one in $\P^N$ satisfying the genus-volume relation.  Assume  that $S_n:=[V_n]/\vol (V_n)\to S$  in the sense of currents.   Assume that for each $i$ and $n$, the projection center $I_i\cap V_n=\emptyset$ and $S$ puts no mass on $\partial \sQ$, i.e.
			$$(S\wedge \omega_{\P^N})(\partial\sQ)=0.$$
			Then there is a uniformly woven current $S_{\sQ}$ on $\Omega_{\sQ}$  such that  $S_{\sQ}\leq S$ and the mass of the measure $(S-S_{\sQ})\wedge \omega_{\P^N}$ on $B$ satisfies
			\begin{equation}\label{eqn: estimate1}
				\Ma_B((S-S_{\sQ})\wedge \omega_{\P^N})\leq Cr^2,
			\end{equation}
			where $C$ is a constant  independent of $\sQ$. 
			
			
			%
		\end{lem}
		Let $T$ be a positive closed current on an open subset  $\Omega\subset \P^N(\C)$. The zero extension of $T$ is a positive current on $\P^N(\C)$. By abuse of notation, we also denote this extension by $T$. In general $T$ is not closed.  So in Lemma \ref{lem: estimate}, $S_\sQ$  can be seen as  a positive but not necessarily closed current on $\P^N(\C)$.

		For any open subset $\Omega\subset \P^N(\C)$,  we denote by $\sV_{\Omega}$  the space of all effective analytic cycles of pure dimension one with volume bounded by $1$ from above. 
		\proof[Proof of Lemma \ref{lem: estimate}]

		Consider one projection $\pi_i: \P^N(\C)\setminus I_i\to \P^1(\C)$.  Let $Q\in \sQ_i$ be an affine cube. For each $n\geq 1$, let $V_n=\sum_{j=1}^l V_{n,j}$ be the decomposition of $V_n$ into irreducible components. Here we allow $V_{n,j}$  to be repeated.  Let $b(V_{n,j},Q)$  be the number of bad components of $\pi_i: V_{n,j}\to \P^1(\C)$ over $Q$, counted with multiplicity.  Let $b(V_n,Q):=\sum_{j=1}^l b(V_{n,j},Q)$, and let $b(V_n,\sQ_i):=\sum_{Q\in \sQ_i} b(V_n,Q)$.  Since $(V_n)_{n\geq 1}$ satisfy  the genus-volume relation, by Lemma \ref{bad components}, there is a constant $C_1>0$ independent of $n$ and $\sQ$ such that  
		\begin{equation}\label{3.2}
			b(V_n,\sQ_i)\leq C_1\vol (V_n).
		\end{equation}

		Next we shall construct the uniformly woven current $S_{\sQ}$  using good components of $\pi_i:V_n\to \P^1(\C)$, $1\leq i\leq N$.   Let $V_{n,\sQ_i}$ be the sum of all good components $\Gamma$ with respect to $\pi_i$ with volume bounded by $1$ from above; moreover $\Gamma\cap \overline{B}\neq \emptyset$.  There exists a compact subset $K\subset \C$ independent of $n$ such that $\pi_i(\Gamma)\subset K$ for every such $\Gamma$.

		Let $\Omega_i:=\cup_{Q\in \sQ_i} \pi_i^{-1}(Q)$. Let $S_{n,\sQ_i}$ be the closed current on $\Omega_i$,
		$$S_{n,\sQ_i}:=\frac{1}{\vol(V_{n})} [V_{n,\sQ_i}].$$  
		
		We can write $S_{n,\sQ_i}$ as a uniformly woven current
		$$S_{n,\sQ_i}=\int_{\sV_{\Omega_i}} [\Gamma_a] \; d\nu_{n,i}(a),$$
		and each $\Gamma_a$ is a graph with respect to $\pi_i$ over an  affine cube $Q\subset \C$, $\nu_{n,i}$ is a finite average of Dirac masses on $\sV_{\Omega_i}$. The condition $\Gamma\cap \overline{B}\neq \emptyset$ implies that the volume of $\Gamma_a$ is bounded uniformly from below by a constant independent of $n$, since such $\Gamma$ is a graph over a cube $Q\subset K$.   Thus the mass of $\nu_{n,i}$  is uniformly bounded from above by a constant independent of $n$.  
		It is clear that $S_{n,\sQ_i}\leq S_n$ on $\Omega_i$.

		Let $V_{n,\sQ}$ be the sum of all  irreducible component $\Gamma$ of $V_n|_{\Omega_\sQ}$, such that $\Gamma$ has volume bounded by $1$, and there exists at least  $1\leq i\leq N$  such that $\Gamma$ is contained in a component of $V_{n,\sQ_i}$.

		Let  $S_{n,\sQ}$ be the closed  current on $\Omega_\sQ$,
		$$S_{n,\sQ}:=\frac{1}{\vol(V_{n})} [V_{n,\sQ}].$$
		Define  $B':=B\setminus \partial \sQ$, which is an open dense subset of $B$ .  It is  clear that for each $1\leq i\leq N$,
		$$S_{n,\sQ}\leq S_n \;\text{on} \;B, \;\text{and} \;S_{n,\sQ}\geq S_{n,\sQ_i} \;\text{on} \;B'.$$
		
		The total number of  good components $\Gamma$ of  $\pi_i:V_n\to \P^1(\C)$ over  some $Q\in \sQ_i$ such that $\vol(\Gamma)\geq 1$ is at most  $\vol(V_n)$.  Together with (\ref{3.2}), this implies that  there is a constant $C_2>0$ independent of $n$ and $\sQ$ such that 
		\begin{equation*}
			\Ma_B ((S_n-S_{n,\sQ_i})\wedge \pi_i^*(\omega_\C))\leq C_2r^2.
		\end{equation*}
		Since $S_{n,\sQ}\geq S_{n,\sQ_i}$ on $B'$ and $S_{n,\sQ}\leq S_n$ on $B$, we have
		\begin{equation*}
			\Ma_{B'} ((S_n-S_{n,\sQ})\wedge \pi_i^*(\omega_\C))\leq C_2r^2.
		\end{equation*}
		Since $\sum_{i=1}^N\pi_i^*(\omega_\C)$ restricted on $B$ is equivalent to $\omega_{\P^N}$ restricted on $B$,  there is a constant $C_3>0$ independent of $n$ and $\sQ$ such that 
		\begin{equation}\label{eqn: 3.4}
			\Ma_{B'} \left(\left(S_n-S_{n,\sQ}\right)\wedge \omega_{\P^N} \right)\leq C_3r^2.
		\end{equation}
		
		\medskip
		
		Write $$S_{n,\sQ}=\int_{\sV_{\Omega_\sQ}} [\Gamma_a] \; d\nu_{n}(a),$$
		where $\nu_{n}$ is a finite  measure on $\sV_{\Omega_\sQ}$ with mass $\Ma(\nu_n)\leq \sum_{i=1}^N \Ma(\nu_{n,i})$, which is uniformly bounded from above by a constant independent of $n$. 
		
		The space $\sV_{\Omega_\sQ}$ is a compact Hausdorff  space by Theorem \ref{thm: Bishop}.   Since the  masses of all $\nu_{n}$ are bounded by $N$,  by passing to a subsequence we may assume $\nu_{n}\to \nu$ in the weak-* sense, where $\nu$ is a finite  measure on $\sV_{\Omega_\sQ}$.  Let 
		$$S_\sQ:=\int_{\sV_{\Omega_\sQ}} [\Gamma_a] \; d\nu(a).$$
		
		

		By Proposition \ref{prop: limit}, we have $S_{n,\sQ}\to S_\sQ$ in the sense of currents. Since $S_{n,\sQ}\leq S_n$ and $S_n\to S$ in the sense of currents, we have $S_\sQ\leq S$.  By (\ref{eqn: 3.4}) we  have 
		
		\begin{equation*}
			\Ma_{B'} \left(\left(S-S_{\sQ}\right)\wedge \omega_{\P^N} \right)\leq C_3r^2.
		\end{equation*}
		
		Since $0\leq S-S_{\sQ}\leq S$ and $S\wedge \omega_{\P^N}$ gives no mass to $B\setminus B'$, we have 
		\begin{equation*}
			\Ma_{B} \left(\left(S-S_{\sQ}\right)\wedge \omega_{\P^N} \right)\leq C_3r^2.
		\end{equation*}
		This implies (\ref{eqn: estimate1}), which finishes the proof.
		\qed
		

		\medskip

		\medskip
		Let $\sQ_0$ be a fixed subdivision of $\C^N$ into affine cubes of size $r$.  Let $\tilde{Q}\in \sQ_0$ be one of the affine cubes.    Let $x\in \tilde{Q}$, let $\sQ_x:= \sQ_0+x$ be the translation of $\sQ_0$, which is also a subdivision of $\C^N$ of size $r$.  Let $0<\la<1$. If $Q$ is an affine cube of size $r$, we let $\la Q$ be the cube with the same center as $Q$, homothetic to it by the factor $\la$.  Let $m$ be the Lebesgue measure on $\C^N$.
		
		\begin{lem}\label{lem: translation}
			Let $\mu$ be a probability measure on $\C^N$ and $0<\la<1$, $0<\delta<1$.  Then there exists a Borel measurable subset $A\subset \tilde{Q}$ such that $m(A)\geq (1-\delta)  m(\tilde{Q})$,  and  for $x\in A$, 
			$$\mu\left(\C^N\setminus \cup_{Q\in\sQ_x}  \la Q \right)\leq 1-\delta \la^{2N},$$
			moreover $\mu$ put no mass on $\partial \sQ_x$.
		\end{lem}
		\begin{proof}
			Consider the product space $ \C^N \times \C^N$, and the product meassure $m \tensor \mu$ on it.  Let $W\subset \C^N \times \C^N$ be the subset containing points $(x,y)$ such that $x\in \tilde{Q}$ and $y\in \cup_{Q\in \sQ_x} \la Q$.  By Fubini's theorem
			$$\int_{\tilde{Q}} \mu\left(\cup_{Q\in\sQ_x}\la Q\right)\;dm(x)=\int_{\C^N \times \C^N} \ch_W \; d(m \tensor \mu)(x,y).$$
			The slicing of $W$ with respect to $y\in \C^N$ gives the set $y+\la \tilde{Q}$. Again by Fubini's theorem
			$$\int_{\C^N \times \C^N} \ch_W \; d(m \tensor \mu)(x,y) = \int_{\C^N} m(y+\la \tilde{Q})\;d\mu(y)=\la^{2N}m(\tilde{Q}).$$
			
			Let $A'\subset \tilde{Q}$ be the subset containing points $x$ such that $\mu\left(\cup_{Q\in\sQ_x}\la Q\right)\geq \delta \la^{2N}$. Since $\int_{\tilde{Q}} \mu\left(\cup_{Q\in\sQ_x}\la Q\right)\;dm(x)=\la^{2N}m(\tilde{Q})$,  we have  $m(A')\geq (1-\delta) m(\tilde{Q})$. 
			
			Let $A''\subset \tilde{Q}$ be the subset containing points $x$ such that $\mu\left(\cup_{Q\in\sQ_x} Q\right)=1$. Since $\int_{\tilde{Q}} \mu\left(\cup_{Q\in\sQ_x} Q\right)\;dm(x)=m(\tilde{Q})$ and $\mu\left(\cup_{Q\in\sQ_x} Q\right)\leq 1$, we have $m(A'')= m(\tilde{Q})$.  For $x\in A''$, $\mu$ put no mass on $\partial \sQ_x$.
			
			The set $A:=A'\cap A''$ is what we want. 
		\end{proof}
		\medskip

		Let $X$ be a complex quasi-projective variety, embedded in $\P^N(\C)$. Let $\omega_X$ be the restriction of  the Fubini-Study form $\omega_{\P^N}$ on  $\P^N(\C)$ to $X$. The volumes and diameters on $X$ are computed with respect to $\omega_X$.    Let $B$ be a ball inside an affine coordinate $\C^N$ such that $B_X:=B\cap X$ is compactly contained in $X$.

		The following is the main result of this section.

		\begin{thm}\label{thm: strongly approximable}
			Let  $(V_n)_{n\geq 1}$ be a sequence of  effective algebraic cycles of pure dimension one in $X$ satisfying the genus-volume relation. Assume that $[V_n]/\vol (V_n)\to S$  in the sense of currents. Let $R$ be a  positive closed $(1,1)$-current on $X$ with continuous local  potential. Then there exists an affine coordinate (Definition \ref{defi:affine}) such that for each $r>0$, there exists a subdivision $\sQ$ of $\C^N$ of size $r$,  a uniformly woven current $S_{r}$ on $\Omega_{\sQ}$ such that  $S_{r}\leq S$ and 
			\begin{equation}\label{eqn: sa2}
				\Ma_{B_X}((S-S_{r})\wedge R)\to 0
			\end{equation}
			when $r\to 0$.  
			
			
			%
		\end{thm}
		\begin{proof}
			We choose the affine coordinate in Definition \ref{defi:affine} after $B$ has been fixed. More precisely, choose a hyperplane $H\subset \P^N(\C)$ such that $H\cap \overline{B}=\emptyset$ and no irreducible component of any $V_n$ is contained in $H$. Then choose $N$ linear subspaces $I_1,\dots,I_N\subset H$ in general position such that $I_i\cap V_n=\emptyset$ for every $i$ and every $n\geq 1$; this is possible by avoiding a countable union of proper algebraic subsets in the corresponding parameter spaces.  We choose $N$ projections $(\pi_i)_{1\leq i\leq N}$ associated to these $I_i$. In the following affine cubes in $\C^N$ are with respect to $(\pi_i)_{1\leq i\leq N}$. To show the theorem, it suffices to show for every $\ep>0$,  there exists $r_0=r_0(\ep)>0$ such that for $0<r<r_0$, there exists a subdivision $\sQ$ of $\C^N$ of size $r$ and a current $S_r$ satisfying conditions in Theorem \ref{thm: strongly approximable} such that $	\Ma_{B_X}((S-S_{r})\wedge R)<2\ep$. 
			
			Let $\mu_0$ be the restriction of the measure $S\wedge R$ on $B_X$, which is a finite measure. Let $\mu$ be the zero extension of  $\mu_0$ on $\C^N$.  By Lemma \ref{lem: translation},  for every $\ep>0$,  there exists  $0<\la=\la(\ep)<1$  sufficiently close to $1$ such that for every  $r>0$, there is a subdivision $\sQ$ of $\C^N$ of size $r$ such that:
			\begin{points}
				\item $\mu$ put no mass on $\partial \sQ$.;
				\item $\mu\left(\C^N\setminus \cup_{Q\in\sQ}  \la Q \right)<\ep$.
			\end{points}
			
			A crucial point here is that (ii) holds uniformly in $r$. Note that the subdivision $\sQ$ satisfies the conditions in Lemma \ref{lem: estimate}. Let $S_r$ be the current $S_\sQ$ in Lemma  \ref{lem: estimate}.  It is clear that $S_r\leq S$. It remains to show $	\Ma_{B_X}((S-S_{r})\wedge R)<2\ep$ when $0<r<r_0$, where $r_0=r_0(\ep)>0$ is a constant to be defined later.
			\medskip
			
			For each $r>0$, let $\sQ$ be the subdivision of $\C^N$ of size $r$ constructed above. The problem is local, so we can write $R=dd^c u$ on $B_X$, where $u$ is a continuous p.s.h. function on a neighborhood of $\overline{B_X}$.  Let $\psi$ be a smooth cut-off function on  $\cup_{Q\in\sQ}   Q$, such that $0\leq \psi_\sQ\leq 1$, $\psi=1$ on $\la Q$ for each $Q\in \sQ$, and we have $||\psi||_{C^2}\lesssim r^{-2}$.  Here $\lesssim $ means that $||\psi||_{C^2}\leq C r^{-2}$ for a constant independent of $r$. Hence we can write $dd^c \psi\wedge [X]=S^{+}-S^{-}$, where $S^{\pm}$ are positive closed currents such that $S^{\pm}\wedge [X]\lesssim r^{-2} \omega_{\P^N}\wedge [X]$ on $B$. On each affine cube $Q\in \sQ$,  let $c_Q\in Q$ be the center of $Q$. By Lemma \ref{lem: estimate} (\ref{eqn: estimate1}), we have 
			
			\begin{align*}
				\int_{B_X} \psi(S-S_r)\wedge R &=\int_{B_X} \psi(S-S_r)\wedge dd^c u
				\\ &=\int_{B_X}[u-u(c_Q)] (S-S_r)\wedge dd^c \psi\wedge [X]
				\\ &\lesssim r^{-2}\omega(u,2r)\Ma_B((S-S_r)\wedge \omega_{\P^N})
				\\  &\lesssim \omega(u,2r),
			\end{align*}
			where $$\omega(u,2r):=\sup_{z_1,z_2\in \overline{B_X},|z_1-z_2|\leq 2r}{|u(z_1)-u(z_2)|}$$ is the modulus of continuity of $u$ of size $2r$.  Since $u$ is continuous  on a neighborhood of $\overline{B_X}$, there is a constant $r_0=r_0(\ep)>0$ such that $ \int_{B_X} \psi(S-S_r)\wedge R<\ep$ when $0<r<r_0$.
			
			Since $\mu\left(\C^N\setminus \cup_{Q\in\sQ}  \la Q \right)<\ep$, we have 
			\begin{align*}
				\Ma_{B_X}((S-S_r)\wedge R)&\leq \Ma_{B_X\setminus \cup_{Q\in\sQ}\la Q} \;S\wedge R+ \int_{B_X} \psi(S-S_r)\wedge R \\&= \mu\left(\C^N\setminus \cup_{Q\in\sQ}  \la Q \right)+ \int_{B_X} \psi(S-S_r)\wedge R
				\\&< 2\ep,
			\end{align*}
			when $0<r<r_0$. This completes the proof.
		\end{proof}
		\medskip
		
		\subsection{Currents that are approximated by algebraic curves with constant gonality}
		In this subsection, we consider positive closed bidimension $(1,1)$ currents on a quasi-projective variety $X$ over $\C$  that can be approximated by  algebraic curves with constant gonality.  Our result is of independent interest in the theory of woven currents: we relax the bounded genus assumption to bounded gonality.  Our result is new even in dimension two.
		
		\medskip

		Let $X$ be a complex quasi-projective  variety.

		\medskip
		
		Let $m\geq 1$.   Let $X^m$ be the $m$-th product space. The permutation group $S_m$ acts on $X^m$ by permuting different coordinates. Let $X^m/S_m$ be the quotient variety. Let  $\pi:V^m\to V^m/S_m$ be the canonical  projection and let  $p:X^m\to X$ be the projection onto the first  coordinate.
		
		The following Lemma is crucial.
		\begin{lem}\label{lem:lift}
			Let $V$ be an irreducible complex quasi-projective curve with $\gonality (V)=m$.  Let $p:V^m\to V$ be  the projection onto the first  coordinate. Then
			
			(1)  There exists an algebraic curve  $\Gamma$ in  $V^m/S_m$  which is birational to $\P^1$;
			
			(2) Let $W:=\pi^{-1}(\Gamma)$. Then $p$ is dominant when restricted to any irreducible component of $W$, and, as cycles on $V$,
			$$p_*[W]=(m-1)![V].$$

		\end{lem}
		\begin{proof}
			Since $\gonality (V)=m$, there exists a degree $m$ rational map $\phi: V\dashrightarrow \P^1$.  Let $U\subset \P^1$ be the Zariski open set containing all $t\in \P^1$ such that $\phi^{-1}(t)$ has $m$ points counted with multiplicity.  Let $\psi:\P^1\dashrightarrow V^m/S_m$ be the rational map given by the following: for $t\in U$, $\phi^{-1}(t)$ has $m$ points  counted with multiplicity,  define $\psi(t):=\left\{x_1,\dots, x_m\right\}\in V^m/S_m$ be the unordered list of elements in $\phi^{-1}(t)$.  It is clear that $\psi$ is injective on $U$, hence the Zariski closure $\Gamma:=\overline{\psi(U)}^{\zar}\subset V^m/S_m$ is birational to $\P^1$. 
			

			
			For the second assertion, observe that $W$ is invariant under the action of $S_m$. Let $x\in W$ be a generic point such that all coordinates of $x$ take different values.  Then $y\in W$ such that  $p(x)=p(y)$ if and only if $\pi(y)=\pi(x)$ and $p(x)=p(y)$, if and only if $y=\sigma(x)$ for $\sigma\in G$,   where $G$ is the subgroup of $S_m$ that  fix the first coordinate.  Thus a generic point of $V$ has exactly $|G|=(m-1)!$ preimages in $W$, counted in the cycle $[W]$. Hence $p_*[W]=(m-1)![V]$.
		\end{proof}
		
		\medskip
		Let $X$ be a complex  quasi-projective variety. We  embedd $X$, $X^m$ and $X^m/S_m$ into some $\P^N$ with $N$ large.  Volume of subvarieties  are  computed with respect to the Fubini-Study form on $\P^N(\C)$.
		
		\begin{thm}\label{thm:sa gonality}
			Let $V_n$ be a sequence of  effective algebraic cycles of pure dimension one in $X$ such that every irreducible component of $V_n$ has gonality $m$ for every $n\geq 1$.  Assume that $[V_n]/\vol (V_n)\to S$  in the sense of currents. Then there exists a positive closed bidimension $(1,1)$ current $S'$ on $X^m$ such that the following hold:  
			
			\begin{points}
				\item $p_*(S')=S$;
				
				\item Let  $R$ be a  positive closed $(1,1)$-current on $X^m$ with continuous local  potential and $B$ be  any  ball compactly contained in $X^m$.  Then for each $r>0$, there is a uniformly woven current $S'_{r}$ on an open subset $\Omega_r\subset X^m$, such that  $S'_{r}\leq S'$ and 
				\begin{equation*}
					\Ma_{B}((S'-S'_{r})\wedge R)\to 0
				\end{equation*}
				when $r\to 0$.  
			\end{points}
		\end{thm}

		\begin{proof}
			We may assume $S\neq 0$, since otherwise the result is trivial.  Write $V_n=\sum_{j=1}^l V_{n,j}$ as a sum of irreducible components. For each $j$, let $\Gamma_{n,j}\subset X^m/S_m$ be the curve associated to $V_{n,j}$ by Lemma \ref{lem:lift}, and set
			$$\Gamma_n:=\sum_{j=1}^l\Gamma_{n,j},\qquad V'_n:=\pi^*\Gamma_n.$$
			By Lemma \ref{lem:lift}, $p_*[V'_n]=(m-1)![V_n]$.  Passing to a subsequence, assume
			$$\frac{[V'_n]}{\vol(V'_n)}\to S''.$$
			Since $[V_n]/\vol(V_n)\to S\neq 0$, after passing to a further subsequence there is $\alpha>0$ such that
			$$\frac{\vol(V_n)}{\vol(V'_n)}\to \alpha.$$
			Then $p_*S''=(m-1)!\alpha S$.  Let  $c:=\frac{1}{(m-1)!\alpha}$ and $S':=cS''.$  Then we have $p_*S'=S$.
			
			It remains to prove (ii).  Let $R$ be as in the statement and let $B\Subset X^m$ be a ball.  Set $\tilde R:=\pi_*R$, and choose a relatively compact subset $\tilde B\subset X^m/S_m$ with $B\subset \pi^{-1}(\tilde B)$.  The current $\tilde R$ has continuous local potential by Proposition \ref{prop:pushforwardcontinuous}.  By Lemma \ref{lem:volumeestimate}, the masses of $[\Gamma_n]/\vol(V'_n)$ are uniformly bounded on compact subsets. Since every irreducible component of $\Gamma_n$ is birational to $\P^1$, the proof of Theorem \ref{thm: strongly approximable} applies, with the normalization $\vol(V'_n)$, to the sequence $\Gamma_n$, the test current $\tilde R$, and the compact set $\tilde B$. Thus, for every $r>0$, after passing to a subsequence, we get an open set $\tilde\Omega_r\subset X^m/S_m$ and subcycles
			$$\Gamma_{n,r}\leq \Gamma_n|_{\tilde\Omega_r}$$
			whose components have uniformly bounded volume on compact subsets, such that the cluster limit of $[\Gamma_{n,r}]/\vol(V'_n)$ is uniformly woven on $\tilde\Omega_r$, and
			\begin{equation}\label{eqn:revised-gonality-good}
				\limsup_{n\to\infty}
				\Ma_{\tilde B}\left(\frac{[\Gamma_n]-[\Gamma_{n,r}]}{\vol(V'_n)}\wedge \tilde R\right)
				\leq \varepsilon(r),
			\end{equation}
			where $\varepsilon(r)\to0$ as $r\to0$.
			
			Let $V'_{n,r}:=\pi^*\Gamma_{n,r}$.  Since $0\leq V'_{n,r}\leq V'_n$, after passing to a subsequence we may assume
			$$\frac{[V'_{n,r}]}{\vol(V'_n)}\to S''_r$$
			with $S''_r\leq S''$.  Moreover, by the bounded-volume property of the curves $\Gamma_{n,r}$ and Proposition \ref{prop:continuouspullback}, the current
			$$S'_r:=cS''_r$$
			is uniformly woven on $\Omega_r:=\pi^{-1}(\tilde\Omega_r)$, and $S'_r\leq S'$.
			
			Finally, using the projection formula for analytic cycles at finite level and then letting $n\to\infty$, we get
			\begin{align*}
				\Ma_B((S'-S'_r)\wedge R)
				&\leq c\liminf_{n\to\infty}
				\Ma_{\pi^{-1}(\tilde B)}
				\left(\frac{[V'_n]-[V'_{n,r}]}{\vol(V'_n)}\wedge R\right)\\
				&=c\liminf_{n\to\infty}
				\Ma_{\tilde B}
				\left(\frac{[\Gamma_n]-[\Gamma_{n,r}]}{\vol(V'_n)}\wedge \tilde R\right)\leq c\,\varepsilon(r)
			\end{align*}
			by (\ref{eqn:revised-gonality-good}).  Since $\varepsilon(r)\to0$ as $r\to0$, this proves (ii), and the theorem follows.
		\end{proof}

		\medskip
		\section{Bifurcation theory for families of endomorphisms of \texorpdfstring{$\P^N$}{projective space}}\label{section:bifur}
		
		\subsection{The bifurcation current and bifurcation set}\label{sec:bifur current}
		Let $f$  be a   holomorphic  family of endomorphisms on $\P^N(\C)$ of degree $d\geq 2$,  parametrized by a Hermitian manifold  $(\La,\omega_\La)$, i.e. $f$ is a holomorphic map
		\begin{align*}
			f:\La\times \P^N(\C)&\to \La\times \P^N(\C)\\
			(t,z)&\mapsto (t,f_t(z)),
		\end{align*}
		such that  for every $t\in \La$, $f_t$ is a degree $d$ endomorphism of $\P^N$ over $\C$.
		Let $\pi_1:\La\times\P^N(\C)\to \La$ and $\pi_2:\La\times\P^N(\C)\to \P^N(\C)$ be the canonical projections. Let $\omega_{\P^N}$ be the Fubini-Study form on $\P^N(\C)$.  Let  $\omega_1:=\pi_1^\ast (\omega_{\La})$ and $\omega_2:=\pi_2^\ast (\omega_{\P^N})$.  The {\em relative Green current}  of $f$ is defined by 
		
		\begin{equation*}
			T_f:=\lim_{n\to+\infty} d^{-n} (f^n)^\ast (\omega_2)=\omega_2+dd^c u,
		\end{equation*}
		where $u$ is a H\"older continuous quasi-p.s.h. function \cite[Lemma 1.19]{dinh2010dynamics}.
		\par For every $t\in \La$, we have $T_f^N\wedge [\pi_1^{-1}(t)]=\mu_{t}$, where $T_f^N$ is the $N$-times  
		self intersection of $T_f$ and $\mu_t$ is the unique maximal entropy measure of $f_t$.

		\medskip
		
		Let $[C_f]$ be  the current of integration on the critical set $C_f$ of $f$ taking into
		account the multiplicities.  In \cite{bassanelli2007bifurcation}, Bassanelli-Berteloot define the {\em bifurcation current} $T_\bif:=(\pi_1)_*(T_f^N\wedge [C_f])$, which generalized DeMarco's definition when $N=1$ \cite{demarco2003dynamics}.  Pham \cite{pham2005lyapunov} define the bifurcation current in the more general setting of polynomial-like maps.  If $\La$ is quasi-projective and $f$ is an algebraic family, then $T_\bif$ has finite mass on $\La$ \cite[Proposition 3.3]{Gauthier2019} 
		%
		
		The current  $T_\bif $ has  a global potential function $L(t)$,  which is the sum of the Lyapunov exponents of the measure $\mu_{t}$ \cite[Corollary 4.6]{bassanelli2007bifurcation}. We have $dd^c L=T_\bif$. The  function $L$  is H\"older continuous   \cite[Theorem 2.47]{dinh2010dynamics}. 
		
		\begin{defi}
			The {\em bifurcation set } $\Bif(f)$ is the support of the bifurcation current $T_\bif$. The {\em stable set} is $\Stab (f):=\La\setminus \Bif(f)$.
		\end{defi}
		
		Let $g$ be an endomorphism of $\P^N$ over $\C$ of degree $d\geq 2$. Let $\mu_g$ be the unique maximal entropy measure, and let $J_g$ be the support of $\mu_g$.  A periodic point $x$ of $g$ is called {\em $J$-repelling} if $x$  is repelling and $x\in J_g$.

		When $N=1$, a repelling periodic point is automatically $J$-repelling, counterexamples appear when $N\geq 2$ \cite[Proposition 6.4]{hubbard1994superattractive}.
		
		\medskip
		
		Let $s\in \La$ and let $x$ be a $J$-repelling periodic point of $f_{s}$. Let $U\subset \La$ be  an open set  containing $s$. We say that $x$ {\em moves holomorphically} on $U$ if there exists a holomorphic map $p:U\to \P^N(\C)$ such that:
		
		(1) $p(s)=x$;
		
		(2) For every $t\in U$,  $p(t)$ is a $J$-repelling periodic point of $f_t$.

		\medskip

		%
		%
		%
		%
		\subsection{The equilibrium web and equilibrium  graph}

		In  \cite{berteloot2018dynamical}, Berteloot--Bianchi--Dupont established the equivalence of various stability definitions, which is the major progress in high-dimensional bifurcation theory in recent years. They in particular introduced a key notion called {\em equilibrium web}, which will be very useful for our purpose.  We now give the definition.  
		
		Let $U\subset \La$ be an open set.    Let
		$$\sG:=\left\{\gamma:U\to \P^N: \gamma \;\text{is holomorphic}\right\}$$
		be the set of graphs over $U$. We endow $\sG$ with the topology
		of local uniform convergence.  The family $f$ induces a natural continuous map $\sF:\sG\to \sG$ given by $(\sF\gamma)(t):=f_t(\gamma(t))$. For each $t\in U$, there is a canonical continuous projection $\pi_t$ from $\sG$ to $\P^N(\C)$ given by $\pi_t(\gamma):=\gamma(t)$.
		
		In \cite{berteloot2018dynamical}, the authors considered the following closed subset of $\sG$: let 
		$$\sJ:=\left\{\gamma:U\to \P^N: \gamma \;\text{is holomorphic and}\; \gamma(t)\in J_t \;\text{for every} \;t\in U\right\}.$$ 
		
		
		\begin{defi}
			An equilibrium web (over $U$) is a probability measure $\sigma$ on $\sJ$  such that: 
			
			(1) $\sigma$ is $\sF$-invariant and has compact support;
			
			(2) For each $t\in U$, the probability measure $(\pi_t)_* \sigma$ is equal to $\mu_t$.
		\end{defi}

		Note that the assumption  that $\sigma$ is $\sF$-invariant and has compact support implies that for every $\gamma\in \supp \sigma$,  the orbit $\{\sF^n(\gamma)\}_{n\geq 0}$ is contained in a compact subset of $\sG$. Hence $\left\{f^n|_\Gamma\right\}_{n\geq 1}$ is a normal family, where $\Gamma$ is the graph of $\gamma$; equivalently, $[\Gamma]\wedge T_f=0$. For a proof, see \cite[Remark 2.13]{Gauthier2019} or \cite[Theorem 2.8]{dujardin2014bifurcation}.
		
		\medskip
		
		Next we introduce the notion of {\em  equilibrium graph}. We will show that the existence of  equilibrium graph implies the existence of equilibrium web. 
		
		Let $g$ be an endomorphism of $\P^N$ over $\C$ of degree $d\geq 2$.  A point $x\in \P^N(\C)$ is called {\em typical} if the Birkhoff average $$\frac{1}{n}\sum_{i=0}^{n-1} \delta_{g^i(x)}\to \mu_g$$ in weak-* topology, where $\delta_y$ is the Dirac mass at $y$ and $\mu_g$ is the maximal entropy measure of $g$.
		
		Since $\mu_g$ is ergodic \cite[Theorem 1.82]{dinh2010dynamics}, by Birkhoff ergodic theorem, $\mu_g$-a.e. points are typical for $g$.
		\begin{defi}\label{defi: graph}
			An equilibrium graph (over $U$) is a holomorphic map $\gamma:U\to \P^N(\C)$ such that: 
			
			(1) There is a dense subset $E\subset U$ such that for every $t\in E$, $\gamma(t)$ is a typical point of $f_t$;
			
			(2) $[\Gamma]\wedge T_f=0$, where $\Gamma\subset U\times \P^N(\C)$ is the graph of $\gamma$.
		\end{defi}
		
		\begin{lemma}\label{lem:equigraph}
			Let $f$ be a holomorphic family of endomorphisms of $\P^N(\C)$ of degree $d\geq 2$, parametrized by a Hermitian manifold $\La$. Let $U\subset \La$ be an open subset. If there exists an equilibrium graph $\gamma$ over $U$, then there exists an equilibrium web $\sigma$ over $U$.
		\end{lemma}
		\begin{proof}
			For each $n\geq 0$, let $\gamma_n:U\to \P^N(\C)$ be the holomorphic map defined by $\gamma_n(t):=f_t^n(\gamma(t))$. We define a probability measure $\sigma_n$ on $\sG$:
			$$\sigma_n:=\frac{1}{n}\sum_{i=0}^{n-1}\delta_{\gamma_i}.$$
			It is well known that  $[\Gamma]\wedge T_f=0$ on $U$  implies that $\left\{ \gamma_n\right\}_{n\geq 1}$ is a normal family over $U$, see \cite[Remark 2.13]{Gauthier2019}.  This implies that there is a compact subset $\sK\subset \sG$ such that $\supp \sigma_n\subset \sK$ for every $n\geq 1$.  After passing to a subsequence, we may assume that $\sigma_n$ converges to a probability measure $\sigma$ supported on $\sK$. Moreover $\sigma$ is $\sF$-invariant. We will show that $\sigma$ is an equilibrium web over $U$.  It remains to show that  for every $t\in U$ we have 
			
			(1) $(\pi_t)_*(\sigma)=\mu_t$;
			
			(2) For every $\delta\in \supp \sigma$, $\pi_t(\delta)\in \sJ_t$.
			
			\medskip
			
			For every $t\in U$ and $n\geq 1$, we have $(\pi_t)_*(\sigma_n)=\frac{1}{n}\sum_{i=0}^{n-1}\delta_{\gamma_i(t)}$. Since $\pi_t$ is continuous, $\sigma_n\to \sigma$ implies that $\frac{1}{n}\sum_{i=0}^{n-1}\delta_{\gamma_i(t)}\to (\pi_t)_*(\sigma)$.  When $t\in E$, we have $\frac{1}{n}\sum_{i=0}^{n-1}\delta_{\gamma_i(t)}\to \mu_t$. This implies  $(\pi_t)_*(\sigma)=\mu_t$ when $t\in E$.  Since $\mu_t$ is the intersection of $T_f^N$ with the fiber of $t$, the fact that $T_f$ has continuous local potentials implies that the map $t\mapsto \mu_t$ from $U$ to $\sP(\P^N(\C))$ is continuous, where $\sP(\P^N(\C))$ denotes the space of probability measures on $\P^N(\C)$ equipped with the weak-* topology. On the other hand, since $\sigma$ has compact support, the map $t\mapsto (\pi_t)_*(\sigma)$ from $U$ to $\sP(\P^N(\C))$ is also continuous.  Since $E$ is dense in $U$,  and these two maps coincide on $E$, the two maps are equal. This proves (1). 
			
			Since for every $t\in U$, $\pi_t: \supp \sigma\to \P^N(\C)$  is continuous and $(\pi_t)_*(\sigma)=\mu_t$, we have $\pi_t(\supp \sigma)=\supp(\mu_t)=\sJ_t$.  This proves (2), and the proof is complete.
		\end{proof}
		
		
		\medskip
		
		By \cite[Theorem 1.2]{bianchi2026equilibrium}, we have:

		\begin{prop}\label{prop: high-dim stab}
			Let $f$ be a holomorphic family of endomorphisms of $\P^N(\C)$ of degree $d\geq 2$, parametrized by a Hermitian manifold $\La$. Let $s\in \La$. Then the following are equivalent:
			
			(1) $s\in \Stab(f)$;
			
			(2) There exists an open set $U$ containing $s$ and an equilibrium web over $U$.
		\end{prop}


		\medskip
		
		\subsection{Uniformly woven currents and bifurcation theory}
		Let $f$ be a holomorphic family of endomorphisms of $\P^N(\C)$ of degree $d\geq 2$, parametrized by the unit disk $\D$.  For every $t\in \D$, let $R_t$ be the current of integration on the fiber $\pi_1^{-1}(t)$. Let $\Omega\subset \D\times \P^N(\C)$ be an open subset.   
		
		
		\begin{thm}\label{thm:wovenbifur}
			Let $f$ be a holomorphic family of endomorphisms of $\P^N(\C)$ of degree $d\geq 2$, parametrized by the unit disk $\D$.   Let $S$ be a bidimension $(1,1)$ uniformly woven current  on $\Omega$ such that
			
			(1) $S\wedge T_f=0$;
			
			(2) For every $t\in \D$ such that $S\wedge R_t$ is admissible, $S\wedge R_t\leq \mu_t$, where $\mu_t$ is the maximal entropy measure of $f_t$;
			
			(3) There exists $s\in \D$ such that $S\wedge R_s$ is admissible and $S\wedge R_s>0$.
			
			\smallskip
			
			Then $s\in \Stab(f)$.
			
		\end{thm}
		
		\begin{proof}
			We write $S$ as $$S=\int_{\sV} [\Gamma_a]\;d\nu(a),$$
			where $\sV$ is a bounded volume space, and $\nu$ is a finite measure on $\sV$, see Definition \ref{defi:uniformwoven}.
			
			Let $\Omega'\subset \Omega$ be a relatively compact open subset  such that the restriction $S':=S|_{\Omega'}$ still satisfies the three conditions in Theorem \ref{thm:wovenbifur}.   Let $K:=\overline{\Omega'}$, which is compact. 
			
			
			Since we only consider curves in this paragraph, write $\sV^*_K:=\sV^*_{1,K}$. By Proposition \ref{prop:irreducible}, there exists a finite measure $\nu_1$ on $\sV^*_{K}$  such that
			$$S'=\int_{\sV^*_{K}} [\Gamma_a]\;d\nu_1(a).$$

			Since $\nu_1$ is a  finite measure on a locally compact separable metrizable space,  $\nu_1$ is a Radon measure. Replacing $\nu_1$ by its restriction to a compact subset of $\sV^*_K$ of positive $\nu_1$-measure, chosen so that $S'\wedge R_s$ remains non-zero,  we may assume $\nu_1$ has compact support in $\sV^*_{K}$ and  $S'$ still satisfies the three conditions in Theorem \ref{thm:wovenbifur}. Now for each $a\in \supp \nu_1$, $\Gamma_a$ is a  closed irreducible analytic curve in $\Omega'$.
			
			Since $S'\wedge R_s>0$,  there is a positive $\nu_1$-measure subset $\sW\subset \sV^*_{K}$ such that for every $a\in \sW$,  $\Gamma_a$ intersects $ \pi_1^{-1}(s)$ properly.   Pick $a\in \sW$. Let $D\subset \D$ be a small disk containing $s$ such that for every $t\in D$, $\Gamma_a$ intersects $ \pi_1^{-1}(t)$ properly.   A parameter $t\in D$ is called {\em bad} if either there exists   a singular point $x\in \Gamma_a$ such that $\pi_1(x)=t$, or there exists   a ramification point  $x$ of the map $\pi_1|_{\Gamma_a}:\Gamma_a\to D$ such that $\pi_1(x)=t$.  Bad parameters form a finite set in $D$.
			
			Assume by contradiction that $s\in \Bif(f)$.  Since $\Bif(f)$ is the support of the bifurcation measure, which has continuous local potentials, $\Bif(f)$  has no isolated point.  Thus there exists $s'\in \Bif(f)\cap D$ such that $s'$ is a good parameter, i.e. $\Gamma_a$  intersects $\pi_1^{-1}(s')$ transversely. By the stability of transverse intersections,   replacing $D$ by a smaller disk centered at $s'$,  we can choose $\sB$ to be a small neighborhood of $a$ such that for every $b\in \sB$ and every $t\in D$,  $\Gamma_b$ intersects $ \pi_1^{-1}(t)$ transversely.  By shrinking $\Omega'$, we may further assume the intersection $\Gamma_b\cap \pi_1^{-1}(t)$ is always unique for $b\in \sB$ and $t\in D$.  Replacing $\nu_1$ by its restriction to $\sB$ and replacing $s$ by $s'$, the new uniformly woven current $S'$ and the point $s'$ still satisfy the three conditions in Theorem \ref{thm:wovenbifur}.  Moreover $S'\wedge R_t>0$ is admissible for every $t\in D$.  For each $a\in \supp \nu_1$, $\Gamma_a$ is the graph of a holomorphic map $\gamma_a:D\to \P^N(\C)$. 
			
			Let $E$ be a countable dense subset of $D$.  For every $t\in E$,  the set of typical points for $f_t$ has full $\mu_t$ measure. By condition (2), $0<S'\wedge R_t\leq \mu_t$.  Hence there exists a full $\nu_1$-measure subset $Y_t\subset \sV^*_{K}$ such that for every $a\in Y_t$,  $\gamma_a(t)$ is a typical point for $f_t$.  Since $E$ is countable, there exists a full $\nu_1$-measure subset $Y\subset \sV^*_{K}$ such that for every $a\in Y$ and $t\in E$,  $\gamma_a(t)$ is a typical point for $f_t$.   By Proposition \ref{prop: uniform intersection} and condition (1), there exists a full $\nu_1$-measure subset $Z\subset \sV^*_{K}$ such that $[\Gamma_a]\wedge T_f=0$ for every $a\in Z$.  We pick a $\Gamma\in Y\cap Z$. Let $\gamma:D\to \P^N(\C)$ be the holomorphic map such that $\Gamma$ is the graph of $\gamma$.  Then $\gamma$ is an equilibrium graph. By  Lemma \ref{lem:equigraph} and Proposition \ref{prop: high-dim stab}, $s'\in \Stab(f)$, which contradicts $s'\in \Bif\cap D$. This completes the proof.
		\end{proof}

		\section{Review of slicing theory}\label{sec:slicing}
		Let $X$ be a Hermitian manifold of dimension $m$ and  $(B,\omega_B)$ be a Hermitian manifold of dimension $l$.  Let $\pi:X\to B$ be a surjective   holomorphic map.  Let $\pi':X\times \P^N(\C)\to B$  be the surjective holomorphic map given by $\pi'(t,x):=\pi(t)$.  We recall Federer's slicing theory \cite[Section 4.3]{MR257325} applied in our setting, see also Dinh-Sibony \cite[Page 280]{dinh2010dynamics}.  Let $y$ denote the coordinates in a chart of $B$.  Let $S$ be a positive closed bidimension $(m,m)$ current on $X\times \P^N(\C)$.  Let $\psi$ be a positive smooth function on this chart with compact support  such that $\int \psi \;\omega_B^l=1$.  Define $\psi_\ep(y):=\ep^{-2l}\psi(\ep^{-1}y)$ and $\psi_{\theta,\ep}(y):=\psi_\ep(y-\theta)$. The measure $\psi_{\theta,\ep}\omega_B^l$ converges in the weak-* sense to the Dirac mass at $\theta$.  We say that the {\em slicing} $S_\theta$ at $\theta$ with respect to $\pi'$ is well defined as a positive closed bidimension $(m-l,m-l)$ current on $X\times \P^N(\C)$, if for every function $\psi$ satisfying above properties and every smooth test form $\Phi$  of bidegree $(m-l,m-l)$  on $X\times \P^N(\C)$,  we have
		$$\left\langle S_\theta, \Phi\right\rangle=\lim_{\ep\to 0}\left\langle S\wedge (\pi')^*(\psi_{\theta,\ep}\omega_B^l), \Phi\right\rangle.$$
		By slicing theory, $S_\theta$ is well defined for Lebesgue-a.e. point $\theta\in B$.  Slicing satisfies the following two properties. 
		
		(1) Assume  the slicing $S_\theta$ is well defined at $\theta$. If $\Omega$ is a smooth form on $X\times \P^N(\C)$, then the slicing $(S\wedge \Omega)_\theta$ is also well defined and equals $S_\theta\wedge \Omega_\theta$ as a current on $(\pi')^{-1}(\theta)=\pi^{-1}(\theta)\times \P^N(\C)$, where $\Omega_\theta$ is the restriction of $\Omega$ to $\pi^{-1}(\theta)\times \P^N(\C)$.
		
		(2)   Let $\rho$ be a smooth probability measure on $B$.  Then
		$$S\wedge (\pi')^*\rho=\int S_\theta \;d\rho(\theta).$$
		

		\section{Background on height theory}\label{sec:height}
		Throughout this subsection, we follow the language and conventions of \cite{yuan2021}.
		Let $F$ be a number field.
		For any quasi-projective variety $V$ over $F$ and $v\in M_F$, we denote by $V_v^{\an}$ the analytification of $V$ over $F_v.$
		When $v$ is archimedean, $V_v^{\an}$ is the complex analytic variety $V(\C)$ with the euclidean topology induced by $v.$ When $v$ is non-archimedean, we use Berkovich's analytification. In our paper, we will only concern the archimedean case.
		

		Let $\pi: U\to X$ be a flat and projective morphism of relative dimension $n$
		of quasi-projective varieties over $F$. Set $m:=\dim X+1$ and  $K=F(X)$.
		Let $Y\to \Spec K$ be the generic fiber of $\pi$.
		
		\subsection{Moriwaki height and geometric height}\label{Sec_mwheight}
		Let $\overline{H}$ be any element in $\widehat{\mathrm{Pic}}(K/\Z)_{\mathrm{nef,\Q}}$. For any point $x \in Y(\overline{K})$ and $\overline{L}\in \widehat{\mathrm{Pic}}(U/\Z)_{\mathrm{nef,\Q}}$, define the \textit{Moriwaki height} of $x$ for $\overline{L}$ and $\overline{H}$ by
		\[
		h_{\overline{L}}^{\overline{H}}(x) := h_{\overline{L}}^{\overline{H}}(x') := \frac{\overline{L}|_{x'} \cdot \overline{H}^{m-1}}{\deg(x')} \in \mathbb{R}.
		\]
		Here $x'$ is the scheme theoretic image of $x$. This gives a height function
		\[
		h_{\overline{L}}^{\overline{H}}: Y(\overline{K}) \longrightarrow \mathbb{R}.
		\]
		More generally, for any closed
		$\overline{K}$-subvariety $Z$ of $Y$, the  \textit{Moriwaki height}  of $Z$ for $\tilde{L}$ and $\tilde{H}$ is
		\[
		h_{\overline{L}}^{\overline{H}}(Z) := h_{\overline{L}}^{\overline{H}}(Z') := \frac{\left(\overline{L}|_{Z'}\right)^{\dim Z + 1} \cdot \overline{H}^{m-1}}{(\dim Z + 1)\deg_{\tilde{L}}(Z'/K)}.
		\]
		Here $Z'$ is the scheme theoretic image of $Z$.
		
		\medskip
		
		Next we recall the geometric height.
		Let $\tilde{H}$ be the image of $\overline{H}$ in $\widehat{\Pic}(X/F)_{\mathrm{nef}}$ and $\tilde{L}$ be the image of $\overline{L}$ in $\widehat{\Pic}(U/F)_{\mathrm{nef}}.$
		We define the \textit{geometric height} of $x$ for $\tilde{L}$ and $\tilde{H}$ by
		\[
		h_{\tilde{L}}^{\tilde{H}}(x) := h_{\tilde{L}}^{\tilde{H}}(x') := \frac{\tilde{L}|_{x'} \cdot \tilde{H}^{m-2}}{\deg(x')} \in \mathbb{R}.
		\]
		Here $x'$ is the scheme theoretic image of $x$. This gives a geometric height function
		\[
		h_{\tilde{L}}^{\tilde{H}}: Y(\overline{K}) \longrightarrow \mathbb{R}.
		\]
		
		More generally, for any closed
		$\overline{K}$-subvariety $Z$ of $Y$, the \textit{geometric height} of $Z$ for $\overline{L}$ and $\overline{H}$ is
		\[
		h_{\tilde{L}}^{\tilde{H}}(Z) := h_{\tilde{L}}^{\tilde{H}}(Z') := \frac{\left(\tilde{L}|_{Z'}\right)^{\dim Z + 1} \cdot \tilde{H}^{m-2}}{(\dim Z + 1)\deg_{\tilde{L}}(Z'/K)}.
		\]
		Here $Z'$ is the scheme theoretic image of $Z$.

		\medskip

		We have the following comparison between the Moriwaki and the geometric heights.
		\begin{prop}\label{prop:two height}
			Assume that $\overline{H}$ is big and nef.
			There exists a constant $C>0$ such that  for every $\overline{K}$-subvariety $Z$, we have
			$$h_{\tilde{L}}^{\tilde{H}}(Z)\leq Ch_{\overline{L}}^{\overline{H}}(Z).$$
		\end{prop}
		This result is  well-known for experts. 
		A proof can be found in \cite[Proposition 2.7]{ji2026geometric}. In \cite[Proposition 2.7]{ji2026geometric}, $\pi: U\to X$ is an abelian scheme, however its proof works in general. 
		
		\medskip
		
		The following result gives an analytic interpretation  of the geometric height. 
		The version we need was proved by Yuan-Zhang in \cite[Lemma 5.4.4]{yuan2021}, see also \cite{Gauthier2019,Guo2025}. 
		In our paper, we only need the case where $v$ is archimedean.
		\begin{thm}\label{thmgeomheiganaly} For any place $v$ of $F$,
			We have 
			
			\[
			h_{\tilde{L}}^{\tilde{H}}(Z)=
			\frac{\int_{(Z')^{\an}_v}
				c_1(\overline{L})_v^{\dim Z+1}\wedge c_1(\overline{H})_v^{m-2}}
			{(\dim Z + 1)\deg_{\tilde{L}}(Z'/K)}.
			\]
			
		\end{thm}

		\subsection{The relative equidistribution theorem}\label{Sec_equi}
		A sequence of points $x_i, i\geq 0$ in $Y(\overline{K})$ is called \emph{generic} if $x_i, i\geq 0$ tends to the generic point of $Y$ in the Zariski topology.
		We recall Yuan-Zhang's relative equidistribution theorem.
		\begin{thm}\cite[Theorem 6.1]{Moriwaki2000}\cite[Theorem 5.4.6]{yuan2021}\cite{CM24}\label{thmequiyuanzhang}
			Let $\overline{H}$ be an element of $\widehat{\Pic}(X/\Z)_{\mathrm{nef}}$ that $\overline{H}$
			is nef and $\tilde{H}^{m-1}>0$.
			Here $\tilde{H}$ is the image of $\overline{H}$ in $\widehat{\Pic}(X/F)_{\mathrm{nef}}$.
			Let $\overline{L}$ be an element of $\widehat{\Pic}(U/\Z)_{\mathrm{nef}}$ such that $\deg_{\tilde{L}}(Y/K)>0$. Here $\tilde{L}$ is the
			image of $\overline{L}$ under the canonical composition
			\[
			\widehat{\mathrm{Pic}}(U/\Z)_{\mathrm{nef}} \longrightarrow \widehat{\mathrm{Pic}}(Y/\Z)_{\mathrm{nef}} \longrightarrow \widehat{\mathrm{Pic}}(Y/F)_{\mathrm{nef}}.
			\]
			
			Let $x_i, i\geq 0$ be a generic sequence in $Y(\overline{K})$ such that $h_{\overline{L}}^{\overline{H}}(x_i)$ converges to
			$h_{\overline{L}}^{\overline{H}}(Y)$. Then for any place $v$ of $F$, there is a weak convergence
			\[
			\frac{1}{\deg(x_i)} \delta_{\Delta(x_i),v} c_1(\pi^* \overline{H})_v^{m-1} \longrightarrow \frac{1}{\deg_{\tilde{L}}(Y/F)} c_1(\overline{L})_v^n c_1(\pi^* \overline{H})_v^{m-1}
			\]
			of measures on $U_v^{\mathrm{an}}$. Here $\Delta(x_i) \subset U$ denotes the Zariski closure of the image
			of $x_i$ in $U$, and $\delta_{\Delta(x_i),v}$ denotes the Dirac current of $\Delta(x_i)_{K_v}^{\mathrm{an}}$ in $U_v^{\mathrm{an}}$.
		\end{thm}

		\begin{rem}
			The equidistribution theorem in this relative setting was first proved in \cite[Theorem 6.1]{Moriwaki2000} under an additional hypothesis called \emph{Moriwaki condition} i.e. 
			$\overline{H}^m=0$.  A more general result for quasi-projective varieties was proved by Yuan--Zhang in \cite[Theorem 5.4.6]{yuan2021}. Chen--Moriwaki removed the Moriwaki condition in \cite[Theorem F]{CM24}, proving a more general form of the theorem in the setting of adelic curves. In this article, we only require the above special case, corresponding to a polarized adelic structure in the sense of Chen--Moriwaki. 
		\end{rem}

		The following definition is given by Yuan-Zhang in \cite[Section 5.4.1]{yuan2021}.
		\begin{defi}The sequence $x_i, i\geq 0$ is called \emph{small} if for any  $\overline{H}\in\widehat{\mathrm{Pic}}(K/\Z)_{\mathrm{nef,\Q}}$,  $h_{\overline{L}}^{\overline{H}}(x_i)$ converges to $h_{\overline{L}}^{\overline{H}}(Y).$
		\end{defi}

		\subsection{Polarized dynamical systems}\label{subsectionpolarizedend}
		Let $X$ be a quasi-projective variety over $F$. Let $K:=F(X)$. Let $f: \P^N_K\to \P^N_K$ be an endomorphism of algebraic degree $d\geq 2$.
		After shrinking $X$, we may extend $f$ to a finite endomorphism $X\times\P^N\to X\times \P^N$, we still denote it by $f$. Let $\pi_1:X\times\P^N\to X$ and $\pi_2:X\times\P^N\to \P^N$ be the first and the second projections.
		Let $L_2:=\pi_2^*(O_{\P^N}(1))$. After shrinking $X$, we may assume that $f^*L_2=dL_2.$

		\medskip

		The triple $(\pi_1: X\times\P^N\to X, f, L_2)$ is a \emph{$d$-polarized dynamical system over $X$} in the sense of Yuan-Zhang \cite{yuan2021}.
		In \cite[Theorem 6.1.1]{yuan2021}, by Tate's limiting argument, Yuan-Zhang constructed an adelic line bundle $\overline{L_2}_f\in \widehat{\Pic}(X\times \P^N/F)_{\rm nef}$ extending $
		L_2$ such that $f^*\overline{L_2}_f=d\overline{L_2}_f.$

		\medskip
		
		Let $V$ be an irreducible closed subvariety of $X\times\P^N$ over $F$ (resp. $\overline F$) of dimension $m=\dim X$. We call $V$ \emph{horizontal} if $\pi_1:V\to X$ (resp. $\pi_1:V\to X_{\overline F}$) is generically finite and surjective. The {\em degree} of $V$ is by definition the topological degree of $\pi_1|_V$.
		Let $x_V\in \P^N(\overline{K})$ be any geometric generic point of $V.$
		\begin{defi}[Canonical height]
			The {\em canonical height} of a horizontal subvariety $V$ is defined as 
			$$\hat{h}^{\overline{H}}_f(V):=h^{\overline{H}}_{\overline{L_2}_f}(x_V).$$
		\end{defi}
		
		\medskip
		
		\begin{defi}[Small horizontal subvarieties]
			A sequence of horizontal subvarieties $V_n, n\geq 0$ (over $F$ or $\overline{F}$)
			is called \emph{small for $f$} if for any  $\overline{H}\in\widehat{\mathrm{Pic}}(K/\Z)_{\mathrm{nef,\Q}}$, we have $\hat{h}_{f}^{\overline{H}}(V_n)\to 0$ as $n\to \infty.$
		\end{defi}
		
		
		\medskip

		Now we fix an embedding $F\hookrightarrow \C$. It gives an archimedean place $v\in M_F.$
		Let $T_f$ be the relative Green current of $f: (X\times \P^N)(\C)\to (X\times \P^N)(\C)$, see Section \ref{section:bifur}.
		We have $T_f=c_1(\overline{L_2}_f)_v.$
		
		\medskip
		
		Let $V$ be a horizontal subvariety over $F$ or $\overline{F}$. Let $x_V$ be any geometric generic point of $V$.
		Let $\overline{H}\in\widehat{\mathrm{Pic}}(K/\Z)_{\mathrm{nef,\Q}}.$
		\begin{defi}[Canonical geometric height]
			The {\em canonical geometric height} of $V$ is defined as 
			$$\hat{h}^{\tilde{H}}_\geom(V):=h^{\tilde{H}}_{\tilde{L_2}_f}(x_V).$$
		\end{defi}

		After shrinking $X$, assume that $\overline{H}\in \widehat{\Pic}(X/F)_{\rm nef}.$
		Let $\omega_X:=c_1(\overline{H})_v$. It is a positive closed $(1,1)$-current on $\overline{X}(\C)$ with continuous potential.
		Set $\omega_1:=\pi_1^*\omega_X.$
		By Theorem \ref{thmgeomheiganaly}, we have 
		
		\begin{cor}\label{cor: geometricheightanalytic}
			
			The following holds:
			$$\hat{h}^{\tilde{H}}_\geom(V)=\frac{1}{\deg V} \int_{X\times \P^N(\C)} [V(\C)]\wedge T_f\wedge \omega_1^{m-1}.$$ 
		\end{cor}

		\subsection{Adelic bifurcation classes and stable families}\label{subsec:adelic-bifurcation}
		
		We recall here an algebro-geometric interpretation of stability in terms of the geometric part of the bifurcation adelic line bundle. We keep the notation and terminology of Yuan-Zhang \cite{yuan2021}; in particular, we use adelic line bundles, their geometric parts, and their intersection theory without recalling the general construction.
		
		Let $\sM_{d,N}$ be the moduli space of degree $d$ endomorphisms of $\P^N$. On the parameter space $\End_{d,N}$ there is a tautological polarized dynamical system; the bifurcation adelic line bundle obtained from it by Deligne pairing is $\PGL_{N+1}$-invariant and descends to a nef adelic line bundle
		$$\overline{L}_{\bif}\in \widehat{\Pic}(\sM_{d,N}/\C)_{\rm nef,\Q}$$
		whose archimedean curvature is the bifurcation current on $\sM_{d,N}(\C)$. Let $\tilde{L}_{\bif}$ be the geometric part of $\overline{L}_{\bif}$.
		
		We will use the following facts: $\tilde{L}_{\bif}$ is nef in the sense of Yuan-Zhang, and the support of the archimedean curvature current of $\overline{L}_{\bif}$ is the bifurcation locus. Moreover, by Yuan-Zhang \cite[Theorem 6.3.6 and Lemma 5.4.4]{yuan2021}, the top self-intersection of $\tilde{L}_{\bif}$ is the total mass of the bifurcation measure; see also Guo \cite[Theorem 1.2]{Guo2025} for a local integration formula in the quasi-projective setting. This mass is positive by \cite[Proposition 2.8]{GauthierTaflinVigny2026}. Hence $\tilde{L}_{\bif}$ is big. For an irreducible algebraic curve $C\subset \sM_{d,N}$, the adelic intersection number
		$$\deg_{\tilde{L}_{\bif}}(C):=\tilde{L}_{\bif}\cdot C$$
		is well-defined. If a one-parameter algebraic family $F$ lies over $C$, in the sense that the image of its moduli map is a Zariski open subset of $C$, then $F$ is stable if and only if $\deg_{\tilde{L}_{\bif}}(C)=0$. Indeed, stability is equivalent to the vanishing of the pullback bifurcation current on $C$, and by Yuan--Zhang \cite[Theorem 6.3.6 and Lemma 5.4.4]{yuan2021} (or Guo \cite[Theorem 1.2]{Guo2025}) this degree is the total mass of that positive current.
		
		\begin{prop}\label{prop:general-not-stable}
			There exists a proper Zariski closed subset $Z\subset \sM_{d,N}$ such that if $C\subset \sM_{d,N}$ is an irreducible algebraic curve not contained in $Z$, then every non-isotrivial one-parameter algebraic family lying over $C$ has non-empty bifurcation locus. In particular, every non-isotrivial one-parameter algebraic family whose moduli curve passes through a general point of $\sM_{d,N}$ is not stable.
		\end{prop}
		
		\begin{proof}
			Since $\tilde{L}_{\bif}$ is big, by the definition of bigness for adelic line bundles on quasi-projective varieties, after passing to a projective model $\rho:\sM'\to \sM_{d,N}$,  we may choose an ample $\Q$-line bundle $A$ on $\sM'$ and an effective geometric adelic class $E$ such that $\tilde{L}_{\bif}\geq A+E$ in the sense of Yuan--Zhang. We choose an effective representative $D$ of $E$ on a projective model dominating $\sM'$, and let $Z\subset \sM_{d,N}$ be the image of $\Supp(D)$ together with the corresponding exceptional locus. Then $Z$ is contained in a proper Zariski closed subset of $\sM_{d,N}$. We may replace $Z$ by this closed subset. Thus we may assume that $Z$ is proper and Zariski closed.
			
			Let $C\subset \sM_{d,N}$ be an irreducible algebraic curve not contained in $Z$, and let $C'$ be its strict transform on the model carrying $D$. Then $C'$ is not contained in $\Supp(D)$, hence $D\cdot C'\geq 0$. Since the pullback of $A$ to this model is still big and nef and has positive degree on $C'$, we have $A\cdot C'>0$. Thus  we have 
			
			$$\deg_{\tilde{L}_{\bif}}(C)=\tilde{L}_{\bif}\cdot C\geq A\cdot C'>0.$$
			If a one-parameter algebraic family lies over $C$, the bifurcation current  has positive mass, so the family has non-empty bifurcation locus. This proves the proposition.
		\end{proof}
		
		\medskip

		In dimension $N=1$, McMullen's rigidity theorem \cite{McMullen1987} implies that the only non-isotrivial stable families of rational maps are the flexible Latt\`es families. It is natural to ask:
		\begin{que}
			Classify the irreducible algebraic curves $C\subset \sM_{d,N}$ such that $\deg_{\tilde{L}_{\bif}}(C)=0$, or equivalently, such that one-parameter algebraic families lying over $C$ are stable. More generally, which positive-dimensional subvarieties of $\sM_{d,N}$ support stable algebraic families?
		\end{que}
		

		%
		%

		\section{Genus and gonality of small curves: \texorpdfstring{proof of Theorems \ref{thm: 1-dim Gonality}, \ref{thm: high-dim Genus} and \ref{thm:stable-dynatomic}}{proof of Theorems \ref{thm: 1-dim Gonality}, \ref{thm: high-dim Genus} and \ref{thm:stable-dynatomic}}}\label{sec:gonality-proof}
		\label{sec: finitely generated}
		Let $f$ be a one-parameter algebraic family of endomorphisms of $\P^N$ of algebraic degree $d\geq 2$ over $\C$, parametrized by a smooth quasi-projective curve $\La$.  Let $\overline{\La}$ be a smooth projective compactification of $\La$. 
		Let $K$ be a subfield of $\C$ which is finitely generated over $\Q$ such that $f$ is over $K$.

		Let $m-1$ be the transcendental degree of the field extension $K/\Q$, $m\geq 1$. There exist a number field $F$, a smooth irreducible projective variety $\overline{B}$ defined over $F$ of dimension $m-1$,   a  surjective flat morphism $\pi:\overline{X}\to \overline{B}$  between smooth projective varieties over $F$  with relative dimension $1$, and $\theta_0\in \overline{B}(\C)$ such that the specialization of $\overline{X}_\C$ at $\theta_0$ is $\overline{\La}$.     Moreover $\theta_0$  is {\em fully transcendental} with respect to $\overline{B}/F$ in the sense that the image of the generic point of $\Spec \C$ under 
		$$\Spec \C\xrightarrow{\;\theta_0\;} \overline{B}_\C\to \overline{B}$$
		is the generic point of $\overline{B}$.     Let $\pi':\overline{X}\times \P^N\to \overline{B}$  be the flat morphism  given by $\pi'(t,x):=\pi(t)$.

		The set of  non fully-transcendental  points in $\overline{B}(\C)$  is contained in a countable union of proper subvarieties, in particular it has Lebesgue measure zero. 
		
		The generic fiber of $\pi$ is connected  and smooth.    Let  $B$ be a Zariski open subset of $\overline{B}$ such that $\theta_0\in B(\C)$,  and for every $\theta\in B(\C)$, $\pi^{-1}(\theta)$ is connected  and smooth.  
		
		By abuse of notation, we let $f$ be the  family of endomorphisms of $\P^N$ over $F$, parametrized by a Zariski open subset  $X\subset \overline{X}$.  Let $\pi_1:\overline{X}\times \P^N\to \overline{X}$ and $\pi_2:\overline{X}\times \P^N\to \P^N$ be the two projections.   We have $\pi'=\pi\circ \pi_1$.
		
		\subsection{Relative arithmetic equidistribution}\label{sec: arithmetic equidistribution}
		The composition of embedding $F\hookrightarrow K\hookrightarrow \C$ gives a place $v\in M_F.$

		Let $L_X$ be a very ample line bundle on $\overline{X}$. Let $\overline{L_X}$ be a big and nef adelic line bundle in  $\widehat{\Pic}(X/F)_{\rm nef}$ whose underlying line bundle is $L_X$ such that $\omega_X:=c_1(\overline{L_X})_v$ is a K\"ahler form on $\overline{X}.$
		Let $L_2:=\pi_2^*(O_{\P^N}(1))$, $L_1:=\pi_1^*L_{X}$, $\omega_1:=\pi_1^*\omega_{X}$ and $\omega_2:=\pi_2^*\omega_{\P^N}.$ Then $\omega:=\omega_1+\omega_2$ is a K\"ahler form on $\overline{X}\times \P^N$ in $c_1(L_1+L_2)$.  
		For every horizontal subvariety $V$, the volume of $V$ with respect to $\omega$ is  $$\vol(V):=\int_{V(\C)}\omega^m=(V\cdot (L_1+L_2)^m).$$

		\begin{prop}\label{prop: deg-vol}
			There exists a constant $C>0$ such that for every irreducible horizontal subvariety $V$ (over $F$ or $\overline{F}$), we have
			$$\vol(V)\geq  \deg(V)\geq \frac{C}{(1+\hat{h}^{L_X}_\geom(V))^m}\vol (V).$$
		\end{prop}

		\begin{proof}
			Only need to treat the case over $F$.
			Pick a smooth projective compactification $Y$ of $X\times \P^N$ such that endomorphisms $\id$ and $f$ on $X\times \P^N$ extend to morphisms $p_1,p_2: Y\to \overline{X}\times \P^N.$ Denote by $\phi:=\pi_1\circ p_1: Y\to \overline{X}.$ 
			We have $p_2^*L_1=p_1^*L_1$ and $p_2^*L_2=dp_1^*L_2+E$, where $E$ is a divisor on $Y$ such that $\phi(\supp E)\neq \overline{X}.$ 
			As $L_{X}$ is ample on $\overline{X}$, there is an effective $\Q$-divisor $D$ on $\overline{X}$ representing $L_{X}$ such that $\phi(\supp E)\subseteq \supp D$.
			There is $C_1>0$ such that $C_1\phi^*D+E$ is effective.
			
			Let $V$ be a horizontal subvariety.  Denote by $\overline{V}$ the Zariski closure of $V$ in $\overline{X}\times\P^N$ and by $V'$ the Zariski closure of $V$ in $Y.$
			Write $$h_2(V):=(\overline{V}\cdot L_2\cdot L_1^{m-1})/\deg V.$$ 
			Tate's limiting argument shows that 
			\begin{equation}\label{equgeolim}\hat{h}^{L_X}_\geom(V)=\lim_{n\to \infty}h_2(f^n(V))/d^n.
			\end{equation}
			Observe that 
			$$(\overline{V}\cdot L_1^{m})/\deg V=\deg (L_X^m)\geq 1.$$
			As $L_2$ is nef, we have 
			$$\vol(V)=(\overline{V}\cdot (L_1+L_2)^m)\geq (\overline{V}\cdot L_1^m)\geq \deg V.$$
			This proves the first inequality in Proposition \ref{prop: deg-vol}.

			Denote by $f_*V$ the cycle $$f_*V:=\deg(f|_{V}: V\to f(V))f(V).$$ We have $$\deg (f_*V):=\deg(f|_{V}: V\to f(V))\deg(f(V))=\deg(V).$$
			We have
			\[
			h_2(f(V))=\deg(f_*V)^{-1}(\overline{f_*(V)}\cdot L_2\cdot L_1^{m-1})
			=\deg(V)^{-1}(V'\cdot p_2^*L_2\cdot p_1^*L_1^{m-1}).
			\]
			Since $C_1\phi^*D+E$ is effective, this gives
			\begin{align*}
				h_2(f(V))
				&\geq \deg(V)^{-1}\bigl(V'\cdot (d p_1^*L_2-C_1\phi^*D)\cdot p_1^*L_1^{m-1}\bigr)\\
				&=dh_2(V)-C_1\deg(V)^{-1}\bigl(V'\cdot \phi^*D\cdot p_1^*L_1^{m-1}\bigr).
			\end{align*}
			By the projection formula, the last intersection number equals
			\[
			(\overline V\cdot \pi_1^*D\cdot L_1^{m-1})=\deg(V)(D\cdot L_X^{m-1}),
			\]
			hence
			$$h_2(f(V))\geq dh_2(V)-C_2$$
			where $C_2:=C_1\deg L_{X}^m.$ It follows that 
			$$h_2(f(V))-\frac{C_2}{d-1}\geq d\left(h_2(V)-\frac{C_2}{d-1}\right).$$
			If $h_2(V)>\frac{C_2}{d-1}$, then $$h_2(f^n(V))\geq d^n\left(h_2(V)-\frac{C_2}{d-1}\right)+\frac{C_2}{d-1}.$$
			By (\ref{equgeolim}), we get $$\hat{h}^{L_X}_\geom(V)\geq h_2(V)-\frac{C_2}{d-1}.$$
			It follows that 
			$$h_2(V)\leq \frac{C_2}{d-1}+\hat{h}^{L_X}_\geom(V).$$
			In other words, $$(\overline{V}\cdot L_2\cdot L_1^{m-1})\leq \left(\frac{C_2}{d-1}+\hat{h}^{L_X}_\geom(V)\right)\deg V.$$
			By the Khovanskii--Teissier inequality \cite[Theorem 1.6.1]{Lazarsfeld}, applied after resolving the irreducible variety $\overline V$, the sequence
			$(\overline{V}\cdot L_2^i\cdot L_1^{m-i}), i=0,\dots,m$
			is log-concave. So for every $i=1,\dots, m$, we have $$(\overline{V}\cdot L_2^i\cdot L_1^{m-i})\leq \frac{(\overline{V}\cdot L_2\cdot L_1^{m-1})^i}{(\overline{V}\cdot L_1^m)^{i-1}}\leq \left(\frac{\frac{C_2}{d-1}+\hat{h}^{L_X}_\geom(V)}{(L_1^m)}\right)^i(L_1^m)\deg V.$$
			Then we have 
			$$\vol(V)=(\overline{V}\cdot (L_1+L_2)^{m})\leq \left(1+\frac{\frac{C_2}{d-1}+\hat{h}^{L_X}_\geom(V)}{(L_1^m)}\right)^m(L_1^m)\deg V.$$
			This concludes the proof.
		\end{proof}

		\medskip
		
		A sequence of irreducible horizontal subvariety  $V_n$  is called {\em generic} if for any proper subvariety $H\subset \La\times \P^N$,  there are only finitely many $V_n$ contained in $H$.     If a horizontal subvariety $V$ is over $\overline{k}$, we let  $O(V)$ be the {\em Galois orbit} of $V$ under $\Gal(\overline{k}/ k)$, which is a (maybe reducible)  horizontal subvariety. Whenever a numerical estimate above is applied to $O(V)$, it is applied componentwise to the irreducible Galois conjugates of $V$ and then summed. The following result is a direct consequence of Theorem \ref{thmequiyuanzhang} in our setting.
		\begin{thm}[Relative arithmetic equidistribution]\label{thm: equidistribution}
			In the above setting, let $(V_n)_{n\geq 1}$ be a generic sequence of  irreducible horizontal subvarieties over $\overline{k}$ which is small for $f$,  and let $\phi:(X\times \P^N)(\C)\to \R$ be a continuous function with compact support. Then 
			$$\lim_{n\to +\infty} \int_{X\times \P^N(\C)} \phi \;\frac{[O(V_n)]\wedge \omega_1^m}{\deg(O(V_n))}= \int_{X\times \P^N(\C)} \phi \; T_f^N\wedge \omega_1^m.$$
		\end{thm}


		\medskip
		
		\subsection{Genus and gonality of small curves}
		Come back to the setting of Theorem \ref{thm: equidistribution}.
		Let $(V_n)_{n\geq 1}$ be a generic sequence of  irreducible horizontal subvarieties over $\overline{k}$ which is small for $f$.

		Define $$S_n:=\frac{[O(V_n)]}{\deg(O(V_n))},$$ which is a positive closed current of bidimension $(m,m)$ on $X\times \P^N(\C)$. For every $s\in X(\C)$, let $R_s$ be the current of integration on the fiber $\pi_1^{-1}(s).$  For a positive closed current $S$ of bidimension $(m,m)$ on $\overline{X}\times \P^N(\C)$, the slicing $S_\theta$ (once it exists) is a positive closed current of bidimension $(1,1)$ supported in $(\pi')^{-1}(\theta)=\pi^{-1}(\theta)\times \P^N(\C)$.   See Definition \ref{def:admissible} for the meaning of admissibility  of $S_\theta\wedge R_s$ when $s\in \pi^{-1}(\theta)$ and $\pi^{-1}(\theta)$ is smooth. 
		
		%
		
		For a positive closed bidimension $(m,m)$ current $S$ on $\overline{X}\times \P^N(\C)$, the slicing $S_\theta$ with respect to $\pi'$ (once it exists) is a positive closed bidimension $(1,1)$ current supported in $(\pi')^{-1}(\theta)=\pi^{-1}(\theta)\times \P^N(\C)$.

		\medskip
		
		\begin{prop}\label{prop: converges}
			In the above setting, by passing to a subsequence, $S_n$ converges in the sense of current to a positive closed bidimension $(m,m)$ current on $X\times \P^N(\C)$.

		\end{prop}

		\begin{proof}
			Applying Proposition \ref{prop: deg-vol} to each irreducible component of the Galois orbit $O(V_n)$ and summing, the mass of $S_n\wedge \omega^m$ is uniformly bounded from above.  The compactness of the space of positive closed  currents with uniformly bounded  mass implies the result. 
		\end{proof}
		\medskip

		\begin{prop}\label{prop: stable}
			Let $S$ be such a limit as in Proposition \ref{prop: converges}. Let $\sigma$ be a smooth positive bidegree $(m-1,m-1)$  form on $X(\C)$.  Then
			$$S\wedge T_f\wedge \pi_1^*(\sigma)=0$$
			on $X\times \P^N(\C).$
			
		\end{prop}

		\begin{proof}
			Since the Moriwaki height is invariant under Galois conjugation,  $\hat{h}_f(V_n)\to 0$ implies that  $\hat{h}_f(O(V_n))\to 0$. By Proposition \ref{prop:two height},  we also have $\hat{h}_\geom(O(V_n))\to 0$.  By Corollary \ref{cor: geometricheightanalytic} we have $S_n\wedge T_f\wedge  \omega_1^{m-1}\to 0$. Since $T_f$ has continuous local potential and $S_n\to S$, we have $$S\wedge T_f\wedge  \omega_1^{m-1}=0.$$
			
			Let $\Omega$ be any relatively compact open subset of  $X(\C)$. Since $\omega_X$ is K\"ahler, there exists a constant $C=C(\Omega,\sigma)>0$ such that $\sigma\leq C\omega_X^{m-1}$ on $\Omega$. Hence $ \pi_1^*(\sigma)\leq C\omega_1^{m-1}$ on $\Omega\times \P^N(\C)$.  Thus $S\wedge T_f\wedge  \omega_1^{m-1}=0$ implies $S\wedge T_f\wedge \pi_1^*(\sigma)=0$ on $\Omega\times \P^N(\C)$. Since $\Omega$  is arbitrary, the conclusion holds.
		\end{proof}

		\medskip
		We refer to Definition \ref{def:admissible} for the meaning of admissibility  of $S_\theta\wedge R_s$ when $s\in \pi^{-1}(\theta)$ and $\pi^{-1}(\theta)$ is smooth. 
		\begin{prop}\label{prop: equidistribution}
			Let $S$ be such a limit as in Proposition \ref{prop: converges}.  If $\pi^{-1}(\theta)$ is smooth and the slicing $S_\theta$ with respect to $\pi'$ is well defined at $\theta$, then
			
			(1) $S_\theta\wedge T_f=0$ on $X\times \P^N(\C)$;
			
			(2) Assume  $S_\theta \wedge R_s$ is admissible for $s\in \pi^{-1}(\theta)$. Then 
			$$S_\theta \wedge R_s=\mu_s,$$
			on $X\times \P^N(\C)$, where $\mu_s$ is the maximal entropy measure of $f_s$. 
			
		\end{prop}
		\begin{proof}
			
			(1) Since the slicing $S_\theta$ with respect to $\pi'$ is well defined at $\theta$,  we can pick a sequence of smooth probability measure $\nu_n$ on $B(\C)$ such that $\nu_n\to \delta_\theta$ in the weak-* sense,  and $S\wedge (\pi')^* \nu_n\to S_\theta$ in the sense of currents.  Since $\sigma_n:=\pi^* \nu_n$ is a smooth positive  closed bidegree $(m-1,m-1)$ form on $X(\C)$, by Proposition \ref{prop: stable} we have 
			$$S\wedge (\pi')^*\nu_n\wedge T_f=S\wedge T_f\wedge (\pi')^*\nu_n=S\wedge T_f\wedge \pi_1^*\sigma_n=0.$$ 
			
			Since $S\wedge (\pi')^* \nu_n\to S_\theta$  and $T_f$ has continuous local potential, let $n\to+\infty$ we get $S_\theta\wedge T_f=0.$
			
			\medskip
			
			(2) By passing to a coordinate change, in a local chart $\D^m\subset X$, where $\D$ is the unit disk in the complex plane
			we can assume that $\pi|_{\D^m}$ is the projection $$\D^m\to \D^{m-1}$$ to the first $m-1$ coordinates. Let $s=0\in \D^m$, $\theta=0\in \D^{m-1}$. Then $\pi^{-1}(\theta)\cong \D$,  and $(\pi')^{-1}(0)\cong \D\times \P^N(\C)$.   Define $u(t,z):=\log |t|$. Then $u$ is a p.s.h. function on $(\pi')^{-1}(0)=\D\times \P^N(\C)$ and we have $R_0=dd^c  u$.  Using standard convolution of $\log |t| $ with some radial function $\rho_l$, we obtain a sequence of smooth p.s.h. functions $\phi_l$ on $\D$ that decrease to $\log |t|$. Let $u_l(t,z):=\phi_l(t)$, then we obtain a sequence of smooth p.s.h. functions $u_l$ on $(\pi')^{-1}(0)$ that decrease to $u$.  For fixed $l\geq 1$, choose $\Phi_l$ to be a bidegree $(m-1,m-1)$ smooth form on $\D^m$ such that the restriction of $\Phi_l$ to the fiber  $\pi^{-1}(0)$ is equal to $dd^c \phi_l$.  Let $\Omega_l:=\pi_1^*\Phi_l$, which is a bidegree $(m-1,m-1)$ smooth form on $\D^m\times \P^N(\C)$ such that the restriction of $\Omega_l$ to the fiber  ($\pi')^{-1}(0)$ is equal to $dd^c u_l$. Moreover $\Omega_l=\beta_l \omega_1^{m-1}$, where $\beta_l$ is a non-negative smooth function. 
			
			Apply Theorem \ref{thm: equidistribution} to  test functions of the form $\phi=\beta_l \tilde{ \phi} $ such that $\tilde{\phi}$ is smooth and has  compact support in $\D^m\times \P^N(\C)$, we get $$S_n\wedge \Omega_l\to T_f^N\wedge \Omega_l,$$ as $n\to +\infty$, for every  fixed $l\geq 1$. 
			
			Since $S_n\to S$ in the sense of current as $n\to +\infty$ and since $\Omega_l$ is a smooth form,  we have $S_n\wedge \Omega_l\to S\wedge \Omega_l$ as $n\to +\infty$.  This implies  
			\begin{equation}\label{eqn:4.1}
				S\wedge \Omega_l=T_f^N\wedge \Omega_l
			\end{equation}
			for every $l\geq 1$.

			Since $T_f$ has continuous local potential, the slicing $T_{f,\theta}$ is well defined  with respect to $\pi'$ for every $\theta\in B$, which is exactly the relative Green current for  $f$ restricted on $\pi^{-1}(t)\times \P^N(\C)$, viewed as a one-parameter holomorphic family.  By our assumption, the slicing $S_\theta$ exists. Since the slicing of the smooth form  $\Omega_l$ at the fiber $(\pi')^{-1}(\theta)$ is $dd^c u_l$, by (\ref{eqn:4.1}) we have 
			
			\begin{equation}\label{eqn:4.2}
				S_\theta \wedge dd^c u_l=T_{f,\theta}^N\wedge dd^c u_l \;\;\;\text{on $(\pi')^{-1}(\theta)$ },
			\end{equation}
			for every $l\geq 1$. 
			
			Since $u_l$ decreases to $u$, and  since by our assumption $dd^c u\wedge S_\theta$ is admissible, by \cite[Proposition A.30]{dinh2010dynamics}, we have 
			\begin{equation}\label{eqn:4.3}
				S_\theta \wedge dd^c u_l\to S_\theta \wedge dd^c u \;\;\;\text{on $(\pi')^{-1}(\theta)$ },
			\end{equation}
			when $l\to +\infty$. 
			
			Since $T_{f,\theta}$ has continuous local potential, $dd^c u\wedge T_{f,\theta}^N$ is also admissible, using \cite[Proposition A.30]{dinh2010dynamics} again, we have 
			
			\begin{equation}\label{eqn:4.4}
				T_{f,\theta}^N\wedge dd^c u_l\to T_{f,\theta}^N\wedge dd^c u \;\;\;\text{on $(\pi')^{-1}(\theta)$ },
			\end{equation}
			when $l\to +\infty$. 
			
			Recall that $R_s=dd^c u$ and $\mu_s=T_{f,\theta}^N\wedge R_s$. Combining (\ref{eqn:4.2}), (\ref{eqn:4.3}), and (\ref{eqn:4.4}) we get the desired result. This completes the proof.
		\end{proof}
		%
		
		\medskip

		For $\alpha\in B(\C)$, let $O(V_n)_\alpha$ be the specialization of $O(V_n)$ at $\alpha$, which is a pure dimension one Zariski closed subvariety supported in $(\pi')^{-1}(\alpha)$. Define $$S_{n,\alpha}:=\frac{[O(V_n)_\alpha]}{\deg(O(V_n))}.$$
		\begin{prop}\label{prop: sequence limit}
			In the setting of Proposition \ref{prop: equidistribution}.   For each $n\geq 1$, there exists $l_n\geq 1$ and a finite set $E_n\subset B(\C)$ consisting of fully transcendental  points, such that
			$$\frac{1}{\# E_n}\sum_{\alpha\in E_n}S_{l_n,\alpha }\to S_\theta$$
			as $n\to +\infty$.
			
		\end{prop}
		\begin{proof}
			Since $S_\theta$ is well defined, we can pick a sequence of smooth probability measure $\nu_n$ on $B(\C)$ such that $\nu_n\to \delta_\theta$ in the weak-* sense,  and $S\wedge (\pi')^* \nu_n\to S_\theta$.  Thus it suffices to show that  if $\nu$ is a fixed smooth probability measure on $B(\C)$,  for each $n\geq 1$, there exist $l_n\geq 1$ and a finite set $E_n\subset B(\C)$ containing fully transcendental  points, such that
			$$\frac{1}{\# E_n}\sum_{\alpha\in E_n}S_{l_n,\alpha }\to S\wedge (\pi')^* \nu$$
			as $n\to +\infty$.
			
			Since $S_n\to S$ when $n\to+\infty$, we have $S_n\wedge (\pi')^* \nu\to S\wedge (\pi')^* \nu$. In applying the following argument to the sequence $\nu_n$, we first choose $l_n$ large enough so that $S_{l_n}\wedge(\pi')^*\nu_n$ is close to $S\wedge(\pi')^*\nu_n$, and then approximate $\nu_n$ by discrete measures supported on fully transcendental points. Thus it suffices to show that  if $\nu$ is a fixed smooth probability measure on $B(\C)$ and $l\geq 1$ is a fixed integer,   for each $n\geq 1$,  there exists a finite set $E_n\subset B(\C)$ containing fully transcendental  points, such that
			$$\frac{1}{\# E_n}\sum_{\alpha\in E_n}S_{l,\alpha }\to S_l\wedge (\pi')^* \nu=\int S_{l,\alpha} \;d\nu(\alpha)$$
			as $n\to +\infty$.
			
			Since $\pi'$ is flat,  $S_{l,\alpha}$ varies continuously with respect to $\alpha$ in the sense of currents, see \cite[Page 152-153]{MR1111477}.  We can approximate $\nu$ by a sequence of discrete measures $\frac{1}{\# E_n}\sum_{\alpha\in E_n} \delta_\alpha$,  where $E_n$ is a finite set.  Since  the set of fully transcendental  points  has full Lebesgue measure, we can further choose $E_n$ such that every point in  $E_n$ is fully transcendental.  Thus we have 
			$$\frac{1}{\# E_n}\sum_{\alpha\in E_n}S_{l,\alpha }=\int S_{l,\alpha } \;d\left(\frac{1}{\# E_n}\sum_{\alpha\in E_n} \delta_\alpha\right)\to \int S_{l,\alpha} \;d\nu(\alpha)$$
			as $n\to +\infty$, which finishes the proof.
		\end{proof}
		
		%
		%

		\medskip

		\subsection{\texorpdfstring{Proof of Theorems \ref{thm: high-dim Genus}, \ref{thm:stable-dynatomic} and \ref{thm: 1-dim Gonality}}{Proof of Theorems \ref{thm: high-dim Genus}, \ref{thm:stable-dynatomic} and \ref{thm: 1-dim Gonality}}}
		\proof[Proof of Theorem \ref{thm: high-dim Genus}]  
		Let $f$ be a one-parameter algebraic family of endomorphisms of $\P^N$ of degree $d\geq 2$ over $\C$ parametrized by a smooth quasi-projective curve $\La$ and with non-empty bifurcation set.  As in the beginning of Section \ref{sec: finitely generated}, we view $f$ as  the  family of endomorphisms of $\P^N$ over a number field $k$, parametrized by a smooth irreducible quasi-projective variety $X$ of dimension $m$. Let $\theta_0\in B(\C)$ such that the specialization of $X_\C$ at $\theta_0$ is $\La$.   Let $V_n$ be the horizontal subvariety in $X\times \P^N$ induced by $\Gamma_n$. Then $V_n$ is a generic sequence. Let  $$S_n:=\frac{[O(V_n)]}{\deg(O(V_n))}, \;S'_n:=\frac{[O(V_n)]}{\vol(O(V_n))}.$$
		Since $(V_n)$ is small, the geometric heights of all irreducible components of $O(V_n)$ are bounded. Applying Proposition \ref{prop: deg-vol} componentwise and summing, the ratios $\vol(O(V_n))/\deg(O(V_n))$ are bounded above and below by positive constants. By Proposition \ref{prop: converges}, by passing to a subsequence, $S_n\to S$, $S'_n\to S'$, and there exists $c>0$ such that $S=cS'$. 
		
		For each $\theta\in B(\C)$, let $f_\theta$ be the one-parameter algebraic family parametrized by $\pi^{-1}(\theta)$.    Let $T_\bif$ be the bifurcation current of $f$,  which is a positive closed bidegree $(1,1)$ current on $X(\C)$ with continuous local potential, see Section \ref{sec:bifur current} for the definition.  By our assumption $\Bif(f_{\theta_0})\neq \emptyset$, which is equivalent to  say $$T_{\bif}\wedge [\pi^{-1}(\theta_0)]>0.$$
		Since $T_\bif$ has continuous local potential, there exists a neighborhood $U\subset B(\C)$ of $\theta_0$ such that for every $\theta\in U$,  $$T_{\bif}\wedge [\pi^{-1}(\theta)]>0,$$
		hence $\Bif(f_\theta)\neq \emptyset.$
		
	We pick a fully-transcendental point $\theta\in U$ such that the slicing of $S$ with respect to $\pi'$ at $\theta$ is well defined. Then $\Bif(f_\theta)\neq \emptyset$. We fix this $\theta$.
		
		For $\alpha\in B(\C)$, let $O(V_n)_\alpha$ be the specialization of the Galois orbit $O(V_n)$ at $\alpha$, which is a pure dimension one Zariski closed subvariety supported in $(\pi')^{-1}(\alpha)$.   For each $n\geq 1$, let $E_n$ be the finite set given in Proposition \ref{prop: sequence limit}, we have 
		\begin{equation}\label{eqn:4.5}
			\frac{1}{\# E_n}\sum_{\alpha\in E_n}S_{l_n,\alpha }\to S_\theta.
		\end{equation}
		
		Let $R:=(\pi_1)^*T_{\bif}$.  Let $\mu_{\bif,\theta}:=T_{\bif}\wedge [\pi^{-1}(\theta)]$, which is a non-vanishing positive measure on $\pi^{-1}(\theta)$ since $\Bif(f_\theta)\neq \emptyset$.   Let $\Omega\subset \pi^{-1}(\theta)\times \P^N(\C)$ be an open subset, and let $T$ be any positive closed bidimension $(1,1)$ current on $\Omega$.  Since $R$ has continuous local potential, $T\wedge R$ is admissible. For $s\in X(\C)$, recall $R_s$ is  the current of integration on the fiber $\pi_1^{-1}(s)\cong \P^N(\C)$. We need the following
		\begin{lem}\label{lem:positive intersection}
			Assume $R\wedge T>0$. Then there exists $s\in \supp \mu_{\bif,\theta}$ such that  $T\wedge R_s$ is admissible and $T\wedge R_s>0.$  
			
		\end{lem}
		\begin{proof}\label{lem:positive}
			By Proposition \ref{prop: uniform intersection},  $T\wedge R_s$ is admissible for $\mu_{\bif,\theta}$-a.e. $s\in \pi^{-1}(\theta)$, moreover
			\begin{equation*}
				0<T\wedge R=\int T\wedge R_s\;d\mu_{\bif,\theta}(s).
			\end{equation*}
			
			Thus there exists a positive $\mu_{\bif,\theta}$-measure subset such that $T\wedge R_s>0 $  when $s$ is contained in this set. In particular we can choose a point $s$ such that $s\in \supp\mu_{\bif,\theta}$.
		\end{proof}

		By Proposition \ref{prop: equidistribution},  we have
		\begin{points}
			\item $S_\theta\wedge T_f=0$;
			\item $S_\theta\wedge R_s=\mu_s$ for $\mu_{\bif,\theta}$-a.e. $s\in \pi^{-1}(\theta)$, where $\mu_s$ is the maximal entropy measure of $f_s$. 
		\end{points}
		
		\medskip
		
		{\bf We first show  $\genus(\Gamma_n)/\deg(\Gamma_n) \to +\infty$. } Assume by contradiction there is a constant $C>0$ such that  $\genus(\Gamma_n)\leq C\deg(\Gamma_n)$. Since genus and degree are  invariant under Galois conjugation, for every fully transcendental $\alpha$, $$\genus(O(V_n)_\alpha)\leq C\deg(O(V_n)_\alpha).$$
		Since $\vol(O(V_n)_\alpha)$ is a constant with respect to $\alpha$, applying Proposition \ref{prop: deg-vol} componentwise to the irreducible components of $O(V_n)_\alpha$,  there exists  $C'>0$ independent of $\alpha$ and $n$  such that 
		$$\genus(O(V_n)_\alpha)\leq C'\vol(O(V_n)_\alpha).$$

		The sequence of algebraic curves  $\cup_{\alpha\in E_n} O(V_{l_n})_\alpha$ satisfies the genus-volume relation in Definition \ref{defi:genus-volume}.   By (\ref{eqn:4.5}), we can apply  Theorem \ref{thm: strongly approximable} to $S_\theta$ and the test $(1,1)$-current $R$.  There is a uniformly woven current $S_r\leq S_\theta$  on an open subset $\Omega\subset \pi^{-1}(\theta)\times \P^N(\C)$ such that $S_r\wedge R>0$. By Lemma \ref{lem:positive intersection}, there exists $s\in\Bif(f_\theta):=\supp \mu_{\bif,\theta}$ such that
		\begin{points}
			\item $S_r\wedge R_s>0$.
			\item $S_r\wedge T_f=0$;
			\item $S_r\wedge R_t\leq \mu_t$ for every $t\in \pi^{-1}(\theta)$ such that $S_r\wedge R_t$  is admissible.
		\end{points}
		
		By Theorem \ref{thm:wovenbifur}, $s\in \Stab(f_\theta)$, which is a contradiction.
		
		\medskip
		
		{\bf Finally, we show $\gonality(\Gamma_n) \to +\infty$.} Assume by contradiction there is a constant $C>0$ such that  $\gonality(\Gamma_n)\leq C$.  After passing to a subsequence, we may assume that $\gonality(\Gamma_n)=l$ for every $n\geq 1$.  Since  gonality is   invariant under Galois conjugation, for every fully transcendental $\alpha$, every irreducible component of $O(V_n)_\alpha$ has gonality $l$. By (\ref{eqn:4.5}), we can apply  Theorem \ref{thm:sa gonality} to $S_\theta$ and the test $(1,1)$-current $R$.   Let $p:(\pi^{-1}(\theta)\times \P^N)^l\to \pi^{-1}(\theta)\times \P^N$ be the projection to the first coordinate.  By  Theorem \ref{thm:sa gonality},  there exists a positive closed current $S'$ of bidimension $(1,1)$ on $(\pi^{-1}(\theta)\times \P^N(\C))^l$ such that $p_*S'=S_\theta$,    and a uniformly woven current $S'_r\leq S'$ on an  open subset $\Omega\subset (\pi^{-1}(\theta)\times \P^N(\C))^l$ such that $S'_r\wedge p^*R>0$.  By Proposition \ref{prop:irreducible},  after shrinking $\Omega$ slightly, we may assume  $S'_r=\int_{\sV^*_K}[\Gamma_a]\;d\nu(a),$ where $K:=\overline{\Omega}$ is compact.   See Definition \ref{defi:irreduciblespace} for the definition of $\sV^*_K$. For each $a\in \sV^*_K$,  $\Gamma_a$ is an irreducible closed analytic curve in $\Omega$.

		Let $a\in \supp \nu$ such that $[\Gamma_a]\wedge p^*R>0$.    Let $U$ be a small ball such that $[\Gamma_a]\wedge p^*R>0$ in $U$ and $p: \Gamma_b\cap U\to p(U)$ is proper  for every $b$ in a small neighborhood $\sW$ of $a$.   Let $\nu'$ by the restriction of $\nu$ on $\sW$,  let $S'_\sW:=\int_{\sV^*_K}[\Gamma_a]\;d\nu'(a)$, which is a uniformly woven current on $\Omega$. Let $\tilde{S}$ be the restriction of $S'_\sW$ on $U$, which is a uniformly woven current on $U$. Then $p_*\tilde{S}$ is a uniformly woven current on $p(U)$ such that $R\wedge p_*\tilde{S}=p^*R\wedge \tilde{S}>0$. Since $\tilde{S}\leq S'$, we have $p_*\tilde{S}\leq p_*S'=S_\theta$.
		
		By Lemma \ref{lem:positive intersection}, there exists $s\in\Bif(f_\theta):=\supp \mu_{\bif,\theta}$ such that
		\begin{points}
			\item $p_*\tilde{S}\wedge R_s>0$.
			\item $p_*\tilde{S}\wedge T_f=0$;
			\item $p_*\tilde{S}\wedge R_t\leq \mu_t$ for every $t\in \pi^{-1}(\theta)$ such that $p_*\tilde{S}\wedge R_t$  is admissible.
		\end{points}
		
		By Theorem \ref{thm:wovenbifur}, $s\in \Stab(f_\theta)$, which is a contradiction. This completes the proof.
		\endproof

		\medskip
		\proof[Proof of Theorem \ref{thm:stable-dynatomic}]
		Let $K\subset \C$ be a field which is finitely generated over $\Q$ such that $f$ and $\La$ are over $K$. Every dynatomic curve is defined over $\overline K$ and has canonical height zero. Thus (2) implies (1) by Theorem \ref{thm: high-dim Genus}.

		Assume (1).  Fix $m\geq 1$. The locus $\{(t,z):f_t^m(z)=z\}$ is finite over $\La$. Since the bifurcation set is empty,  every $J$-repelling periodic point can be continued along paths in $\La$ and remains $J$-repelling. Hence $\pi_1$ restricted to each irreducible component $\Gamma$ of this locus containing a $J$-repelling periodic point is finite \'{e}tale over $\La$.  	Let $\overline{\La}$ be a smooth projective compactification of $\La$, and let $g:=\genus(\overline{\La})$ and $r:=\#(\overline{\La}\setminus\La)$. Let $\overline\Gamma$ be the smooth projective compactification of $\Gamma$, and let $e:=\deg(\Gamma)$. Let $\tilde{\pi}_1:\overline\Gamma\to\overline{\La}$ be the  finite morphism  induced by $\pi_1$,   then $\tilde{\pi}_1$ can ramify only over $\overline{\La}\setminus\La$. Its ramification divisor $R_\Gamma$ satisfies $\deg R_\Gamma\leq r(e-1)$. By the Riemann--Hurwitz formula, we have
		$$2\genus (\Gamma)-2=e(2g-2)+\deg R_\Gamma\leq e(2g-2)+r(e-1).$$
		In particular,
		$$\frac{\genus(\Gamma)}{\deg(\Gamma)}\leq g+\frac r2.$$

	The $J$-repelling periodic points are Zariski dense in each fiber, so
	the union of these periodic dynatomic curves is Zariski dense in
	$\La\times\P^N$. Since these curves are defined over the countable
	field $\overline K$, the same diagonal argument as in
	\cite[proof of Corollary~3.9, p.~855]{UllmoYafaev2014} gives  a generic
	sequence $(\Gamma_n)_{n\geq 1}$ among them. The above uniform bound
	proves (2), which completes the proof.
		\endproof

		%

		\proof[Proof of Theorem \ref{thm: 1-dim Gonality}]  
		By McMullen's rigidity theorem \cite{McMullen1987}, see also \cite{ji2023homoclinic}, if $f$ is not the  flexible Latt\`es family, then the bifurcation set $\Bif(f)$ is non-empty.  Since a sequence of distinct horizontal curves on $\La\times \P^1$ is automatically a generic sequence, the conclusion 	$\gonality(\Gamma_n) \to +\infty$ and $\genus(\Gamma_n)/\deg(\Gamma_n) \to +\infty$ hold by Theorem \ref{thm: high-dim Genus}. 
		\endproof

		\medskip
		
		\section{Dynamical uniform boundedness: proof of Theorem \ref{thm:1-dim preimages}, \ref{thm:high-dim preimages}, \ref{thm: 1-dim UBC},  and \ref{thm: high-dim UBC}}\label{sec:uniform-boundedness-proof}
		
		%
		%
		We first recall:
		\begin{lem}[Frey]\label{lem:frey}
			Let $C$ be an irreducible curve over a number field $K$, and let $B\geq 1$ be an integer. If the smooth projective model of $C$ contains infinitely many points $x\in C(\overline{K})$ with $[K(x):K]\leq B$, then $\gonality(C)\leq 2B$. \end{lem}
		
		This follows from \cite[Proposition 2]{Frey1994}.
		
		\medskip
		
		\proof[Proof of Theorem \ref{thm:1-dim preimages}]
		If $f$ is the flexible Latt\`es family, the theorem follows from Ingram's uniform boundedness result for the corresponding one-parameter Latt\`es family \cite{MR2870098}. We may therefore assume that $f$ is not the flexible Latt\`es family.
		
		Let $K$ be a number field over which $\La$, $f$, and the marked point $a$ are defined. Let $\Gamma_a\subset \La\times \P^1$ be the horizontal curve which is the graph of $a$. For every $m\geq 0$, set $Y_m:=f^{-m}(\Gamma_a)\subset \La\times \P^1$. Let $\sE_m$ be the union of the irreducible components $\Gamma$ of $Y_m$ which are not contained in $\cup_{j<m}Y_j$. We call such components exact preimage curves of depth $m$.
		
		Fix any polarization $\overline{H}\in \widehat{\Pic}(K/\Z)_{\rm nef,\Q}$. If $\Gamma$ is an irreducible component of $\sE_m$, then $f^m(\Gamma)=\Gamma_a$, and hence $\hat{h}^{\overline{H}}_f(\Gamma)=d^{-m}\hat{h}^{\overline{H}}_f(\Gamma_a)$. In particular, any sequence $\Gamma_m\subset \sE_m$ with $m\to \infty$ is small for $f$. By Theorem \ref{thm: 1-dim Gonality}, the gonalities of any distinct sequence of such exact preimage curves tend to infinity.
		
		Fix $D\geq 1$. Let $B:=D^2$. For $t,z\in \overline{\Q}$ satisfying $[\Q(t):\Q]\leq D$ and $[\Q(z):\Q]\leq D$,  the point $(t,z)\in \La\times \P^1$ has degree at most $B$ over $K$.
		
		Assume by contradiction that the theorem is false for this $D$. For $m\geq 0$, let $\mathcal{S}_m$ be the set of points $(t,z)\in Y_m(\overline K)$ such that
		\begin{points}
			\item $[K(t,z):K]\leq B$;
			\item $m$ is the minimal non-negative integer such that $f_t^m(z)=a_t$.
		\end{points}
		Thus the sets $\mathcal{S}_m$ parametrize iterated preimages of exact depth $m$ and bounded degree. Notice that $\mathcal{S}_m\subset \sE_m(\overline K)$.
		
		For every $M\geq 0$, the number of points $z$ satisfying $f_t^m(z)=a_t$ for some $m<M$ is bounded above by $1+d+\cdots+d^{M-1}$, independently of $t$. Since the desired uniform bound is assumed to fail, it follows that for every $M\geq 0$, the union $\bigcup_{m\geq M}\mathcal{S}_m$ is infinite.
		
		Fix $M\geq 0$. For each $m\geq M$, the map $f^{m-M}$ sends $\mathcal{S}_m$ to $\mathcal{S}_M$: the degree bound is preserved because $f^{m-M}(t,z)$ is defined over $K(t,z)$, and the exactness of the depth is preserved by minimality. We claim that the image of $\bigcup_{m\geq M}\mathcal{S}_m\longrightarrow \mathcal{S}_M$ is infinite. Indeed, fix a point $(t,y)\in \mathcal{S}_M$. The points in its fiber are iterated preimages of $y$ under the single endomorphism $f_t:\P^1\to \P^1$, and all have degree at most $B$ over $K$. By Northcott's theorem for the canonical height of $f_t$, there are only finitely many such points. Since $\bigcup_{m\geq M}\mathcal{S}_m$ is infinite, the image in $\mathcal{S}_M$ must be infinite. 	Thus$\mathcal{S}_M$ is infinite for every $M$. Since $\sE_M$ has only finitely many irreducible components, there is an irreducible component $\Gamma_M$ of $\sE_M$ whose smooth projective model contains infinitely many points of degree at most $B$ over $K$. By Lemma \ref{lem:frey}, we have $\gonality(\Gamma_M)\leq 2B$. Moreover, the components $\Gamma_M$ obtained for distinct $M$ are distinct by the definition of $\sE_M$. For every polarization $\overline{H}$, we have $\hat{h}^{\overline{H}}_f(\Gamma_M)=d^{-M}\hat{h}^{\overline{H}}_f(\Gamma_a)\to 0$. Thus we get a sequence of distinct exact preimage curves with bounded gonality which is small for $f$, which contradicts Theorem \ref{thm: 1-dim Gonality}. This proves the theorem.
		\endproof
		
		\medskip
		
		\proof[Proof of Theorem \ref{thm:high-dim preimages}]
		Let $K$ be a number field over which $\La$, $f$, and $a$ are defined, and let $F:=K(\La)$. We work on the base change $\P^N_{\overline F}$ of the generic fiber. The graph $\Gamma_a$ induces a point $a_F\in \P^N_F(F)\subset \P^N_{\overline F}(\overline F)$. We identify irreducible horizontal curves in $\La\times \P^N$ over $\overline K$ with the corresponding closed points of $\P^N_{\overline F}$.
		
		For every $m\geq 0$, set $Y_m:=f^{-m}(\Gamma_a)$. Let $\sE_m$ be the union of the irreducible components $\Gamma$ of $Y_m$ which are not contained in $\cup_{j<m}Y_j$; we call them exact preimage curves of depth $m$. Equivalently, after the above identification, the irreducible components of $\sE_m$ are the closed points $x\in \P^N_{\overline F}$ which is the exact $m$-th preimage of $a_F$.
		
		If $\Gamma$ is an irreducible component of $\sE_m$, then $f^m(\Gamma)=\Gamma_a$. Hence for every polarization $\overline H\in \widehat{\Pic}(K/\Z)_{\rm nef,\Q}$, $\hat h_f^{\overline H}(\Gamma)=d^{-m}\hat h_f^{\overline H}(\Gamma_a)$. In particular, any sequence of exact preimage curves whose depths tend to infinity is small for $f$.
		
		Fix $D\geq 1$ and let $B:=D^2$. For $m\geq 1$, let $\sB_m$ be the finite set of irreducible components $\Gamma$ of $\sE_m$ such that the smooth projective model of $\Gamma$ contains infinitely many points of degree at most $B$ over $K$. We call the elements of $\sB_m$ bad curves of depth $m$. By Lemma \ref{lem:frey}, every $\Gamma\in \sB_m$ satisfies $\gonality(\Gamma)\leq 2B$. Moreover, if $\Gamma\in \sB_m$ and $m\geq 2$, then $f(\Gamma)\in \sB_{m-1}$. 
		
		We claim that $\sB_m=\emptyset$ for all sufficiently large $m$. Assume this does not hold. Since each $\sB_m$ is finite and $f(\sB_m)\subset \sB_{m-1}$ for $m\geq 2$, we may apply K\"onig's lemma \cite[Lemma 8.1.2]{Reinhard10} to the locally finite rooted tree whose level $m$ is $\sB_m$ and whose edges are induced by the map $f:\sB_m\to \sB_{m-1}$. We get an infinite path, namely a sequence $x_m\in \sB_m, m\geq 1$, such that $f(x_m)=x_{m-1}$ for all $m\geq 2$. Put $x_0:=a_F$.
		
		We endow $\P^N_{\overline F}$ with the constructible topology. Recall that this topology is generated by constructible subsets. Then the space  is compact and Hausdorff, and every Zariski closed subset is open and closed for it \cite[Tag 0901]{stacks-project}. Since $\overline F$ is countable, this topology is second countable, hence compact metrizable. Consequently, the space of probability measures on it is sequentially compact. Consider the Birkhoff averages $\mu_n:=\frac{1}{n}\sum_{i=0}^{n-1}\delta_{x_i}$. After passing to a subsequence, assume $\mu_n$ converges weakly to a probability measure $\mu$. Since $f_*\mu_n-\mu_n=\frac{1}{n}(\delta_{x_0}-\delta_{x_n})$, we have $f_*\mu=\mu$.
		
		Identifying a scheme-theoretic point with the corresponding irreducible closed subset of the generic fiber, we apply \cite[Theorem 1.12 and Lemma 5.2]{Xie2021} to $\P^N_{\overline F}$ with the constructible topology. The measure $\mu$ is a convex combination of probability measures supported on periodic orbits of scheme-theoretic points. Let $\xi$ be a point in the support of $\mu$, and let $r\geq 1$ be its period. Set
		$$Z:=\overline{\{\xi\}}\cup \overline{\{f(\xi)\}}\cup\cdots\cup \overline{\{f^{r-1}(\xi)\}},$$
		where the closures are Zariski closures in $\P^N_{\overline F}$. Then $Z$ is a Zariski closed subset of $\P^N_{\overline F}$ satisfying $f(Z)=Z$.
		
		Since $Z$ is open and closed for the constructible topology and $\mu(Z)>0$, we have $\mu_n(Z)>0$ for all large $n$. Hence $x_i\in Z$ for some $i\geq 1$. Since $f(Z)=Z$ and $f^i(x_i)=a_F$, we get $a_F\in Z$. Therefore the whole forward orbit of $a_F$ is contained in $Z$. By the assumption that the forward orbit of $\Gamma_a$ is Zariski dense in $\La\times \P^N$, the forward orbit of $a_F$ is Zariski dense in $\P^N_{\overline F}$. Thus $Z=\P^N_{\overline F}$. This implies that $\xi$ is the generic point of $\P^N_{\overline F}$.
		
		It follows that every weak limit of the measures $\mu_n$ is the Dirac mass at the generic point of $\P^N_{\overline F}$. Consequently, for every proper subvariety $H\subsetneq \P^N_{\overline F}$ over $\overline F$, we have $\mu_n(H)\to 0$. Using the standard diagonal extraction argument over the countable set of proper subvarieties defined over $\overline F$, we can choose a subsequence of the bad curves $x_{n_j}$ which is generic in $\P^N_{\overline F}$. This gives a generic sequence of exact preimage curves with depths tending to infinity, small for $f$, and with gonality bounded above by $2B$, contradicting Theorem \ref{thm: high-dim Genus}. This proves the claim.
		
		Assume by contradiction that the conclusion fails for this $D$. For $m\geq 0$, let $\mathcal S_m$ be the set of points $(t,z)\in Y_m(\overline K)$ such that
		\begin{points}
			\item $[K(t,z):K]\leq B$;
			\item $m$ is the minimal non-negative integer such that $f_t^m(z)=a_t$.
		\end{points}
		Then $\mathcal S_m$ parametrizes bounded-degree preimages of exact depth $m$.
		
		Since the number of points of depth $<M$ in a fiber is at most $1+d^N+\cdots+d^{N(M-1)}$,  id there is no uniform bound, them for every $M\geq 0$ the union $\bigcup_{m\geq M}\mathcal S_m$ is infinite. Fix $M$ so large that $\sB_M=\emptyset$. For $m\geq M$, the map $f^{m-M}$ sends $\mathcal S_m$ to $\mathcal S_M$. 
		
		The image of $\cup_{m\geq M}\mathcal S_m$ in $\mathcal S_M$ is infinite. Indeed, the fiber over a point $(t,y)\in \mathcal S_M$ consists of iterated preimages of $y$ under the fixed endomorphism $f_t:\P^N\to \P^N$, all of degree at most $B$ over $K$; by Northcott's theorem for the canonical height of $f_t$, this fiber is finite. Hence $\mathcal S_M$ is infinite. Since $\sE_M$ has only finitely many irreducible components, one component of $\sE_M$ contains infinitely many points of degree at most $B$ over $K$, i.e. $\sB_M\neq \emptyset$. This contradicts the choice of $M$ and proves the theorem.
		\endproof
		
		\medskip
		
		\proof[Proof of Theorem \ref{thm: 1-dim UBC}]
		Fix an integer $D\geq 1$. Let $\pi:\End_{d,1}\to \sM_{d,1}$ be the canonical projection. Choose an irreducible algebraic curve $\La\subset \End_{d,1}$ such that $\pi(\La)$ contains a non-empty Zariski open subset of $S$. After normalizing $\La$, the restriction of $\pi$ induces a generically finite morphism $\rho:\La\to S$. Let $q_0:=\deg \rho$. Let $F:\La\times \P^1\to \La\times \P^1$ be the corresponding one-parameter algebraic family.
		
		By Levy's uniform bound for stabilizers \cite[Theorem 3.1]{Levy2011}, there is a constant $M_d$, depending only on $d$, such that for every degree $d$ rational map $g$ on $\P^1$, the group $\Aut(g)$ has order at most $M_d$. Since the set of conjugacies between two conjugate rational maps is a torsor under such an automorphism group, $M_d$ is also an upper bound for the number of conjugacies between two conjugate degree $d$ rational maps on $\P^1$. Set $q:=q_0M_d$. Set $D_S:=qD$. By Theorem \ref{thm: 1-dim Gonality}, and by Nguyen-Saito's theorem in the flexible Latt\`es case \cite{nguyen1996d}, the gonalities of distinct dynatomic curves of $F$ tend to infinity. Hence the set
		$$
		\mathcal D_{S,D}:=\left\{\Gamma\subset \La\times \P^1:
		\begin{array}{l}
			\Gamma\text{ is a dynatomic curve of }F,\\
			\gonality(\Gamma)\leq D_S
		\end{array}
		\right\}
		$$
		is finite.
		
		Now let $f$ be a non-isotrivial endomorphism of $\P^1$ over $k=\C(B)$ which lies over $S$.
		
		Let $B_1$ be the normalization of an irreducible component of $B\times_S\La$ dominating $B$. Then $\deg(B_1\to B)\leq q_0$, and we get a morphism $\beta_1:B_1\to \La$. The two families over $B_1$, namely the pullback of the family associated to $f$ and the pullback $\beta_1^*F$, have the same moduli map. Thus they are  conjugate over the geometric generic point of $B_1$.
		
		Next we choose such a conjugacy over a function field. Consider the relative conjugacy scheme over $B_1$ whose generic fiber is
		$$
		\left\{\phi\in \PGL_2:\phi\circ f_{B_1}\circ \phi^{-1}=\beta_1^*F\right\}.
		$$
		This generic fiber is non-empty and finite, of cardinality at most $M_d$. By taking the normalization of an irreducible component dominating $B_1$, we obtain a finite cover $B'\to B_1$ of degree at most $M_d$. Hence $\deg(B'\to B)\leq q$. After this base change, the family associated to $f$ is conjugate to the pullback of $F$ by the induced morphism $\beta:B'\to \La$. 
		
		Let $x\in \Preper(f,D)$. Let $C$ be the smooth projective curve with function field $k(x)$. Then $\gonality(C)\leq D$. Let $C'$ be an irreducible component of the normalization of $C\times_B B'$ which dominates $C$. Since $\deg(C'\to C)\leq q$, by Fact \ref{fact:gonalityimage}, we have $\gonality(C')\leq q\,\gonality(C)\leq D_S$. The point $x$, after the base change $B'\to B$, determines a horizontal curve in $B'\times \P^1$. Its image under the morphism $\beta\times {\rm id}:B'\times \P^1\to \La\times \P^1$ is contained in a dynatomic curve, denoted by $\Gamma_x$, of $F$. Since $C'$ dominates $\Gamma_x$,  by Fact \ref{fact:gonalityimage}, we have $\gonality(\Gamma_x)\leq \gonality(C')\leq D_S$. Thus $\Gamma_x\in \mathcal D_{S,D}$.
		
		It remains to bound the number of possible points $x$. For each $\Gamma\in \mathcal D_{S,D}$, the pullback $(\beta\times{\rm id})^{-1}(\Gamma)$ has degree at most $\deg(\Gamma/\La)$ over $B'$. Hence it contributes at most $\deg(\Gamma/\La)$ points to the generic fiber over $\overline{k}=\overline{\C(B')}$. Therefore
		$$
		\#\Preper(f,D)\leq \sum_{\Gamma\in \mathcal D_{S,D}}\deg(\Gamma/\La).
		$$
		The right-hand side depends only on $D$ and $S$. This proves that the GUBC holds along $S$ when $N=1$.
		\endproof
		
		\medskip
		
		The proof of Theorem \ref{thm: high-dim UBC} follows the same finite base change and conjugacy reduction as the proof above. The main difference is that, in higher dimension, one has to allow an exceptional subvariety. This exceptional subvariety is produced by taking the Zariski closure of all low-gonality dynatomic curves in the fixed family over $\La$; Theorem \ref{thm: high-dim Genus} ensures that this closure is proper.
		
		\medskip
		
		\proof[Proof of Theorem \ref{thm: high-dim UBC}]
		Fix an integer $D\geq 1$. Let $\pi:\End_{d,N}\to \sM_{d,N}$ be the canonical projection. Choose an irreducible algebraic curve $\La\subset \End_{d,N}$ such that $\pi(\La)$ contains a non-empty Zariski open subset of $S$. After normalizing $\La$, the restriction of $\pi$ induces a generically finite morphism $\rho:\La\to S$. Let $q_0:=\deg\rho$, and let $F:\La\times \P^N\to \La\times \P^N$ be the corresponding one-parameter algebraic family. The family $F$ has non-empty bifurcation set.
		
		By Levy's uniform bound for stabilizers \cite[Theorem 3.1]{Levy2011}, there is a constant $M_{d,N}$, depending only on $d$ and $N$, such that for every degree $d$ endomorphism $g$ of $\P^N$, the group $\Aut(g)$ has order at most $M_{d,N}$. Since the set of conjugacies between two conjugate endomorphisms is a torsor under such an automorphism group, $M_{d,N}$ is also an upper bound for the number of conjugacies between two conjugate degree $d$ endomorphisms of $\P^N$. Set $q:=q_0M_{d,N}, D_S:=qD$.
		
		Let $W_{S,D}$ be the Zariski closure in $\La\times \P^N$ of the union of all dynatomic curves $\Gamma$ of $F$ such that $\gonality(\Gamma)\leq D_S$. We claim that $W_{S,D}\subsetneq \La\times \P^N$. Indeed, if $W_{S,D}=\La\times \P^N$, then by a standard diagonal extraction argument one can choose a generic sequence of distinct dynatomic curves $\Gamma_n$ of $F$ with $\gonality(\Gamma_n)\leq D_S$ for all $n$. Such a sequence is small for $F$, contradicting Theorem \ref{thm: high-dim Genus}. Thus $W_{S,D}$ is a proper subvariety.
		
		Now let $f$ be a non-isotrivial endomorphism of $\P^N$ over $k=\C(B)$ which lies over $S$. Let $B_1$ be the normalization of an irreducible component of $B\times_S\La$ dominating $B$. Then $\deg(B_1\to B)\leq q_0$, and we get a morphism $\beta_1:B_1\to\La$. As in the proof of Theorem \ref{thm: 1-dim UBC}, by considering the relative conjugacy scheme between the pullback of $f$ to $B_1$ and $\beta_1^*F$, after a further finite cover $B'\to B_1$ of degree at most $M_{d,N}$, the family associated to $f$ becomes conjugate to the pullback of $F$ by the induced morphism $\beta:B'\to\La$. In particular, $\deg(B'\to B)\leq q$.
		
		We define $V_{f,D}\subset \P^N_{\overline{k}}$ as follows.  We first pull-back $W_{S,D}$ to $B'\times \P^N$ by the morphism $\beta\times{\rm id}$. Then we take its generic fiber over $B'$.  Since $W_{S,D}$ is a proper subvariety of $\La\times\P^N$ and $\beta$ is dominant, $V_{f,D}$ is a proper subvariety of $\P^N_{\overline{k}}$.
		
		We show that $\Preper(f,D)\subset V_{f,D}(\overline{k})$. Let $x\in\Preper(f,D)$. Let $C$ be the smooth projective curve with function field $k(x)$, then we have  $\gonality(C)\leq D$. Let $C'$ be an irreducible component of the normalization of $C\times_B B'$ which dominates $C$. Since $\deg(C'\to C)\leq q$, by the Fact \ref{fact:gonalityimage}, we have $\gonality(C')\leq q\,\gonality(C)\leq D_S$. After the base change $B'\to B$ and the above conjugacy, the point $x$ determines a horizontal curve in $B'\times \P^N$. Its image in $\La\times \P^N$ is contained in a dynatomic curve $\Gamma_x$ of $F$. Since $C'$ dominates $\Gamma_x$, again by Fact \ref{fact:gonalityimage}, we have  $\gonality(\Gamma_x)\leq D_S$. Hence $\Gamma_x\subset W_{S,D}$, and therefore $x\in V_{f,D}(\overline{k})$. This proves the inclusion. Thus the required inequality holds with $C(D,S)=0$, and the GUBC holds along $S$.
		\endproof
		
		\medskip

		\medskip

		\section{\texorpdfstring{Families with non-empty bifurcation sets}{Non-empty bifurcation sets}: proof of Theorem \ref{thm: example2} and Proposition \ref{prop:very-general-assumptionA}}
		\label{sec:assumptionA}
		

		\proof[Proof of  Theorem \ref{thm: example2}.]
		
		We show that if $f_t$ is a Latt\`es  map and $f$ is not entirely contained in the Latt\`es locus,  then the bifurcation set $\Bif(f)\neq \emptyset$. Let $L(t)$ be the sum of Lyapunov exponents of $\mu_t$, where $\mu_t$ is the maximal entropy measure of $f_t$. Then $\Bif(f)= \emptyset$ if and only if $L$ is a  harmonic function on $\La(\C)$. By  \cite{MR2142250},  $L(t)\geq \frac{N\log d}{2}$, where the equality holds if and only if $f_t$ is a Latt\`es map.  Since $f_{t}$ is a Latt\`es map and $f$ is not contained in the Latt\`es locus, the function $L$ is non-constant and achieves its  minimum at $t$. By the minimum principle of harmonic functions, $L$ is not harmonic on $\La(\C)$.  This implies  $\Bif(f)\neq \emptyset$.  This completes the proof.
		\endproof
		
		\proof[Proof of Proposition \ref{prop:very-general-assumptionA}.]
		Let $C\subset \sM_{d,N}$ be the image of the moduli map of the family. Let $Z\subset \sM_{d,N}$ be the proper Zariski closed subset given by Proposition \ref{prop:general-not-stable}. Since $C$ passes through a general point of $\sM_{d,N}$, we have $C\not\subset Z$. Hence Proposition \ref{prop:general-not-stable} implies that the bifurcation locus of the family is non-empty. This completes the proof.
		
		\endproof

		\medskip
\newcommand{\etalchar}[1]{$^{#1}$}

	\end{document}